\documentclass[11pt,letterpaper]{amsart}
\usepackage{amsmath,amssymb,hyperref}

\usepackage{tikz}
\usetikzlibrary{cd,decorations.markings,patterns,perspective}

\usepackage[bb=ams,scr=euler]{mathalpha}
\usepackage{enumitem}
\setenumerate{label=(\arabic*)}

\usepackage{thmtools}
\declaretheorem[name=Theorem,numberwithin=section]{theorem}
\declaretheorem[name=Corollary,sibling=theorem]{corollary}
\declaretheorem[name=Lemma,sibling=theorem]{lemma}
\declaretheorem[name=Claim,sibling=theorem]{claim}
\declaretheorem[name=Proposition,sibling=theorem]{proposition}
\declaretheorem[name=Definition,sibling=theorem]{definition}

\declaretheoremstyle[bodyfont=\normalfont]{remarkstyle}
\declaretheorem[style=remarkstyle,name=Remark,sibling=theorem]{remark}

\makeatletter
\newcommand{\abs}[1]{\left|#1\right|}
\newcommand{\bd}{\partial}
\newcommand{\C}{\mathbb{C}}

\renewcommand{\d}{\mathrm{d}}
\newcommand{\id}{\mathit{id}}
\newcommand{\pr}{\mathrm{pr}}

\newcommand{\R}{\mathbb{R}}
\newcommand{\set}[1]{\left\{#1\right\}}
\newcommand{\Z}{\mathbb{Z}}
\newcommand{\coker}{\mathop{\mathrm{coker}}}

\renewcommand\section{\@startsection{section}{1}{0pt}{-3.5ex \@plus -1ex \@minus -.2ex}{2.3ex \@plus.2ex}{\centering\bfseries}}
\renewcommand{\subsection}{\@startsection{subsection}{2}\z@{.5\linespacing\@plus.7\linespacing}{-.5em}{\normalfont\bfseries}}
\renewcommand{\subsubsection}{\@startsection{subsubsection}{3}%
  \z@{.5\linespacing\@plus.7\linespacing}{-.5em}%
  {\normalfont\bfseries}}
\makeatother

\author{Dylan Cant}
\email{dylan@dylancant.ca}
\address{Départment de mathématiques et de statistique, Université de Montreal, Pavillon André-Aisenstadt, 2920 Chemin de la Tour, Montreal, Quebec, H3T 1J4, Canada}

\author{Eric Kilgore}
\email{eric.kilgore@usc.edu}
\address{Department of Mathematics, University of Southern California, 3620 S. Vermont Ave, KAP 104, Los Angeles, CA 90089}

\author{Jun Zhang}
\email{jzhang4518@ustc.edu.cn}
\address{The Institute of Geometry and Physics, University of Science and Technology of China, 96 Jinzhai Road, Hefei, Anhui, 230026, China}

\date{\today}

\begin{document}
\title[Equivariant contact Floer
cohomology]{Equivariant Floer cohomology for
  contactomorphisms of quotient spaces}
\begin{abstract}
  This paper establishes the orderability of
  contact manifolds which are quotients of
  fillable contact manifolds under finite group
  actions compatible with the filling, the
  prototypical example being $\mathbb{R}P^{2n-1}$
  as the quotient of $S^{2n-1}$. Our approach
  employs an equivariant formulation of the
  so-called contact Floer cohomology theory. This
  leads us to develop an analogue of Givental's
  nonlinear Maslov index using the
  $\mathbf{k}[[x]]$-module structure on an
  equivariant version of contact Floer
  cohomology. A key idea is that mapping cones of
  continuation maps detect crossings with the
  discriminant (recall that Givental's index is a
  continuous integer valued function on the
  complement of the discriminant). To properly
  handle the inherent non-canonicity in defining
  such mapping cones, we lift the structure of
  contact Floer cohomology to chain level by
  defining it as an $\infty$-functor on a suitable
  $\infty$-categorification of the
  Eliashberg--Polterovich orderability relation on
  the universal cover of the contactomorphism
  group.
\end{abstract}

\maketitle

\section{Introduction}
\label{sec:introduction}

Let $G$ denote a finite group; our main results
concern rigidity phenomena for contact manifolds
$Y$ which admit the following additional
structure:
\begin{definition}\label{definition:G-filling}
  A \emph{$G$-filling} of a compact contact
  manifold $Y$ is an open\footnote{Here
    \emph{open} means that each connected
    component of $W$ is non-compact.}
  convex-at-infinity symplectic manifold
  $(W, \omega)$ with a contact-at-infinity
  symplectic $G$-action\footnote{Each $g\in G$
    acts by a symplectomorphism of $W$ which
    commutes with the Liouville flow in the convex
    end of $W$; see
    \cite[\S2.1]{alizadeh-atallah-cant-math-z-2025}
    for our conventions on convex ends.}  whose
  ideal restriction acts freely on the ideal
  contact boundary $\bd W$, and a
  contactomorphism\footnote{Note that $Y$ inherits
    a preferred coorientation from any
    $G$-filling.} $Y\to \bd W/G$. We say that $Y$
  admits an \emph{aspherical $G$-filling} if the
  filling can be chosen so that $\omega$ vanishes
  on spherical homology classes.
\end{definition}

In a sentence: we prove $Y$ is orderable in the
sense of \cite{eliashberg-polterovich-gafa-2000}
if it admits an aspherical $G$-filling $W$ and the
$G$-action on $W$ is not free.

\subsection{Background}
\label{sec:background}

To start, let us recall some basic concepts in
contact geometry. Let $(Y^{2n-1}, \xi)$ be a
contact manifold. The group of diffeomorphisms
preserving its contact structure is denoted by
$\mathrm{Cont}(Y)$; elements of this group are
called \emph{contactomorphisms}. A \emph{contact
  isotopy}\footnote{To lighten the notation, we
  use the symbol $\varphi_{t}$ rather than, say
  $\set{\varphi_{t}:t\in [0,1]}$, to denote
  contact isotopies. In the same way, a path of
  contact isotopies is denoted
  $\varphi_{s,t}$. This notational convention is
  the same as the one used in, e.g.,
  \cite{cant-sh-barcode,
    cant-hedicke-kilgore-arXiv-2023,
    cant-arXiv-2024}.}  is a path
$\varphi_t\in \mathrm{Cont}(Y)$, defined for
$t\in [0,1]$, based at $\varphi_{0}=\id$, and we
denote by $\mathrm{CI}(Y)$ the group of all such
contact isotopies. The quotient of
$\mathrm{CI}(Y)$ by the subgroup of contractible
loops is the universal cover of
$\mathrm{Cont}(Y)$.

Attempts to study $\mathrm{Cont}(Y)$ and its
universal cover have long been a major theme of
contact geometry. In
\cite{eliashberg-polterovich-gafa-2000},
Eliashberg and Polterovich propose to study these
groups via the following relation: let us say that
$\varphi_{0,t}\le \varphi_{1,t}$ if there is a
path in $\mathrm{CI}(Y)$ from $\varphi_{0,t}$ to
$\varphi_{1,t}$, say $\varphi_{s,t}$, such that:
\begin{equation}\label{eq:non-neg-path}
  s\mapsto \varphi_{s,1}(y)\text{ is never negatively transverse to $\xi$ for each $y\in Y$.}
\end{equation}
Note that, in order for this to make sense, we
require that $\xi$ is cooriented.

If $\varphi_{0,t}$ and $\varphi_{1,t}$ differ by a
contractible loop, then they are mutually related
by $\le$, and so the relation descends to the
universal cover of $\mathrm{Cont}(Y)$. If the
induced relation on the universal cover of
$\mathrm{Cont}(Y)$ is a partial order, then $Y$ is
called \emph{orderable}; otherwise, $Y$ is called
\emph{non-orderable}.

It was shown in
\cite{eliashberg-polterovich-gafa-2000} that
certain contact manifolds are orderable:
\begin{itemize}
\item real projective spaces $\mathbb{R}P^{2n-1}$
  (with standard contact structure),
\item prequantization bundles over certain
  symplectic manifolds,
\item spherical cotangent bundles $ST^{*}M$, for
  certain $M$.\footnote{For the state of the art
    result that all spherical cotangent bundles
    are orderable we refer the reader to
    \cite[Theorem~1.1.]{chernov-nemirovski-JSG-2016}.}
\end{itemize}
Since \cite{eliashberg-polterovich-gafa-2000},
much research has been done concerning
orderability of contact manifolds. A particularly
noteworthy milestone is
\cite{eliashberg-kim-polterovich-GT-2006} which
establishes the non-orderability of the ideal
boundary of stabilizations $W\times \C^{d}$ of
Liouville manifolds $W$ when $d\ge 2$; in
particular, $S^{3},S^{5},\dots$, with their
standard contact structure, are all
non-orderable.\footnote{See also
  \cite{hedicke-shelukhin-arXiv-2024} for the case
  of $W\times \C$ if $W$ is $T^{*}S^{1}$ or is a
  Weinstein manifold with dimension at least $4$.}

The tool used to prove orderability of
$\mathbb{R}P^{2n-1}$ is Givental's ``non-linear
generalization of the Maslov index'' from
\cite{givental-nl-maslov-advsov,givental-nl-maslov-LMS},
defined on the universal cover of the
contactomorphism group of projective spaces using
equivariant cohomology of generating
functions. Let us recall here that Givental's
index is a map
$\mu:\mathrm{CI}(\mathbb{R}P^{2n-1})\to \Z$
satisfying:

\begin{enumerate}[label=(G\arabic*)]
\item\label{item:G-monotonicity}
  \emph{(monotonicity)}
  $\mu(\varphi_{0,t})\le \mu(\varphi_{1,t})$ if
  $\varphi_{0,t}\le \varphi_{1,t}$;
\item\label{item:G-continuity} \emph{(continuity
    from above)}
  $\mu(\varphi_{t})=\lim_{s\to
    0+}\mu(R_{st}\varphi_{t})$ for any positive
  contact isotopy $R_{s}$;\footnote{A positive
    contact isotopy $R_{s}$ is one for which the
    curves $s\mapsto R_{s}(p)$ are positively
    transverse to the contact distribution}
\item\label{item:G-normalization}
  \emph{(normalization)} $\mu(\id)=0$;
\item\label{item:G-non-triviality}
  \emph{(non-triviality)} $\mu$ attains
  arbitrarily large values;
\item\label{item:G-discriminant}
  \emph{(discriminant)} $\mu$ is constant on the
  path-connected components of the complement of
  the discriminant.
\end{enumerate}
Recall the \emph{discriminant} in $\mathrm{CI}(Y)$
is the subset of those $\varphi_{t}$ such that
$\varphi_{1}$ has a fixed point $x$ satisfying
$\varphi_{1}^{*}\alpha_{x}=\alpha_{x}$, for
some/any contact form $\alpha$ near $x$.

One of our main results (Theorem
\ref{theorem:main-G}) involves the construction of
invariants satisfying \ref{item:G-monotonicity}
through \ref{item:G-discriminant} for contact
manifolds $Y$ other than real projective spaces,
and using Floer theory (settling Givental's claim
in \cite[pp.\,43]{givental-nl-maslov-LMS} that a
Floer theoretic approach to his index should be
possible).

\begin{remark}
  Givental's index famously generalizes the Maslov
  index: under the inclusion
  $\mathrm{Sp}(2n)\to
  \mathrm{Cont}(\mathbb{R}P^{2n-1})$ his index
  pulls-back to the Maslov index; in particular,
  this yields \ref{item:G-normalization} and
  \ref{item:G-non-triviality}. The index we define
  in Theorem \ref{theorem:main-G} also has this
  property, see Proposition
  \ref{prop:agrees-with-CZ}.
\end{remark}

Property \ref{item:G-discriminant} suggests that
``sufficiently positive paths'' must intersect the
discriminant. This is related to the dynamics of
Reeb vector fields. Recall that a \emph{Reeb
  vector field} is any vector field $R$ whose flow
preserves $\xi$ and which is positively
transverse\footnote{It can be shown that each Reeb
  vector field satisfies $\alpha(R)=1$ and
  $\d\alpha(R,-)=0$ for a unique contact form
  $\alpha$; indeed, many authors define Reeb
  vector fields in this way.} to the contact
distribution $\xi$. Let us denote the time $s$
flow of $R$ by the symbol $R_{s}$. It then follows
that:
\begin{equation*}
  \psi_{s,t}=\varphi_{t}^{-1}\circ R_{st}\text{ (defined for $s\in \R$)}
\end{equation*}
satisfies
$\lim_{s\to\pm\infty}\mu(\psi_{s,t})=\pm\infty$. In
particular, since Givental's index is locally
constant on the complement of the discriminant,
$\psi_{s,t}$ must intersect the discriminant for
infinitely many values of $s$.

In Sandon's work \cite{sandon-ann-inst-four-2011,
  sandon-equivariant-JSG-11,
  sandon-int-j-math-2012,
  sandon-geom-dedicata-2013}, the values of $s$
when $\psi_{s,t}$ intersects the discriminant were
discovered to be the critical values of suitable
generating functions and were called \emph{lengths
  of translated points}.

\begin{definition}\label{definition:spectrum}
  The \emph{spectrum} $\mathrm{Spec}_{R}(\varphi_{t})$ of $\varphi_{t}$ relative to
  a Reeb flow $R$, is the set of numbers $s\in \R$
  such that $\varphi_{t}^{-1}\circ R_{st}$ lies in
  the discriminant.
\end{definition}

The work of \cite{albers-merry-JFPTA-2013,
  albers-merry-JSG-2018} develops an elliptic
Morse theory for a pair $(R,\varphi_{t})$ whose
associated spectrum of critical values is
$\mathrm{Spec}_{R}(\varphi_{t})$; their theory is
based on the Rabinowitz--Floer homology (RFH)
theory of \cite{cieliebak-frauenfelder-PJM-2009}.

The approaches of Sandon and Albers--Merry led to
the introduction of \emph{spectral invariants},
namely, functions $c_{R}:\mathrm{CI}(Y)\to \R$
satisfying:
\begin{enumerate}[label=(R\arabic*)]
\item\label{item:R-spectrality}
  \emph{(spectrality)}
  $c_{R}(\varphi_{t})\in
  \mathrm{Spec}_{R}(\varphi_{t})$.
\end{enumerate}
The existence of such a spectral invariant is not
guaranteed, as it implies the spectrum is
non-empty, a fact which is not always true; see
\cite{cant-JSG-2024}.

More recently \cite{cant-sh-barcode,
  djordjevic_uljarevic_zhang, cant-arXiv-2024,
  djordjevic-uljarevic-zhang-arXiv-2025}
constructed such spectral invariants for certain
contact manifolds $Y$ using Hamiltonian Floer
(co)homology in a symplectic filling; one requires
that $Y$ admits a filling $W$ satisfying certain
hypotheses; we refer the readers to
\cite{cant-arXiv-2024,
  djordjevic-uljarevic-zhang-arXiv-2025} for more
details. This led to the construction of spectral
invariants satisfying \ref{item:R-spectrality} and
four additional properties:
\begin{enumerate}[label=(R\arabic*),resume]
\item\label{item:R-monotonicity}
  \emph{(monotonicity)}
  $c_{R}(\varphi_{0,t})\le c_{R}(\varphi_{1,t})$
  if $\varphi_{0,t}\le \varphi_{1,t}$,
\item\label{item:R-continuity-from-above}
  \emph{(continuity from above)}
  $c_{R}(\varphi_{t})=\lim_{s\to 0+}
  c_{R}(R_{st}\varphi_{t})$, where $R_{s}$ is any
  Reeb flow,
\item\label{item:R-normalization}
  \emph{(normalization)} $c_{R}(\id)=0$,
\item\label{item:R-sub-additivity}
  \emph{(sub-additivity)}
  $c_{R}(\varphi_{0,t}\varphi_{1,t})\le
  c_{R}(\varphi_{0,t})+c_{R}(\varphi_{1,t})$.
\end{enumerate}
One of our main results (Theorem
\ref{theorem:main-R}) guarantees the existence of
such measurements under the assumption that $Y$
admits a $G$-filling $W$ with at least one point
with non-trivial stabilizer (for the $G$-action).

Such a spectral invariant is sufficient to
guarantee the orderability of $Y$:
\begin{proposition}\label{proposition:orderable}
  The existence of $c_{R}$ satisfying
  \ref{item:R-spectrality} through
  \ref{item:R-sub-additivity} implies $Y$ is
  orderable.
\end{proposition}
This statement is known in the literature (see,
e.g., \cite{allais-arlove-CCM-2026}). We quickly review the
argument; the first step is to prove:
\begin{proposition}\label{proposition:lipschitz}
  Any spectral invariant $c_{R}$ satisfying axioms
  \ref{item:R-spectrality} through
  \ref{item:R-sub-additivity} also satisfies:
  \begin{equation*}
    \abs{c_{R}(\varphi_{0,t})-c_{R}(\varphi_{1,t})}\le \mathrm{dist}_{\alpha}(\varphi_{0,t},\varphi_{1,t}),
  \end{equation*}
  where $\mathrm{dist}_{\alpha}$ is
  \emph{Shelukhin's Hofer distance}
  \cite{shelukhin-JSG-2017} associated to the
  contact form $\alpha$ whose Reeb flow is $R$.
\end{proposition}
\begin{proof}[Proof of Proposition
  \ref{proposition:lipschitz}]
  If $\varphi_{0,t}\le R_{st}\varphi_{1,t}$ then:
  \begin{equation*}
    c_{R}(\varphi_{0,t})\le c_{R}(R_{st}\varphi_{1,t})\le c_{R}(\varphi_{1,t})+c_{R}(R_{st})\le c_{R}(\varphi_{1,t})+s,
  \end{equation*}
  where the last step follows from:
  \begin{equation*}
    c_{R}(R_{st})\le \lim_{k\to\infty}c_{R}(R_{st/k})k=(s/k)k=s;
  \end{equation*}
  here we appeal to the \emph{spectrum gap for
    Reeb flows}\footnote{$R_{s}$ has no
    discriminant points if
    $s\in (-\hbar,0)\cup (0,\hbar)$ for some
    $\hbar>0$.} to conclude that $s/k$ is the only
  number in the spectrum of $R_{st/k}$ which lies
  in some small neighborhood of $0$. Thus:
  \begin{equation*}
    \begin{aligned}
      \abs{c_{R}(\varphi_{0,t})-c_{R}(\varphi_{1,t})}
      &\le \min\set{s:\varphi_{0,t}\le R_{st}\varphi_{1,t}\text{ and }\varphi_{1,t}\le R_{st}\varphi_{0,t}}\\
      &\le \mathrm{dist}_{\alpha}(\varphi_{0,t},\varphi_{1,t}),
    \end{aligned}
  \end{equation*}
  where the latter inequality is well-known; see
  \cite{allais-arlove-CCM-2026,nakamura-arXiv-2023}.
\end{proof}
\begin{proof}[Proof of Proposition
  \ref{proposition:orderable}]
  It follows easily from continuity and
  spectrality that $c_{R}(R_{st})=s$, and hence
  $\mathrm{dist}_{\alpha}(1,R_{st})\ge s$. This
  implies, by work of \cite{hedicke-JTA-2024} that
  $Y$ is orderable.
\end{proof}

\begin{remark}\label{remark:spectral-geometry}
  The emerging picture from the spectral
  invariants $c_{R}$ is of a contact geometry
  analogue to the well-known theory of spectral
  invariants for Hamiltonian diffeomorphisms of
  symplectic manifolds initiated by
  \cite{viterbo92-GF,schwarz_spectral_invariants,oh-2005-duke},
  except there is one such structure \emph{for
    each Reeb flow} $R$ (as with the Hofer-type
  norms in \cite{shelukhin-JSG-2017}). However,
  certain closed contact manifolds, e.g., the
  non-orderable ones, do not admit such structures
  (this contrasts with the symplectic setting
  where all closed symplectic manifolds are
  supposed to admit spectral invariants).
\end{remark}

\begin{remark}[Hypertightness]\label{remark:hypertight-variation}
  It does not seem to be too hard to adapt the
  methods of this paper and
  \cite{albers-fuchs-merry-compositio-2015,
    albers-fuchs-merry-selecta-2017,
    oh-legendrian-entanglement, cant-arXiv-2024,
    djordjevic-uljarevic-zhang-arXiv-2025} to show
  that any counterexample $Y$ to the Weinstein
  conjecture admits invariants of type
  \ref{item:R-spectrality} through
  \ref{item:R-sub-additivity}. The essential idea
  is to exploit the fact that any counterexample
  to the Weinstein conjecture is
  \emph{hypertight}, that is, has no contractible
  Reeb orbits for some Reeb flow $R$. For such
  manifolds, one can ``do Floer theory'' directly
  in the symplectization $SY$, by appealing to the
  compactness results of \cite{BEHWZ}.
\end{remark}

\begin{remark}[Reeb orbits]
  Let us comment on the relationship of the
  invariants $\mu$ satisfying
  \ref{item:G-monotonicity} through
  \ref{item:G-discriminant} and the existence of
  Reeb orbits: any contact manifold $Y$ admitting
  such a measurement $\mu$ must admit closed Reeb
  orbits. Indeed, since $\mu(R_{st})$ can only
  change when $R_{st}$ crosses the discriminant,
  and $\mu$ attains arbitrarily large values, so
  $R_{st}$ must intersect the discriminant for
  infinitely many values of $s$. This observation
  and the discussion in Remark
  \ref{remark:hypertight-variation} shows the
  invariants of type $\mu$ are, in a certain
  sense, rarer than the invariants of type
  $c_{R}$.\footnote{In fact, in certain cases, it
    will follow from our construction Theorem
    \ref{theorem:main-G} that the values when
    $\mu(R_{st})$ changes are periods of
    \emph{contractible} Reeb orbits (we need to
    assume the map
    $\pi_{1}(\partial W)\to \pi_{1}(W)$ is
    injective to make this deduction).}
\end{remark}

\subsection{Statement of main results and
  examples}
\label{sec:statement-main-result}

Recall Definition \ref{definition:G-filling} on
aspherical $G$-fillings.

\begin{definition}\label{definition:W-contractible}
  Let $\varphi_{t}\in \mathrm{CI}(Y)$, and let $W$
  be a $G$-filling of $Y$. Define a
  \emph{$W$-contractible discriminant point} of
  $\varphi_{t}$ to be a discriminant point $x$ of $\varphi_{1}$
  such that $t\mapsto \varphi_{t}(x)$ lifts from
  $Y$ to $\bd W$ and any lift is a contractible
  loop in $W$. The \emph{$W$-contractible
    discriminant} is the subset consisting of
  those $\varphi_{t}\in \mathrm{CI}(Y)$ for which
  $\varphi_{t}$ has a $W$-contractible
  discriminant point.
\end{definition}

\begin{theorem}\label{theorem:main-R}
  Let $G$ be a finite group and suppose that $Y$
  is a contact manifold admitting an aspherical
  $G$-filling $W$. If the $G$-action on $W$
  possesses at least one point with non-trivial
  stabilizer, then there exists a contact spectral
  invariant $c_{R}$ satisfying
  \ref{item:R-spectrality} through
  \ref{item:R-sub-additivity} for any Reeb flow on
  $Y$.

  Moreover, $c_{R}(\varphi_{t})$ is the value of
  $s$ for which $R_{st}\circ \varphi_{t}^{-1}$ has
  a $W$-contractible discriminant point.
\end{theorem}

Thus we prove that all such manifolds $Y$ admit
spectral invariants for each Reeb flow, in the
sense of Remark \ref{remark:spectral-geometry},
and are, in particular, orderable.

If we assume moreover that $\mathrm{SH}(W)=0$ then
our construction produces a measurement $\mu$
generalizing the Maslov class:
\begin{theorem}\label{theorem:main-G}
  Assume the hypotheses of Theorem
  \ref{theorem:main-R}. If, in addition, the
  symplectic cohomology\footnote{The symplectic
    cohomology $\mathrm{SH}(W;\Z/p\Z)$ is a
    well-studied invariant for convex-at-infinity
    manifolds; see, e.g., \cite{viterbo-GAFA-1999,
      seidel-IP-2008, ritter-jtopol-2013}. To be
    concrete, we follow the conventions of
    \cite{cant-hedicke-kilgore-arXiv-2023,
      cant-arXiv-2024}, and work only in the free
    homotopy class of contractible loops. The
    orientation scheme needed to work with primes
    $p>2$ will be explained in
    \S\ref{sec:asymp-ops-orient-lines}}
  $\mathrm{SH}(W;\Z/p\Z)$ vanishes for some prime
  number $p$ dividing the cardinality of a
  stabilizer of the $G$-action on $W$, then there
  is a measurement $\mu:\mathrm{CI}(Y)\to \Z$
  satisfying properties \ref{item:G-monotonicity}
  through \ref{item:G-discriminant}.

  Moreover, $\mu$ can be chosen to satisfy the
  strengthened version of
  \ref{item:G-discriminant}: $\mu$ is constant on
  the path connected components of the complement
  of the $W$-contractible discriminant.
\end{theorem}

The obvious example of $Y$ satisfying the
hypotheses of Theorem \ref{theorem:main-R} and
\ref{theorem:main-G} is $\mathbb{R}P^{2n-1}$, with
its standard contact structure, since it admits
the $\Z/2\Z$-filling $\C^{n}$ with a fixed point,
and $\C^{n}$ is known to have vanishing symplectic
cohomology (for any choice of coefficients
$\mathbf{k}$). Similarly, our theorem applies to
the standard lens spaces $L_{p}$, $p>1$, obtained
by quotienting $S^{2n+1}$ by the $\Z/p\Z$ action
generated by $z\mapsto e^{2\pi i /p}z$, and we
recover some of the results of
\cite{granja-karshon-pabiniak-sandon,allais-arlove-sandon-arXiv-2024}.

The fact that our measurement generalizes the
Maslov class can be made quite precise:
\begin{proposition}\label{prop:agrees-with-CZ}
  Suppose $W=\C^{n}$, with the antipodal
  $\Z/2\Z$-action. If $\varphi_{t}$ is an isotopy
  valued in the subgroup $\mathrm{Sp}(2n)$ whose
  time-1 map does not have $1$ as an eigenvalue,
  then
  $\mu(\varphi_{t})=\mathrm{CZ}(\varphi_{t})-n,$
  where the Conley-Zehnder index associated to a
  non-degenerate linear symplectic flow is as
  defined in, e.g., \cite[Chapter
  8]{polterovich-zhang-et-al-revision}.\footnote{For
    $\mu$ to be normalized to zero, we require
    $\mathrm{CZ}(R_{\epsilon t})=n$ for
    $\epsilon\in (0,1)$, where $R$ is the standard
    Reeb flow on the sphere $S^{2n-1}$, extended
    radially to $\C^{n}$. The quadratic
    Hamiltonian generating this linear isotopy is
    $\epsilon \pi \abs{z}^{2}$.}
\end{proposition}

The proof of this proposition is given in
\S\ref{sec:agre-with-conl}. In the following
subsections we give concrete examples of manifolds
to which both Theorem \ref{theorem:main-R} and
\ref{theorem:main-G} can be applied.

\subsubsection{Quotients of stabilized cotangent
  bundles}
\label{sec:stabilizations}

Let $W=T^{*}M\times \C^{d}$ and suppose that
$\Z/p\Z$ acts smoothly on $M$ with isolated fixed
points. There is an induced action on $T^{*}M$ by
canonical transformations --- importantly, this
action preserves the Liouville form
$\lambda_{T^{*}M}$. Then $\Z/p\Z$ acts
diagonally on $W$, where the action on $\C^{d}$ is
by multiplication by $e^{2\pi i/p}$. This action
preserves $\lambda_{T^{*}M}+\lambda_{\C^{d}}$, and
thus has an ideal restriction; this is a free
action on $\bd W$ because the action on $M$ has
isolated fixed points. Let $Y=\bd W/(\Z/p\Z)$. In
this setting, Theorem \ref{theorem:main-R} applies
if the action on $M$ has at least one fixed
point. If, in addition, $d>0$, then Theorem
\ref{theorem:main-G} applies. The case when
$M=\mathrm{pt}$ recovers the case of
$\mathbb{R}P^{2d-1}$ and the lens spaces discussed
above.

Interestingly, \cite[Theorem
1.16]{eliashberg-kim-polterovich-GT-2006} and
\cite[Corollary 1.3]{hedicke-shelukhin-arXiv-2024}
show the ideal boundary $\bd(T^*M \times \C^d)$ is
non-orderable (for $d\ge 2$ by
\cite{eliashberg-kim-polterovich-GT-2006}, or
$d\ge 1$ if $\dim M\ge 1$ by
\cite{hedicke-shelukhin-arXiv-2024}), and hence
does not admit measurements satisfying
\ref{item:R-spectrality} through
\ref{item:R-sub-additivity} or
\ref{item:G-monotonicity} through
\ref{item:G-discriminant}.

\subsubsection{Non-primitive prequantization
  spaces}
\label{sec:non-prim-preq-spac}

Let $(B,\omega)$ be a symplectic manifold, and
suppose there is a unitary line bundle
$\pi:W\to B$ satisfying\footnote{Recall that
  $c_{1}$ is the deRham cohomology class of
  represented by any differential two-form
  $\mathbf{c}$ solving
  $\d\alpha=\pi^{*}\mathbf{c}$, where $\alpha$ is
  a \emph{global angular form}, i.e., $\alpha$ is
  a one-form on $W$ which restricts to
  $(2\pi)^{-1}(x\d y-y\d x)/(x^{2}+y^{2})$ in each
  fiber.} $c_{1}=-[\omega]$. As explained in
\cite[pp.\,656]{mcduff-contact-boundaries}, there
is a natural way to equip $W$ with the structure
of a symplectic manifold which is
convex-at-infinity (the construction uses the
unitary structure on $W$). The fiberwise action of
$e^{2\pi i /p}$ is a symplectomorphism with fixed
point set equal to the zero set $B$. If $B$ is
symplectically aspherical, it is known that the
non-equivariant regular symplectic cohomology
$\mathrm{SH}(W;\Z/p\Z)$ vanishes (this is shown
in, e.g., \cite{cant-hedicke-kilgore-arXiv-2023}
in the case $p=2$). Therefore Theorem
\ref{theorem:main-R} and \ref{theorem:main-G}
apply to $Y=(\bd W)/(\Z/p\Z)$, when $B$ is
aspherical.

The construction shows that $\bd W$ is naturally
identified with the unit circle bundle in $W$, and
that $Y$ is simply the ideal boundary of the
negative line bundle $W^{\otimes p}$ with base
space $(B,p\omega)$. The orderability of such $Y$
was established already for a special class of
aspherical $B$ in
\cite[\S1.3.C]{eliashberg-polterovich-gafa-2000}.

Note that in order for our methods to apply, we
require that $Y=\partial W^{\otimes p}$ for $p>1$,
i.e., $Y$ is ``non-primitive,'' with respect to
tensor powers. Some assumption is clearly
necessary, since $Y=S^{2n-1}$ (which is a
primitive prequantization space) is non-orderable
\cite{eliashberg-kim-polterovich-GT-2006,hedicke-shelukhin-arXiv-2024}.

As a comparison, using the generating function
approach of
\cite{givental-nl-maslov-advsov,givental-nl-maslov-LMS},
the work of
\cite{borman-zapolsky-GT-2015,zapolsky-dedicata-2020,granja-karshon-pabiniak-sandon}
establishes the existence of measurements
satisfying \ref{item:G-monotonicity} through
\ref{item:G-discriminant} for certain special
choices of $Y$ --- in these constructions the base
$B$ is a toric manifold; see also
\cite{sandon-geom-dedicata-2013,tervil-JSG-2021}
which uses generating functions to show
$\mathrm{Spec}_{R}(\varphi_{t})\ne \emptyset$ for
all $\varphi_{t}$ and a specific Reeb vector field
$R$ when the base $B$ is a closed monotone
symplectic toric manifold (provided that $Y$ is
not the standard contact sphere).

Floer theoretic approaches also have been
developed to study the orderability (and
construction of spectral invariants) for
prequantization spaces; see
\cite{albers-shelukhin-zapolsky-2016}, and the RFH
approach\footnote{We emphasize that here the
  spectral invariants we construct are monotone
  \ref{item:R-monotonicity} and sub-additive
  \ref{item:R-sub-additivity}, and these
  properties are not established in the existing
  RFH literature.} of
\cite{albers_frauenfelder_nonlinear_maslov,albers-merry-JSG-2018,albers-kang-adv-math-2023,bae-kang-kim-math-ann-2024}.\footnote{The
  paper
  \cite{albers_frauenfelder_nonlinear_maslov}
  generalizes Givental's work
  \cite{givental-nl-maslov-advsov,givental-nl-maslov-LMS}
  in a different direction from our paper. Their
  construction does not produce integer valued
  measurements for contact isotopies
  $\varphi_{t}$, but rather studies the
  ``asymptotic growth rates'' of Rabinowitz-Floer
  homology groups when one iterates positive
  paths.}

\subsubsection{Milnor fibers and their stabilizations}
\label{sec:milnor-fibers}

We can also apply our results to the links of
isolated singularities of certain weighted
homogeneous polynomials.\footnote{We thank
  Klaus Niederkr\"uger for bringing these
  examples of $G$-fillings to our attention.}

Following classical singularity theory
\cite{milnor68-singular-points,
  milnor-orlik}, consider the polynomial
$f_{k}:\C^{n+1}\to\C$ given by
\begin{equation*}
  f_{k}(z) = z_{0}^{k} + z_{1}^{2} +\cdots+ z_{n}^{2},
\end{equation*}
and let
$W_{k} = \set{z \in \C^{n+1} :
  f_{k}(z)=\epsilon}$ be its \emph{Milnor
  fiber} for some small $\epsilon > 0$. The
space $W_{k}$ is a Liouville manifold, and
its ideal boundary $Y = W_{k} \cap \bd B(r)$
(for $r$ sufficiently large relative to
$\epsilon$) is a contact manifold.

The manifold $W_{k}$ carries a natural
$\Z/2\Z$-action generated by the involution
\begin{equation*}
  \iota(z_{0},z_{1},\dots,z_{n}) = (z_{0},-z_{1},\dots,-z_{n}).
\end{equation*}
The fixed points of $\tau$ in $W_{k}$ occur
where $z_{1}=\dots=z_{n}=0$. Imposing the
defining equation $f_{k}(z)=\epsilon$, we
find exactly $k$ fixed points corresponding
to the $k$-th roots of $\epsilon$. Because we
choose $r$ such that $r^{k} \gg \epsilon$,
these fixed points lie deep in the interior
of $W_{k}$, so $\iota$ acts freely on the
boundary $Y$. Consequently, $W_{k}$ is a
valid $\Z/2\Z$-filling for the quotient
contact manifold $Y / (\Z/2\Z)$. Because the
action has non-trivial stabilizers, Theorem
\ref{theorem:main-R} applies to guarantee the
existence of the spectral invariants $c_{R}$
for $Y / (\Z/2\Z)$.

It is worth noting that $W_{k}$ typically
does not satisfy the hypotheses of Theorem
\ref{theorem:main-G} for $k\ge 2$. As
discussed in
\cite[Theorem~6.3]{kwon-van-koert-BLMS-2016},
the Milnor fibers of isolated singularities
of weighted homogeneous polynomials (such as
$W_{k}$) have non-vanishing symplectic
cohomology (because they possess weakly exact
Lagrangians).

However, to obtain a measurement $\mu$
satisfying properties
\ref{item:G-monotonicity} through
\ref{item:G-discriminant}, we can pass to the
stabilization $W_{k} \times \C^{d}$ for
$d \ge 1$, as in \S\ref{sec:stabilizations};
this satisfies the hypotheses of Theorem
\ref{theorem:main-G}.

\subsection{Infinity categorification of the
  Eliashberg--Polterovich relation}
\label{sec:infin-categ-eliashb-intro}

The arguments used in the proofs of Theorems
\ref{theorem:main-R} and \ref{theorem:main-G} are
based on mapping cones of chain maps. We briefly
preview the set-up:
\begin{itemize}
\item to each contact isotopy $\varphi_{t}$ lying
  in the complement of the discriminant (see
  Definition \ref{definition:W-contractible}), we
  will associate a chain complex
  $\mathrm{CF}_{\mathrm{eq}}(\varphi_{t})$;
\item to each non-negative path $\varphi_{s,t}$
  \eqref{eq:non-neg-path} between $\varphi_{0,t}$
  and $\varphi_{1,t}$ we will associate a chain
  map
  $\mathrm{CF}_{\mathrm{eq}}(\varphi_{0,t})\to
  \mathrm{CF}_{\mathrm{eq}}(\varphi_{1,t})$;
\item the mapping cone of this chain map is
  related to the intersections of the path
  $\varphi_{s,t}$ with the discriminant.
\end{itemize}

It is known that it is difficult to work with such
cones on the level of homology groups (e.g., this
leads to the theory of triangulated
categories). Many arguments can be presented much
more simply when one is working ``on chain
level.'' On the other hand, it is also known to be
somewhat difficult to properly keep track of chain
homotopy terms when working on chain level. A
modern but well-established approach to working on
chain level is the theory of $\infty$-categories
following \cite{lurie-HTT-2009,lurie-HA-2017}.

The first goal in this section (continued in
\S\ref{sec:infin-categ-contact-isotopies}) is
to introduce an $\infty$-category
$\mathscr{C}(Y)$ of contact
isotopies. Roughly speaking, an infinity
category is a collection of sets
$\mathscr{C}_{k}(Y)$ of \emph{$k$-simplices},
for $k=0,1,\dots$, together with ``face'' and
``degeneracy'' relationships between these
sets of simplices. The $0$-simplices should
be considered as the ``objects'' of the
category and the $1$-simplices should be
considered as the ``morphisms.'' In our
construction, the $0$-simplices
$\mathscr{C}_{0}(Y)$ are contact isotopies
$\varphi_{t}$ lying outside the
discriminant\footnote{The construction admits
  a minor variation where
  $\mathscr{C}_{0}(Y)$ is the complement of
  the $W$-contractible discriminant.} and the
$1$-simplices $\mathscr{C}_{1}(Y)$ are
essentially the non-negative paths
$\varphi_{s,t}$ as in
\eqref{eq:non-neg-path}.\footnote{The precise
  definition of $\mathscr{C}_{1}(Y)$ uses the
  notion of a Moore path; see Definition
  \ref{definition:moore-path}.} Each
$1$-simplex
$\varphi_{s,t}\in \mathscr{C}_{1}$ has two
\emph{faces} in $\mathscr{C}_{0}(Y)$, and
these are simply the endpoints of the
non-negative path.

In a general $\infty$-category, the
$k$-simplices represent homotopies between
morphisms and their compositions. The
definition of the set of $k$-simplices
$\mathscr{C}_{k}(Y)$ in our $\infty$-category
and the face and degeneracy maps will be
explained in
\S\ref{sec:infin-categ-contact-isotopies}.

Our main structural theorem is that the
\emph{equivariant Floer cohomology
  complex}\footnote{Here we are referring to the
  ``Borel'' version of equivariant cohomology, as
  advanced by \cite{seidel-smith-GAFA-2010,
    seidel-eq-pop, shelukhin-zhao-JSG-2021,
    gonzalez-mak-pomerleano-arXiv-2023,
    cazassus-jtopol-2024, sampietro-christ-2025}.}
induces an $\infty$-functor:
\begin{equation}\label{eq:infinity-functor}
  \mathrm{CF}_{\mathrm{eq}}:\mathscr{C}(Y)\to \mathrm{N}_{\mathrm{dg}}\mathrm{Ch}(\mathbf{k}[[x]]).
\end{equation}
The right hand side is the so-called
\emph{dg-nerve} of the category of
graded\footnote{Our gradings will be valued in the
  group $\Z/2\Z$ (or sometimes the trivial
  group).} chain complexes over the ring
$\mathbf{k}[[x]]$. The only piece of information
we need about this dg-nerve is \emph{how to define
  a functor valued in it}. Luckily the definition
is rather compact:

\begin{definition}[Taken from
  {\cite[\S7.6]{pardon-GT-2016}}]\label{definition:infinity-functor}
  An $\infty$-functor \eqref{eq:infinity-functor}
  consists of the following assignments:
  \begin{itemize}
  \item each zero simplex $\sigma=\varphi_{t}$ is
    sent to a (graded) $\mathbf{k}[[x]]$-module:
    $$V_{\sigma}=\mathrm{CF}_{\mathrm{eq}}(\varphi_{t})$$
    with a differential $d_{\sigma}$ of degree
    $1$;
  \item each $n$-dimensional simplex $\sigma$ is
    sent to a map:
    $$\mathfrak{c}_{\sigma}:V_{\sigma|0}\to
    V_{\sigma|n}$$ of degree $1-n$; here
    $\sigma|j$ is the $j$th vertex of $\sigma$.
  \end{itemize}
  These maps are required to satisfy the
  structural equation:
  \begin{equation}\label{eq:dg-nerve}
    \sum_{j=1}^{n-1}(-1)^{j}(\mathfrak{c}_{\sigma|[j\dots n]}\circ \mathfrak{c}_{\sigma|[0\dots j]}-\mathfrak{c}_{\sigma|[0\dots \hat{j}\dots n]})=\mathfrak{c}_{\sigma}\circ d_{\sigma|0}+(-1)^{n}d_{\sigma|n}\circ \mathfrak{c}_{\sigma};
  \end{equation}
  additionally we require the following for
  degenerate simplices:
  \begin{equation}\label{eq:dg-nerve-degen}
    \mathfrak{c}_{\sigma}=\left\{
      \begin{aligned}
        &\id&&\text{ if $\sigma$ is a degenerate $1$-simplex},\\
        &0&&\text{ if $\sigma$ is a degenerate $k$-simplex with $k>1$}.
      \end{aligned}
    \right.
  \end{equation}
\end{definition}
Importantly, because $1$-simplices are sent to
chain maps, we can talk about their mapping cones
in a simple and canonical way. We now state the
theorem which will be used to prove Theorem
\ref{theorem:main-R} and \ref{theorem:main-G}:
\begin{theorem}\label{theorem:main-infinity}
  If $Y$ admits an aspherical $\Z/p\Z$-filling $W$
  with $p$ prime, then there is an
  $\infty$-functor \eqref{eq:infinity-functor} for
  $\mathbf{k}=\Z/p\Z$ satisfying the following:
  \begin{enumerate}[label=(\alph*)]
  \item\label{theorem:main-infinity-0} the output
    $\mathrm{CF}_{eq}(\varphi_{t})$ is a finitely
    generated $\mathbf{k}[[x]]$-module, for each
    zero simplex
    $\varphi_{t}\in \mathscr{C}_{0}(Y)$;
  \item\label{theorem:main-infinity-discriminant} if $\varphi_{s,t}$ is a $1$-simplex
    in $\mathscr{C}(Y)$ such that
    $\varphi_{s,t}$ does not lie on the
    $W$-contractible discriminant, for each
    $s$, then the associated chain map is a
    chain homotopy equivalence;
  \item\label{theorem:main-infinity-1} the
    homology of the cone of the chain map
    associated to any $1$-simplex in
    $\mathscr{C}(Y)$ is $x$-torsion (i.e.,
    multiplication by $x^{d}$ on the homology of
    the cone acts by $0$, for some positive
    integer $d$);
  \item\label{theorem:main-infinity-2} if the
    $\Z/p\Z$-action on $W$ has at least one fixed
    point, then the homology of
    $\mathrm{CF}_{\mathrm{eq}}(\varphi_{t})$ is
    not $x$-torsion, for any object
    $\varphi_{t}\in \mathscr{C}_{0}(Y)$;
  \item\label{theorem:main-infinity-3} for
    $\epsilon$ small enough, the cone of
    $\mathrm{CF}_{\mathrm{eq}}(R_{-\epsilon
      t})\to
    \mathrm{CF}_{\mathrm{eq}}(R_{\epsilon
      t})$ is not an acyclic
    complex.\footnote{Recall that an acyclic
      complex is one whose homology is zero.}
  \end{enumerate}
\end{theorem}
The construction of the $\infty$-functor is
performed in
\S\ref{sec:equiv-floer-compl}. The proof of
\ref{theorem:main-infinity-1} uses a local
Floer homology argument in
\S\ref{sec:local-floer-cohom}. The proofs of
\ref{theorem:main-infinity-2} and
\ref{theorem:main-infinity-3} uses the PSS
isomorphism (Theorem
\ref{theorem:pss}). Property
\ref{theorem:main-infinity-3} is added to
ensure the theorem is not trivially satisfied
by a constant functor; we will in fact prove
something stronger and identify a particular
homology class in
$\mathrm{CF}_{\mathrm{eq}}(R_{\epsilon t})$
which projects to a non-zero element in the
cone from \ref{theorem:main-infinity-3}.

\subsubsection{Homology level functor}
\label{sec:homology-level-functor}

Taking homology of the $\infty$-functor in Theorem
\ref{theorem:main-infinity} gives an ordinary
functor:
\begin{equation}\label{eq:taking-homology}
  \mathrm{HF}_{\mathrm{eq}}:\mathrm{h}\mathscr{C}(Y)\to \mathrm{Mod}(\mathbf{k}[[x]])
\end{equation}
where $\mathrm{h}\mathscr{C}(Y)$ is the so-called
\emph{homotopy category}\footnote{see
  \cite[\S1.2.3]{lurie-HTT-2009} for the
  definition of $\mathrm{h}\mathscr{C}$.} of
$\mathscr{C}(Y)$ and
$\mathrm{Mod}(\mathbf{k}[[x]])$ is the category of
modules over $\mathbf{k}[[x]]$. Since
$\mathrm{HF}_{\mathrm{eq}}$ is a regular functor
defined on a regular category, we can take its
colimit in the usual sense. This leads to
the definition:
\begin{definition}
  The equivariant symplectic cohomology
  $\mathrm{SH}_{\mathrm{eq}}(W)$ is defined to be
  the colimit of \eqref{eq:taking-homology} in
  the category of $\mathbf{k}[[x]]$-modules.
\end{definition}
The $\mathbf{k}[[x]]$-module
$\mathrm{SH}_{\mathrm{eq}}(W)$ will be used to
define the invariants $c_{R}$ and $\mu$. As we
shall see, the invariant $\mu$ relies on the
$\mathbf{k}[[x]]$-module structure in an important
way.

\subsection{PSS map}
\label{sec:pss-introduction}

There is a larger $\infty$-category
$\bar{\mathscr{C}}(Y)$ such that
$\bar{\mathscr{C}}_{0}(Y)=\mathrm{CI}(Y)$ and
$\mathscr{C}(Y)\subset \bar{\mathscr{C}}(Y)$ is
the full subcategory spanned by the contact
isotopies without discriminant points, (or
$W$-contractible discriminant points, both
variants work equally well).

Inside this larger $\infty$-category
$\bar{\mathscr{C}}(Y)$, we will define
another $\infty$-category $\mathscr{P}(Y)$,
which we call the \emph{PSS
  category}.\footnote{The name comes from
  \cite{piunikhin-salamon-schwarz-1996}.}
Roughly speaking, simplices in
$\mathscr{P}(Y)$ have some of their vertices
in $\mathscr{C}(Y)$, and the other vertices
are fixed at the constant system $\id$ (see
\S\ref{sec:PSS-category} for the precise
definition). For the purposes of the
introduction, it suffices to say:
\begin{itemize}
\item any $1$-simplex $\id\to \varphi_{t}$ in
  $\bar{\mathscr{C}}(Y)$, with
  $\varphi_{t}\in \mathscr{C}(Y)$, is in
  $\mathscr{P}(Y)$.
\end{itemize}
Such $1$-simplices represent ``continuation data''
from the constant system $\id$ to some output
system $\varphi_{t}\in \mathrm{CI}(Y)$, which must
necessarily satisfy $\id \le \varphi_{t}$ in the sense of \cite{eliashberg-polterovich-gafa-2000}.

\begin{theorem}\label{theorem:pss}
  Assume the hypotheses of Theorem
  \ref{theorem:main-infinity}.  Counting solutions
  of moduli spaces of type
  \cite{piunikhin-salamon-schwarz-1996,frauenfelder-schlenk-israeljm-2007}
  defines an $\infty$-functor:
  \begin{equation}\label{eq:pre-PSS-diagram}
    \mathscr{P}(Y)\to \mathrm{N}_{\mathrm{dg}}(\mathrm{Ch}(\mathbf{k}[[x]]))
  \end{equation}
  where:
  \begin{itemize}
  \item non-identity $0$-simplices $\varphi_{t}$
    are sent to the outputs
    $\mathrm{CF}_{\mathrm{eq}}(\varphi_{t})$ of
    the functor from Theorem
    \ref{theorem:main-infinity};
  \item the identity $0$-simplex $\id$ is sent to
    the equivariant Morse cohomology complex
    $\mathrm{CM}_{\mathrm{eq}}(P)$ associated to a
    $G$-equivariant Morse-Smale pseudogradient $P$
    which points outwards in the convex end of $W$
    (the definition is reviewed in
    \ref{sec:defin-borel-equiv}).
  \end{itemize}
  By virtue of how $\infty$-functors work, the
  aforementioned $1$-simplices in $\mathscr{P}(Y)$
  joining $\id$ to a
  $\varphi_{t}\in \mathscr{C}(Y)$ are sent to
  chain maps:
  \begin{equation*}
    \mathrm{PSS}:\mathrm{CM}_{\mathrm{eq}}(X)\to \mathrm{CF}_{\mathrm{eq}}(\varphi_{t}).
  \end{equation*}
  If we apply this to a $1$-simplex of the form
  $\varphi_{s,t}=R_{\epsilon st}$, for
  $\epsilon>0$ small enough, then:
  \begin{equation*}
    \mathrm{PSS}:\mathrm{CM}_{\mathrm{eq}}(X)\to \mathrm{CF}_{\mathrm{eq}}(R_{\epsilon t}).
  \end{equation*}
  is a chain homotopy equivalence which can
  be understood as an equivariant version of
  the PSS isomorphism of
  \cite{piunikhin-salamon-schwarz-1996,
    frauenfelder-schlenk-israeljm-2007}.
\end{theorem}

\subsubsection{The negative slope PSS map}
\label{sec:negative-slope-pss}

There is a ``reflected'' version of the
aforementioned set-up. Inside
$\bar{\mathscr{C}}(Y)$, one can also define
an $\infty$-category $\mathscr{N}(Y)$, which
we call the \emph{negative PSS category}; see
\S\ref{sec:PSS-category} for the precise
definition. For the purposes of the
introduction, it suffices to say:
\begin{itemize}
\item any $1$-simplex $\varphi_{t}\to \id$ in
  $\bar{\mathscr{C}}(Y)$, with
  $\varphi_{t}\in \mathscr{C}(Y)$, is in
  $\mathscr{N}(Y)$.
\end{itemize}
Such $1$-simplices represent ``continuation
data'' from $\varphi_{t}$ to the constant
system $\id$, and so the existence of a
$1$-simplex implies $\varphi_{t}\le \id$.

\begin{theorem}\label{theorem:negative-pss}
  Assume the hypotheses of Theorem
  \ref{theorem:main-infinity}. Counting solutions
  of moduli spaces of type
  \cite{piunikhin-salamon-schwarz-1996,frauenfelder-schlenk-israeljm-2007}
  defines an $\infty$-functor:
  \begin{equation}\label{eq:pre-PSS-diagram}
    \mathscr{N}(Y)\to \mathrm{N}_{\mathrm{dg}}(\mathrm{Ch}(\mathbf{k}[[x]]))
  \end{equation}
  where:
  \begin{itemize}
  \item non-identity $0$-simplices $\varphi_{t}$
    are sent to the outputs
    $\mathrm{CF}_{\mathrm{eq}}(\varphi_{t})$ of
    the functor from Theorem
    \ref{theorem:main-infinity};
  \item the identity $0$-simplex $\id$ is
    sent to the equivariant Morse cohomology
    complex $\mathrm{CM}_{\mathrm{eq}}(-P)$
    associated to a $G$-equivariant
    Morse-Smale pseudogradient $-P$ which
    points \emph{inwards} in the convex end
    of $W$ (here $P$ points outwards so $-P$
    points inwards).
  \end{itemize}
  By virtue of how $\infty$-functors work,
  the aforementioned $1$-simplices in
  $\mathscr{N}(Y)$ joining
  $\varphi_{t}\in \mathscr{C}(Y)$ to $\id$
  are sent to chain maps:
  \begin{equation*}
    \mathrm{PSS}:\mathrm{CF}_{\mathrm{eq}}(\varphi_{t})\to \mathrm{CM}_{\mathrm{eq}}(-X).
  \end{equation*}
  If we apply this to a $1$-simplex of the form
  $\varphi_{s,t}=R_{\epsilon(s-1)t}$, for
  $\epsilon>0$ small enough, then:
  \begin{equation*}
    \mathrm{PSS}:\mathrm{CF}_{\mathrm{eq}}(\varphi_{t})\to \mathrm{CM}_{\mathrm{eq}}(-X).
  \end{equation*}
  is a chain homotopy equivalence.
\end{theorem}

\subsubsection{Continuation across slope zero}
\label{sec:cont-across-slope}

The two complexes
$\mathrm{CM}_{\mathrm{eq}}(X)$ and
$\mathrm{CM}_{\mathrm{eq}}(-X)$ are naturally
related to one another, in that there is a
canonical chain homotopy class of maps:
\begin{equation*}
  \mathfrak{c}_{\mathrm{CM}}:\mathrm{CM}_{\mathrm{eq}}(-X)\to \mathrm{CM}_{\mathrm{eq}}(X)
\end{equation*}
defined by counting Morse continuation lines;
the definition is reviewed in
\S\ref{sec:defin-morse-continuation
  map}. This Morse continuation map interacts
with nicely PSS and the Floer cohomology
continuation maps; this result is well-known
to experts and similar statements appear
already in
\cite{cieliebak-frauenfelder-oancea-ASENS-2010,
  ritter_circle_actions,
  albers-kang-vanishing-RFH,
  cant-hedicke-kilgore-arXiv-2023}. We
include the precise formulation for
completeness.
\begin{theorem}\label{theorem:commutative-square}
  Let $\mathfrak{c}$ be the continuation map
  associated to the $1$-simplex $\Phi$ in
  $\mathscr{C}(Y)$ such that $\Phi^{01}$ is
  the path
  $\varphi_{s,t}=R_{\epsilon(2s-1)t}$ joining
  $R_{-\epsilon t}$ to $R_{\epsilon t}$. Then
  the following diagram commutes up to chain
  homotopy:
  \begin{equation*}
    \begin{tikzcd}
      \mathrm{CF}_{\mathrm{eq}}(R_{-\epsilon t})\arrow[d,"\mathrm{PSS}"]\arrow[r,"\mathfrak{c}"]&\mathrm{CF}_{\mathrm{eq}}(R_{\epsilon t})\arrow[from=2-2,"\mathrm{PSS}"]\\
      \mathrm{CM}_{\mathrm{eq}}(-X)\arrow[r,"\mathfrak{c}_{\mathrm{CM}}"]&\mathrm{CM}_{\mathrm{eq}}(X).
    \end{tikzcd}
  \end{equation*}
  The PSS maps are those from Theorem
  \ref{theorem:pss} and
  \ref{theorem:negative-pss}.
\end{theorem}

\begin{remark}
  It is a bit awkard to have three PSS
  theorems; in an ideal world one imagines
  there is a single $\infty$-category
  containing both $\mathscr{P}(Y)$ and
  $\mathscr{N}(Y)$, with a single
  $\infty$-functor, such that Theorem
  \ref{theorem:commutative-square} is a
  simple corollary rather than an independent
  theorem. However, the proof of Theorem
  \ref{theorem:commutative-square} involves
  nodal gluing and it is not a priori clear
  how to package these nodal interfaces into
  a single infinity-category.
\end{remark}

\subsection{On canonicity}
\label{sec:canonicity}

The reader may wonder whether our construction is
``canonical,'' i.e., whether any two readers who
follow our construction will obtain the ``same''
$\infty$-functor. Our invariant is not canonical,
in the same way ``the Morse complex'' is not
canonical: two readers can follow each step of the
construction correctly but pick auxiliary data
differently and end up with non-isomorphic
complexes $\mathrm{CF}_{\mathrm{eq}}(\varphi_{t})$
(they should, at least, be quasi-isomorphic).

Our invariant is canonical in a weaker sense:
there exists a \emph{canonical diagram} in the
category of $\infty$-categories:
\begin{equation*}
  \begin{tikzcd}
    {\mathscr{D}(W)}\arrow[r,"{\mathrm{CF}_{\mathrm{eq}}}"]\arrow[d,"\pi"]&{\mathrm{N}_{\mathrm{dg}}\mathrm{Ch}(\mathbf{k}[[x]])}\\
    {\mathscr{C}(Y)}&{}
  \end{tikzcd}
\end{equation*}
where the vertical map is a \emph{trivial Kan
  fibration}; see \cite[\S2]{lurie-HTT-2009} or
\S\ref{sec:survey-infty-categories} for the
definition. This diagram is canonical, in that two
readers will agree on the result. This statement
is of course not entirely mathematical (and
borders on the philosophical). A more mathematical
formulation of ``canonicity'' is that the diagram
is well-defined up to a canonical isomorphism. A
similar statement holds with $\mathscr{C}(Y)$
replaced by the PSS category $\mathscr{P}(Y)$.

The $\infty$-functor from Theorem
\ref{theorem:main-infinity} is defined as a
precomposition $\mathrm{CF}_{\mathrm{eq}}\circ s$,
where $s$ is a section of $\pi$ (a similar
statement holds for the $\infty$-functor in
Theorem \ref{theorem:pss}). Well-known results
from \cite{lurie-HTT-2009} imply these sections
exist in abundance. Moreover, using basic
properties of trivial Kan fibrations, one can
relate the two functors
$\mathrm{CF}_{\mathrm{eq}}\circ s$ and
$\mathrm{CF}_{\mathrm{eq}}\circ s'$ by picking a
``path'' between $s$ and $s'$ in the ``space'' of
sections and thereby prove the outputs are chain
homotopy equivalent. Such a comparison between
sections $s,s'$ will prove the homology level
functors agree up to natural isomorphism. In this
sense, the lack of canonicity is isolated into the
single ``contractible choice'' of section $s$.

Of course, such geometric language as ``path'' and
``space'' needs to be interpreted properly, and we
return to this more precisely in
\S\ref{sec:survey-infty-categories}.

\subsection{Definition of the measurements}
\label{sec:defin-meas}

The first step is to appeal to the fact:
{\itshape if a prime number $p$ divides the
  order of the stabilizer of some point
  $w\in W$, then there is a $\Z/p\Z$-subgroup
  of the stabilizer}.\footnote{See, e.g.,
  \cite[Theorem IV.2.1]{aluffi-GSM-2009}.}
Then define $Y'$ to be the quotient by the
action of this $\Z/p\Z$ subgroup. Since $Y'$
is a covering of $Y$, it is sufficient to
prove Theorem \ref{theorem:main-R} and
\ref{theorem:main-G} for $Y'$ instead of
$Y$. Thus, for the remainder of the
discussion, we may suppose that $Y=Y'$ and
$G=\Z/p\Z$. We can therefore apply Theorems
\ref{theorem:main-infinity} and
\ref{theorem:pss}.

The quantities $c_{R}$ and $\mu$ satisfying
Theorems \ref{theorem:main-R} and
\ref{theorem:main-G} are defined in terms of a
\emph{unit element}
$1\in \mathrm{SH}_{\mathrm{eq}}$. It is
constructed as follows: the PSS chain maps
$\mathrm{CM}_{\mathrm{eq}}(X)\to
\mathrm{CF}_{\mathrm{eq}}(\varphi_{t})$ from
Theorem \ref{theorem:pss}, associated to a
$1$-simplex $\varphi_{s,t}$, yield cycles
$1(\varphi_{s,t})\in
\mathrm{CF}_{\mathrm{eq}}(\varphi_{t})$
corresponding to the sum of local minima in
$\mathrm{CM}_{\mathrm{eq}}(X)$.
\begin{lemma}\label{lemma:unit-well-defined-in-colimit}
  The image of the element $1(\varphi_{s,t})$ in
  the colimit $\mathrm{SH}_{\mathrm{eq}}$ is
  independent of the choice of $1$-simplex
  $\varphi_{s,t}$.
\end{lemma}

This lemma (proved in
\S\ref{sec:unit-element-well}) furnishes a
canonical element
$1\in \mathrm{SH}_{\mathrm{eq}}$. Following
\cite{cant-sh-barcode, djordjevic_uljarevic_zhang,
  cant-arXiv-2024,
  djordjevic-uljarevic-zhang-arXiv-2025} we
define:
\begin{equation}\label{eq:defin-spec-inv}
  c_{R}(\varphi_{t}):=\inf\set{s:1\text{ is in the image of }\mathrm{HF}_{\mathrm{eq}}(\varphi_{t}^{-1}\circ R_{st})\to \mathrm{SH}_{\mathrm{eq}}(W)}.
\end{equation}
In \S\ref{sec:axioms-spectr-invar} we prove that
this definition of $c_{R}$ satisfies
\ref{item:R-spectrality} through
\ref{item:R-sub-additivity}.

Turning to $\mu$, the new ingredient in Theorem
\ref{theorem:main-G} is the hypothesis that
$\mathrm{SH}(W;\mathbf{k})=0.$ We first explain
how the classical symplectic cohomology fits into
the framework of our paper. Define the
\emph{non-equivariant quotient} as a tensor
product:\footnote{Note that this is a chain level
  construction, which cannot be directly accessed
  from the homology level invariants.}
\begin{equation}\label{eq:non-eq-quotient}
  \mathrm{CF}_{\mathrm{neq}}(\varphi_{t}):=\mathrm{CF}_{\mathrm{eq}}(\varphi_{t})\otimes_{\mathbf{k}[[x]]} \mathbf{k},
\end{equation}
where $x$ acts by $0$ on $\mathbf{k}$. Its
homology is denoted by
$\mathrm{HF}_{\mathrm{neq}}(\varphi_{t})$ and we
define:
\begin{equation*}
  \mathrm{SH}_{\mathrm{neq}}(W):=\text{colimit of $\mathrm{HF}_{\mathrm{neq}}(\varphi_{t})$ over $\varphi_{t}\in \mathrm{h}\mathscr{C}(Y)$}.
\end{equation*}
We show in \S\ref{sec:axioms-integ-valu} that:
\begin{equation}\label{eq:implication}
  \mathrm{SH}(W;\mathbf{k})=0\implies \mathrm{SH}_{\mathrm{neq}}(W)=0,
\end{equation}
where $\mathrm{SH}(W;\mathbf{k})$ agrees with
standard definitions of symplectic cohomology,
such as the one in \cite{seidel-IP-2008}.

From \eqref{eq:non-eq-quotient} and
\eqref{eq:implication}, there is a long exact
sequence:
\begin{equation*}
  \begin{tikzcd}
    {\cdots}\arrow[r,"{}"] &{\mathrm{SH}_{\mathrm{eq}}}\arrow[r,"{x}"] &{\mathrm{SH}_{\mathrm{eq}}}\arrow[r,"{}"] &{0}\arrow[r,"{}"] &{\cdots},
  \end{tikzcd}
\end{equation*}
The measurement $\mu$ satisfying Theorem
\ref{theorem:main-G} is:
\begin{equation*}
  \mu(\varphi_{t}):=\sup\set{d\in \Z:x^{-d}1\text{ is in the image of }\mathrm{HF}_{\mathrm{eq}}(\varphi_{t})\to \mathrm{SH}_{\mathrm{eq}}(W)};
\end{equation*}
we extend from $\mathscr{C}_{0}(Y)$ (the
complement of the discriminant) to all of
$\mathrm{CI}(Y)$ by continuity from above, i.e.,
in the unique way compatible with
\ref{item:G-continuity}.

In \S\ref{sec:axioms-integ-valu} we prove
\eqref{eq:implication} and show that $\mu$
satisfies \ref{item:G-monotonicity} through
\ref{item:G-discriminant}.

\subsection{Measurements associated to positive loop of contactomorphisms}
\label{sec:integ-valu-meas}

Following
\cite{fraser-polterovich-rosen-AMQ-2018}, see
also \cite[\S1.4.7]{cant-arXiv-2024}, one can
use a positive loop of contactomorphisms in
lieu of a Reeb flow when defining the
spectral invariants; this naturally leads to
an integer valued measurements.

Suppose the hypotheses of Theorem
\ref{theorem:main-R}, and let $\zeta_{t}$ be
a positive loop of contactomorphisms (which
is necessarily non-contractible!). Then
define:
\begin{equation*}
  \ell(\zeta_{t};\varphi_{t}):=\sup \set{k:\mathrm{HF}_{eq}(\varphi_{t}\zeta_{t}^{-k})\to \mathrm{SH}_{eq}(W)\text{ hits the unit}}.
\end{equation*}
Then, as is well-known, $\ell(\zeta_{t};-)$
is an unbounded, integer valued, monotone,
conjugation invariant measurement, and is
constant on the path connected components of
the complement of the discriminant.

It is a natural question whether not the
invariant $\mu$ we construct in Theorem
\ref{theorem:main-G} is related to such
``positive loop'' measurements; we plan to
investigate this further in future work.

\subsection{Local Floer cohomology and cones of
  continuation maps}
\label{sec:local-floer-cohom-and-cones-intro}

An important ingredient is part
\ref{theorem:main-infinity-1} of Theorem
\ref{theorem:main-infinity}, which asserts that
cones of continuation maps are always
$x$-torsion. The key idea is that the cone of a
continuation map can be understood via the
\emph{local Floer homologies} of isolated
crossings with the discriminant. The rigorous
development of this idea is the object of
\S\ref{sec:local-floer-cohom}. Beyond the
statement about cones of continuation maps being
$x$-torsion, another important consequence of this
theory is:
\begin{theorem}\label{theorem:zig-zag}
  
  Consequently, if two isotopies
  $\varphi_{0,t},\varphi_{1,t}$ can be joined in
  the complement of the $W$-contractible
  discriminant by \emph{any} path (not assumed to
  be non-negative), then the morphisms
  $\mathrm{HF}_{\mathrm{eq}}(\varphi_{i,t})\to
  \mathrm{SH}_{\mathrm{eq}}$ are isomorphic.
\end{theorem}
The first part of the statement goes back to
\cite{uljarevic-zhang-JFPTA-2022}. The method of
proof of the second part is to approximate any
path by a zig-zag of positive and negative paths;
see \cite[\S9]{uljarevic-compositio-2023} and
\cite[\S2.1.3]{djordjevic-uljarevic-zhang-arXiv-2025}
for further discussion of such zig-zags. We prove
Theorem \ref{theorem:zig-zag} in
\S\ref{sec:proof-zig-zag}.

\begin{remark}
  One can localize the category
  $\mathrm{h}\mathscr{C}(Y)$ at the
  subcategory of morphisms which do not cross
  the discriminant (or the $W$-contractible
  discriminant) producing a new category
  $\mathscr{L}$ with a
  factorization:
  $$\mathrm{h}\mathscr{C}(Y)\to \mathscr{L}\to
  \mathbf{k}[[x]]\text{-mod}.$$ By the
  zig-zag trick, the fundamental groupoid
  $\Pi$ of the complement of the discriminant
  (without any positivity assumption) admits
  a functor into $\mathscr{L}$. Thus the
  restriction of $\mathrm{HF}_{\mathrm{eq}}$
  to $\Pi$ is a local system on the
  complement of the discriminant. Of course
  this discussion is rather speculative, and
  the details of the ``localization of
  morphisms which do not cross the
  discriminant'' step seems to be an
  interesting avenue of exploration.
\end{remark}

\subsection{On $\infty$-categories}
\label{sec:survey-infty-categories}

Since the \(\infty\)-categorical language is not
yet completely standard in symplectic geometry we
recall here the reason for its appearance in the
Floer theoretic setting, as well as some basic
properties; the reader comfortable with
\(\infty\)-categories and their role in Floer
theory may skip this section.

One of the fundamental issues in Floer theory is
that of \emph{invariance}: how to show that the
chain complexes we construct do not depend (in a
suitable sense) on any choices auxiliary to the
geometry at hand (e.g. almost complex structures,
choices of Morse-Smale pseudogradient,
etc). Morally, one expects these extra choices to
be ultimately irrelevant whenever the space of
such choices is contractible, and more generally,
certain allowable paths in the data should give
rise to chain maps.

In our setting we are naturally provided with the
following setup: there is a topological category
\( \mathscr{C}_{\mathrm{top}}(Y) \) whose objects
are contact isotopies, and whose morphism spaces
consist of the non-negative paths connecting two
given objects. Lying over this (via a natural
forgetful functor) is a larger topological
category \( \mathscr{D}_{\mathrm{top}}(W) \) whose
objects are all data necessary to define the Floer
cohomology complex, and whose morphism spaces
consist of the spaces of regular continuation data
with specified endpoints. However, this
topological category
$\mathscr{D}_{\mathrm{top}}(W)$ and the
dg-category of chain complexes do not directly
interact very nicely. The language of
quasi-categories systematically developed in
\cite{lurie-HTT-2009} provides a setting into
which both dg-categories and topological
categories can be naturally integrated by means of
taking the \emph{topological and dg-nerves},
\(
\mathrm{N}_{\mathrm{top}},\mathrm{N}_{\mathrm{dg}}
\) respectively.

In the present paper the aforementioned
$\infty$-category $\mathscr{C}(Y)$ models the
topological nerve of the category
$\mathscr{C}_{\mathrm{top}}(Y)$, and we define in
\S\ref{sec:auxil-categ-floer} an $\infty$-category
of Floer data $\mathscr{D}(W)$ modeling the
topological nerve of
$\mathscr{D}_{\mathrm{top}}(W)$; there is
forgetful functor
$\mathscr{D}(W)\to \mathscr{C}(Y)$. We prove in
\S\ref{sec:defin-infty-funct} that this forgetful
map is a trivial Kan fibration:

\begin{definition}[See
  {\cite[\S2]{lurie-HTT-2009}}]\label{definition:trivial-kan-fibration}
  A \emph{trivial Kan fibration} is a morphism of
  simplicial sets $\mathscr{D}\to \mathscr{C}$
  such that following lifting diagram is
  satisfied:
  \begin{equation*}
    \begin{tikzcd}
      \partial \Delta^n \arrow[r] \arrow[d,"\iota"] & \mathscr{D} \arrow[d,"\pi"]\\
      \Delta^n \arrow[r]
      \arrow[ur,dashed,"\exists"] & \mathscr{C}
    \end{tikzcd}
  \end{equation*}
  for all \( n \geq 0, \) where $\iota$ is the
  obvious inclusion map. If
  $\mathscr{S}\to \Delta^{0}$ is a trivial Kan
  fibration, then $\mathscr{S}$ is called a
  \emph{contractible Kan complex}, and should be
  thought of as the simplicial set analog of a
  contractible space.
\end{definition}

Trivial Kan fibrations and contractible Kan
complexes enjoy many convenient properties; the
most important for our purposes is the following:
\begin{proposition}\label{proposition:trivial-kan-fibration}
  Let \( \pi:\mathscr{D}\rightarrow \mathscr{C} \)
  be a trivial Kan fibration. Then \( \pi \)
  admits a section $s$. Moreover, the collection
  of such sections forms the zero simplices
  in a contractible Kan complex $\mathscr{S}$.
\end{proposition}
\begin{proof}
  Define an $n$-simplex in the space of sections
  $\mathscr{S}$ to be a map
  $$s:\Delta^{n}\times B\to E$$ such that
  $\pi s=\pr_{B}$. There is an obvious structure
  of a simplicial set with these as the
  $n$-simplices. Since \( \pi \) is a trivial Kan
  fibration, monomorphisms (injections) of
  simplicial sets have the left lifting property
  against \( \pi \); see \cite[Example
  2.0.0.2]{lurie-HTT-2009}. By using this lifting
  property with the monomorphisms
  \( \partial\Delta^n \times B \hookrightarrow
  \Delta^n \times B \), we prove that this
  simplicial set of sections $\mathscr{S}$ is a
  contractible Kan complex.
\end{proof}

The general strategy is then as follows: There is
a canonical Floer complex functor:
\begin{equation*}
  \mathrm{CF}:\mathscr{D}(W)\rightarrow \mathrm{N}_{\mathrm{dg}}(\mathrm{Ch}(\mathbf{k}[[x]]))
\end{equation*}
since the left hand side consists of exactly the
data necessary to define Floer theory. Then every
section \( s \) of \( \pi \) induces a functor:
\[ \mathrm{CF}_{\mathrm{eq}}\circ s:\mathscr{C}(Y)
  \rightarrow
  \mathrm{N}_{\mathrm{dg}}(\mathrm{Ch}(\mathbf{k}[[x]])) \]
by composition. Contractibility of the space of
sections then says that any choice of \( s \)
(i.e., any compatible collection of Floer data)
induces homotopic Floer functors and, moreover,
all such homotopies are themselves homotopic,
etc. We return to this discussion in
\S\ref{sec:defin-infty-funct}.

Carrying out such a ``chain-level construction''
has the advantage of making certain constructions
(such as taking cones and tensor
products\footnote{We use a chain level
  construction in the definition of
  $\mathrm{HF}_{\mathrm{neq}}(\varphi_{t})$ as the
  homology of
  $\mathrm{CF}_{\mathrm{eq}}(\varphi_{t})\otimes_{\mathbf{k}[[x]]}\mathbf{k}.$})
considerably simpler. In our case, it is used in
\S\ref{sec:local-floer-cohom} to verify that a
certain Floer complex may be expressed as a
mapping cone for a very particular choice of data,
and it then follows by the abstract machinery that
any other choice of data extending the underlying
system has the same property.

Such setups have been leveraged before in Floer
theory to similar ends, cf.,
\cite{pardon-GT-2016},
\cite{hendricks-lipshitz-sarkar-plms-2020},
\cite[\S4]{ganatra-pardon-shende-pub-IHES-2020},
\cite[\S{A}]{varolgunes-GT-2021}.

\section{Equivariant Floer complex as infinity
  functor}
\label{sec:equiv-floer-compl}

The goal of the first subsection
\S\ref{sec:infin-categ-contact-isotopies} is to
introduce the $\infty$-category of contact
isotopies (continuing from the outline we gave in
\S\ref{sec:infin-categ-eliashb-intro}). In
\S\ref{sec:case-p-2} we prove Theorem
\ref{theorem:main-infinity} in the case $p=2$,
which has various differences and simplifications
from the general case $p\ge 3$, which we treat in
\S\ref{sec:case-p-3} and
\S\ref{sec:asymp-ops-orient-lines}.

% S2.1
\subsection{Infinity category of contact
  isotopies}
\label{sec:infin-categ-contact-isotopies}

As an $\infty$-category, $\mathscr{C}(Y)$ is,
first and foremost, a special type of
\emph{simplicial set}; see \cite[\S
A.2.7]{lurie-HTT-2009} for a review of simplicial
sets. This means that $\mathscr{C}(Y)$ is a
contravariant set-valued functor from the
\emph{combinatorial simplex category}, whose
objects are non-negative integers
$[0], [1],[2],\dots,$ and where the morphisms
$[i]\to [j]$ are the non-decreasing maps
$\set{0,\dots,i}\to \set{0,\dots,j}$. More
prosaically, $\mathscr{C}(Y)$ is the data of sets
$\mathscr{C}_{0}(Y),\mathscr{C}_{1}(Y),\mathscr{C}_{2}(Y),\dots$,
where $\mathscr{C}_{j}(Y)$ is called the set of
$j$\emph{-simplices}, together with morphisms:
\begin{equation}\label{eq:face-and-degeneracy}
  f^{*}:\mathscr{C}_{j}(Y)\to \mathscr{C}_{i}(Y)\text{ for every morphism }f:[i]\to [j],
\end{equation}
and such that the morphisms in
\eqref{eq:face-and-degeneracy} are functorial.

\subsubsection{Definition of the set of simplices}
\label{sec:defin-set-simplices}

We define\footnote{In this section, we primarily
  discuss the construction of
  $\mathscr{C}(Y)$. However, the discussion works
  equally well for constructing the larger
  category $\bar{\mathscr{C}}(Y)$ (see
  \S\ref{sec:pss-introduction}), the only
  difference being that $0$-simplices in the
  full-subcategory
  $\mathscr{C}(Y)\subset \bar{\mathscr{C}}(Y)$ are
  required to be in the complement of the
  discriminant.} $\mathscr{C}_{n}(Y)$ to be
certain coherent cubes of non-negative paths in
the group of contact isotopies $\mathrm{CI}(Y)$.

The formal construction uses the space of non-negative ``Moore paths'' in $\mathrm{CI}(Y)$. Let us introduce $\mathrm{CI}(Y)^{\R,+}\subset \mathrm{CI}(Y)^{\R}$ as the set of smooth non-negative paths as defined in \eqref{eq:non-neg-path}.

An element $(s\mapsto \varphi(s))\in \mathrm{CI}^{\R,+}(Y)$ is said to be \emph{stationary} on an interval $I\subset \R$ provided $\varphi(s_{1})=\varphi(s_{2})$ for all $s_{1},s_{2}\in I$.

\begin{definition}\label{definition:moore-path}
  A non-negative \emph{Moore path} in
  $\mathrm{CI}(Y)$ is an element:
  \begin{equation*}
    (w,\varphi)\in [0,\infty)\times \mathrm{CI}(Y)^{\R,+}
  \end{equation*}
  such that $\varphi$ is stationary on the
  intervals $(-\infty,0]$ and
  $[w,\infty)$. The space of non-negative
  Moore paths shall be denoted
  $\mathrm{NN}(Y)$. The values of the path at
  $0$ and $w$ are called the \emph{initial
    and terminal points} of the path. Two
  Moore paths $(w_{1},\varphi_{1})$ and
  $(w_{2},\varphi_{2})$ are said to be
  \emph{composable} provided the terminal
  point of the former is the initial point of
  the latter; in this case, the
  $\ell$-concatenation is denoted
  $(w_{1},\varphi_{1})\#_{\ell}(w_{2},\varphi_{2})$
  which inserts a stationary cylinder of
  length $\ell$ in between.
\end{definition}
In other words, $(w,\varphi)\in \mathrm{NN}(Y)$ represents a
non-negative paths of width $w$; we allow
paths of width $0$ which must be stationary
on all of $\R$.
\begin{definition}
  A family of Moore paths parametrized by a
  cube
  $$(w,\varphi):[0,1]^{n-1}\to \mathrm{NN}(Y)$$ is
  said to be \emph{smooth} provided the
  underlying functions
  $w:[0,1]^{n-1}\to [0,\infty)$ and
  $\varphi:[0,1]^{n-1}\to \mathrm{CI}(Y)^{\R,+}$ are
  smooth.
\end{definition}

Now we turn to the formal definition of
$\mathscr{C}(Y)$. First of all, the set of
zero simplices $\mathscr{C}_{0}(Y)$ is simply
the subset of elements of $\mathrm{CI}(Y)$ which
do not lie on the discriminant. For higher
simplices, the definition will depend on the auxiliary choice of the smooth homeomorphism $\tau:[0,\infty)\to [0,\infty)$:
\begin{equation}\label{eq:smooth-homeomorphism-tau}
  \tau(x)=\int_{0}^{x}e^{-1/y}dy
\end{equation}
which:
\begin{itemize}
\item vanishes to infinite order at $x=0$, and
\item is bounded by $\tau(x)\le x$.
\end{itemize}
We then have the following
definition.\footnote{We thank Joseph Helfer
  for suggesting such a model to the
  authors. See \cite[Definition
  4.2]{ganatra-pardon-shende-pub-IHES-2020}
  for similar definition.}
\begin{definition}\label{definition:simplices}
  For $n\ge 1$, an $n$-simplex is a
  collection $\Phi$ of elements
  $\Phi^{v_{0}\dots v_{m}}$, one for each
  collection $0\le v_{0}<\dots<v_{m}\le n$ of
  vertices, such that:
  \begin{itemize}
  \item $\Phi^{v_{0}}\in \mathscr{C}_{0}(Y)$
    for each collection with $m=1$,
  \item
    $\Phi^{v_{0}\dots
      v_{m}}:[0,\infty)^{m-1}\to
    \mathrm{NN}(Y)$ is a smooth cube of
    non-negative Moore paths joining
    $\Phi^{v_{0}}$ and $\Phi^{v_{m}}$,
  \end{itemize}
  These data are required to satisfy the
  following compatibilities:
  \begin{enumerate}
  \item \label{item:break} as $x_{j}\to \infty$,
    \begin{equation*}
      \Phi^{v_{0}\dots v_{m}}(x)=\Phi^{v_{0}\dots v_{j}}(x_{1},\dots,x_{j-1})\#_{\tau(x_{j})} \Phi^{v_{j}\dots v_{m}}(x_{j+1},\dots,x_{m-1}),
    \end{equation*}
  \item \label{item:set-face-zero} when $x_{j}=0$,
    \begin{equation*}
      \Phi^{v_{0}\dots v_{m}}(x)=\Phi^{v_{0}\dots \hat{v_{j}}\dots v_{m}}(x_{1},\dots,x_{j-1},x_{j+1},\dots,x_{m-1}),
    \end{equation*}
   
  \item \label{item:microcollaring} on the
    face $x_{j}=0$,
    $\partial_{x_{j}}\Phi^{v}(x)$ vanishes to
    infinite order.
  \end{enumerate}
  The collection of all $n$-simplices is
  denoted $\mathscr{C}_{n}(Y)$.
\end{definition}

\subsubsection{Definition of the face maps}
\label{sec:defin-face}

In this section, we define the face maps and
verify the structural equations of a
semi-simplicial set are satisfied. In the
next section \S\ref{sec:defin-degen} we
define the degeneracy maps and complete the
proof that $\mathscr{C}(Y)$ is a genuine
simplicial set.

Recall that any $n$-simplex $\Phi$ should
have $n+1$ \emph{faces}
$d_{0}\Phi,d_{1}\Phi,\dots,d_{n}\Phi$ so that
$d_{i}\Phi$ is the ``pullback'' of $\Phi$ by
the simplicial map
$\Delta^{n-1}\to \Delta^{n}$ which omits the
$i$th vertex.

\begin{definition}\label{definition:faces}
  Let $\Phi$ be an $n$-simplex, as in
  Definition \ref{definition:simplices}.

  For $i=0,\dots,n$, the $i$th face is
  $d_{i}\Phi$ is defined by forgetting all
  those $\Phi^{v_{0}\dots v_{m}}$ containing
  $v_{j}=i$ amongst the collection.
\end{definition}
One can now check that
$\mathscr{C}(Y)=(\mathscr{C}_{0}(Y),\mathscr{C}_{1}(Y),\dots)$
with the face maps in \ref{definition:faces}
satisfies the axioms of a
\emph{semi-simplicial set}, namely that the
structural relation holds:
\begin{equation*}
  d_{i}d_{j}=d_{j-1}d_{i}\text{ for }i<j;
\end{equation*}
see \cite[\S1.1.1]{lurie-kerodon}.

\subsubsection{Definition of the degeneracy
  maps}
\label{sec:defin-degen}

A simplicial set is also equipped with
\emph{degeneracies},
$s_{i}:\mathscr{C}_{n}(Y)\to
\mathscr{C}_{n+1}(Y)$, $i=0,\dots,n$; here
$s_{i}$ corresponds to the simplicial map
$\Delta^{n+1}\to \Delta^{n}$ for which the
$i$th vertex has two preimages and all other
vertices have one preimage; see
\cite[\S1.1.2]{lurie-kerodon}.

We find it useful to introduce the following
operation:
\begin{definition}\label{defn:smooth-max}
  Fix $\tau:[0,\infty)\to [0,\infty)$, as in
  \eqref{eq:smooth-homeomorphism-tau}. The
  \emph{smoothed max function} associated to
  $\tau$ is:
  \begin{equation*}
    \mathrm{smax}(x_{1},\dots,x_{k}):=\tau^{-1}(\tau(x_{1})+\dots+\tau(x_{k})).
  \end{equation*}
  Note that this function is not smooth at
  the origin, but it is smooth everywhere
  else.
\end{definition}

\begin{definition}\label{def:degeneracies}
  Let $\Phi$ be an $n$-simplex, as in
  Definition \ref{definition:simplices}, and
  fix $\tau$ and its associated smoothed max
  function as in Definition
  \ref{defn:smooth-max}. Suppose that $\tau$
  vanishes to infinite order at the origin.

  For $i=0,\dots,n$, the $i$th degeneracy
  $s_{i}\Phi$ is an $(n+1)$-simplex defined
  as follows. Let $\sigma_{i}:[n+1]\to [n]$
  be the surjective map where
  $\sigma_{i}(j)=j$ for $j\le i$ and
  $\sigma_{i}(j)=j-1$ for $j>i$.

  For any subset
  $V=\{0\le v_{0}<\dots<v_{m}\le n+1\}$, let
  $U = \sigma_{i}(V) \subset [n]$. The
  subcube $(s_{i}\Phi)^{V}$ is defined in two
  cases:
  
  Case 1: If $V$ does not contain both $i$ and $i+1$,
  then $\sigma_{i}$ is injective on $V$. We define:
  \begin{equation*}
    (s_{i}\Phi)^{V}(x_{1},\dots,x_{m-1}) = \Phi^{U}(x_{1},\dots,x_{m-1}).
  \end{equation*}

  Case 2: If $V$ contains both $i$ and $i+1$,
  then $v_{k}=i$ and $v_{k+1}=i+1$ for some
  $0\le k<m$. The definition depends on the
  position of the doubled vertices:
  \begin{itemize}
  \item If $0 < k < m-1$, the doubled vertices are
    both internal. We define:
    \begin{equation*}
      (s_{i}\Phi)^{V}(x) = \Phi^{U}(x_{1},\dots,x_{k-1},\mathrm{smax}(x_{k},x_{k+1}),x_{k+2},\dots,x_{m-1}).
    \end{equation*}
  \item If $k=0$, the doubled vertices are at
    the beginning of $V$. We define:
    \begin{equation*}
      (s_{i}\Phi)^{V}(x) = \id \#_{\tau(x_{1})} \Phi^{U}(x_{2},\dots,x_{m-1}).
    \end{equation*}
  \item If $k=m-1$, the doubled vertices are
    at the end of $V$. We define:
    \begin{equation*}
      (s_{i}\Phi)^{V}(x) = \Phi^{U}(x_{1},\dots,x_{m-2}) \#_{\tau(x_{m-1})} \id.
    \end{equation*}
  \end{itemize}
  If $m=1$ and $V=\{i,i+1\}$, we interpret
  these formulas as yielding the stationary
  path of width zero based at the $0$-simplex
  $\Phi^{i}$, which acts as the identity
  morphism $\id$.
\end{definition}

\begin{lemma}\label{lemma:degeneracy-well-defined-and-identities}
  For any $n$-simplex $\Phi$ and
  $0\le i\le n$, the collection of subcubes
  $s_{i}\Phi$ from Definition
  \ref{def:degeneracies} is a well-defined
  $(n+1)$-simplex (satisfying the axioms of
  Definition
  \ref{definition:simplices}). Furthermore,
  the face and degeneracy operators satisfy
  the simplicial identities:
  \begin{equation*}
    s_{i}s_{j}=s_{j+1}s_{i}\text{ for all }0\le i\le j\le n,
  \end{equation*}
  and
  \begin{equation*}
    d_{j}s_{i}=\left\{
      \begin{aligned}
        &s_{i-1}d_{j}&&\text{ if }j<i,\\
        &\id &&\text{ if }j=i\text{ or }j=i+1,\\
        &s_{i}d_{j-1}&&\text{ if }j>i+1;
      \end{aligned}
    \right.
  \end{equation*}
  therefore $\mathscr{C}(Y)$ with these face
  and degeneracy operators defines a
  simplicial set.
\end{lemma}
\begin{proof}
  First we show $s_{i}\Phi$ satisfies the
  $(n+1)$-simplex axioms. The proof of axiom
  \ref{item:set-face-zero}, that
  $\bar{\Phi} = s_{i}\Phi$ restricts
  correctly to its subcubes as variables
  $x_{k} \to 0$, is identical to the proof of
  the identities $d_{j}s_{i}$ verified below.

  To verify smoothness and the microcollaring
  axiom \ref{item:microcollaring}, observe
  that:
  $$\mathrm{smax}(x_{k},x_{k+1}) =
  \tau^{-1}(\tau(x_{k})+\tau(x_{k+1}))$$ is
  smooth everywhere except when
  $x_{k}=x_{k+1}=0$. However, because $\Phi$
  satisfies microcollaring, its partial
  derivatives vanish to infinite order as
  variables approach $0$. This infinite-order
  vanishing compensates for the
  non-smoothness of $\tau^{-1}$ at the
  origin, ensuring the pulled-back subcubes
  are smooth.

  The breaking axiom \ref{item:break} as
  $x_{j} \to \infty$ holds by direct
  calculation. For instance, if
  $x_{k} \to \infty$ in a subcube where
  $v_{k}=i$ and $v_{k+1}=i+1$, then
  $\mathrm{smax}(x_{k},x_{k+1}) \to
  \infty$. Using that
  $\tau(\mathrm{smax}(x_{k},x_{k+1})) =
  \tau(x_{k}) + \tau(x_{k+1})$ we have:
  \begin{align*}
    \bar{\Phi}^{V}(x) &= \Phi^{U}(\dots,\mathrm{smax}(x_{k},x_{k+1}),\dots) \\
                      &= \Phi^{U_{-}}(\dots) \#_{\tau(\mathrm{smax}(x_{k},x_{k+1}))} \Phi^{U_{+}}(\dots) \\
                      &= \Phi^{U_{-}}(\dots) \#_{\tau(x_{k})} \id \#_{\tau(x_{k+1})} \Phi^{U_{+}}(\dots),
  \end{align*}
  which matches the formulas for the subcubes
  $\bar{\Phi}^{V_{-}}$ and
  $\bar{\Phi}^{V_{+}}$. The edge cases $k=0$ and
  $k=m-1$ follow by a similar argument.

  We now verify the face relations. For
  $j\ne 0,n+1$, applying $d_{j}$ to
  $s_{i}\Phi$ amounts to fixing the variable
  corresponding to the vertex $j$ to $0$.  If
  $j=i$ or $j=i+1$, we are evaluating a cube
  where the variable between $i$ and $i+1$
  vanishes. Because
  $\mathrm{smax}(x,0) = \mathrm{smax}(0,x) =
  x$, substituting $0$ into
  $\mathrm{smax}(x_{i},x_{i+1})$ leaves the
  remaining coordinate unchanged. For
  example, if we apply $d_{i}$ to the top
  cube by setting $x_{i}=0$, the $i$-th
  variable is omitted and $x_{i+1}$ takes its
  place in the argument list, recovering the
  original top cube of $\Phi$. For the
  boundary cases $i=0$ and $i=n$, the
  internal faces $d_{1}s_{0} = \id$ and
  $d_{n}s_{n} = \id$ are handled similarly by
  setting the corresponding variable to zero,
  since $\tau(0)=0$ means $\#_{\tau(0)}\id$
  acts as the identity. The outermost faces
  $d_{0}s_{0} = \id$ and $d_{n+1}s_{n} = \id$
  follow from the breaking axiom
  \ref{item:break} as the relevant variable
  tends to infinity. Thus $d_{i}s_{i} = \id$
  and $d_{i+1}s_{i} = \id$ hold for all $i$.

  For $j > i+1$, setting the $j$-th variable
  to $0$ omits it, while $x_{i}$ and
  $x_{i+1}$ are unaffected. Writing this out
  for the top cube yields:
  \begin{equation*}
    \Phi(x_{1},\dots,\mathrm{smax}(x_{i},x_{i+1}),\dots,x_{j-1},0,x_{j+1},\dots,x_{n}).
  \end{equation*}
  The $0$ occurs at the $(j-1)$-th position
  in the argument list passed to $\Phi$. By
  Axiom \ref{item:set-face-zero} for $\Phi$,
  this evaluates to the $j-1$ face of $\Phi$
  evaluated with the doubled $i$-th vertex,
  which is exactly $s_{i}(d_{j-1}\Phi)$. An
  identical argument handles $j < i$, proving
  $d_{j}s_{i} = s_{i-1}d_{j}$.

  We now verify the degeneracy relations
  $s_{i}s_{j} = s_{j+1}s_{i}$ for
  $0\le i\le j\le n$. When $0 < i < j < n$,
  applying $s_{j}$ then $s_{i}$ (or
  vice-versa) results in an $(n+2)$-simplex
  containing the non-overlapping terms:
  \begin{equation*}
    \Phi(x_{1},\dots,\mathrm{smax}(x_{i},x_{i+1}),\dots,\mathrm{smax}(x_{j+1},x_{j+2}),\dots,x_{n+1}),
  \end{equation*}
  which is symmetric. When
  $i = j > 0$, both operators create a bracketed
  maximum (of the form $\mathrm{smax}(a,\mathrm{smax}(b,c))$). Since
  $$\tau(\mathrm{smax}(a,b)) = \tau(a) +
  \tau(b),$$ we have:
  \begin{align*}
    \mathrm{smax}(\mathrm{smax}(x_{i},x_{i+1}),x_{i+2}) &= \tau^{-1}(\tau(x_{i})+\tau(x_{i+1})+\tau(x_{i+2})) \\
                                                        &= \mathrm{smax}(x_{i},\mathrm{smax}(x_{i+1},x_{i+2})).
  \end{align*}
  In the boundary case $i=j=0$, evaluating
  $s_{0}s_{0}\Phi$ introduces two prepended
  identity paths. Their concatenated length
  is $\tau(x_{1}) + \tau(x_{2})$, giving:
  \begin{equation*}
    \id \#_{\tau(x_{1}) + \tau(x_{2})} \Phi(x_{3},\dots,x_{n+1}).
  \end{equation*}
  Conversely, $s_{1}s_{0}\Phi$ yields:
  \begin{equation*}
    \id \#_{\tau(\mathrm{smax}(x_{1},x_{2}))} \Phi(x_{3},\dots,x_{n+1}).
  \end{equation*}
  These are equal since
  $\tau(\mathrm{smax}(x_{1},x_{2})) =
  \tau(x_{1}) + \tau(x_{2})$. The remaining
  boundary cases where $j=n$ are handled
  symmetrically. This completes the proof.
\end{proof}

\subsubsection{Verification of the
  $\infty$-category axiom}
\label{sec:verif-infty-categ}

An $\infty$-category is not merely a simplicial
set, but is required to satisfy additional
\emph{horn-filling} properties; see
\cite[Definition 1.1.2.4]{lurie-HTT-2009}. See
also \cite[Notation A.2.7.2]{lurie-HTT-2009} and
\cite[Example A.2.7.3]{lurie-HTT-2009}.

\begin{proposition}\label{theorem:C-is-infty-category}
  The simplicial set $\mathscr{C}(Y)$ is an
  $\infty$-category.
\end{proposition}
\begin{proof}
  Fix $i\ne 0,n$, and fix $n-1$ simplices $\Phi_{0},\dots,\Phi_{i-1},\Phi_{i+1},\dots,\Phi_{n}$ such that:
  \begin{equation*}
    d_{j}\Phi_{k}=d_{k-1}\Phi_{j}
  \end{equation*}
  holds for all $j<k$, $j\ne i$. One needs to
  prove the equation:
  \begin{equation}\label{eq:horn-filling}
    d_{j}\Phi=\Phi_{j}\text{ for }j\ne i
  \end{equation}
  admits a solution $\Phi$. The key idea is
  that $\Phi$ is uniquely determined by its
  main cube $\Phi^{0\dots n}$, and
  \eqref{eq:horn-filling} specifies all faces
  of this main cube except for the $x_{i}=0$
  face. The existence of $\Phi$ then follows
  from the usual ``Serre lifting property,''
  which asserts that data specified on all
  faces of a cube, except for one, can be
  extended to the interior of the cube. The
  collaring conditions \ref{item:break} and
  \ref{item:microcollaring} for $\Phi_{j}$
  ensure that one can pick an extension which
  is smooth, and which satisfies
  \ref{item:break}, \ref{item:set-face-zero},
  and \ref{item:microcollaring}. The details
  of this latter assertion go as follows: let
  $P$ be the union of all boundary faces of a
  cube $Q$ except for one face, and let $p$
  be the center point of this missing
  face. The ``stereographic'' projection map
  based at $p$ maps $Q-p$ continuously onto
  $P$, and in fact defines a continuous
  retraction $R$ of the inclusion
  $P\subset Q$. By precomposing with $R$, we
  can extend the data defined on $P$ to data
  defined on $Q-p$. The crucial observation
  is that, even though $R$ is not smooth at
  the points which are mapped to the faces of
  codimension two or higher, the fact the
  data on $P$ satisfies \ref{item:break} and
  \ref{item:microcollaring} ensures that the
  pulled back data is in fact smooth. It is a
  routine affair to extend from $Q-p$ to all
  $Q$, by flowing by a vector field on $Q$
  which vanishes near $P$ and which maps $Q$
  into $Q-p$.
\end{proof}
\begin{remark}\label{remark:barC-also}
  The same exact argument proves that the larger
  category $\bar{\mathscr{C}}(Y)$, where we do not
  require vertices lie off of the discriminant, is
  an infinity category.
\end{remark}

\subsubsection{Equivalent morphisms}
\label{sec:equivalent-morphisms}

In an $\infty$-category, two $1$-simplices
$\sigma,\sigma'$ are \emph{equivalent} if
there is $2$-simplex $\Phi$ satisfying
$\Phi^{01}=\sigma$,
$\Phi^{02}=\sigma'$, and
$\Phi^{12}=\id$, where $\id$ means a
degenerate $1$-simplex; see
\cite[pp.\,39]{lurie-HTT-2009}.

Consider $\sigma,\sigma'$ as elements
in $\mathrm{NN}(Y)$. In our case, all the
information of $\Phi$ is contained in the top
dimensional cube:
\begin{equation*}
  \Phi^{012}(x_{1})\in \mathrm{NN}(Y),
\end{equation*}
depending on $x_{1}\in [0,\infty)$, satisfying:
\begin{itemize}
\item $\Phi^{012}$ has fixed endpoints,
\item $\Phi^{012}(0)=\Phi^{02}$,
\item $\Phi^{012}(x_{1})=\Phi^{01}\#_{\tau(x_{1})} \id$,
\end{itemize}
for $x_{1}$ large enough.  From this
description, and standard facts about Moore
paths, we see that: {\itshape two
  $1$-simplices $\sigma,\sigma'$ are
  equivalent if and only if they are
  homotopic in the space of non-negative
  paths $\mathrm{NN}(Y)$ with fixed
  endpoints.} It follows that the homotopy
category $\mathrm{h}\mathscr{C}(Y)$ is the
category considered in
\cite[2.2.9]{cant-hedicke-kilgore-arXiv-2023}.

\subsubsection{On the complexity of
  $\mathscr{C}(Y)$}
\label{sec:complexity-mathscrcy}

In many appearances of $\infty$-categories in
Floer theory, the $\infty$-category of data
$\mathscr{C}$ is a contractible Kan complex (this
is the case, in, e.g., the treatment of
Hamiltonian Floer theory in
\cite{pardon-GT-2016}). In such cases, the
$\infty$-category itself is not of significant
independent interest, and its use is solely in
achieving more canonical constructions (this is
not to say the use of $\infty$-categories is
unimportant in these contexts). In other
appearances, such as
\cite{ganatra-pardon-shende-pub-IHES-2020,
  varolgunes-GT-2021} the $\infty$-category of
data $\mathscr{C}$ satisfies the property that the
natural map to the nerve of the homotopy category:
\begin{equation}\label{eq:natural-map}
  \mathscr{C}\to \mathrm{N}(\mathrm{h}\mathscr{C})
\end{equation}
is a trivial Kan fibration (Definition
\ref{definition:trivial-kan-fibration}); see
\cite[\S4.4]{ganatra-pardon-shende-pub-IHES-2020}
and \cite[\S3.2.2]{varolgunes-GT-2021} for a
related discussion. In these cases, the
homotopy category $\mathrm{h}\mathscr{C}$
carries some amount of complexity, but no
additional complexity is hidden in
$\mathscr{C}$. It is notable that in our
case, the category $\mathscr{C}(Y)$ does not
satisfy this property.

\begin{lemma}\label{lemma:natural-not}
  The morphism \eqref{eq:natural-map} is not a
  trivial Kan fibration for
  $\mathscr{C}=\mathscr{C}(Y)$ if
  $\pi_{k}(\mathrm{Cont}(Y))$ is non-trivial for
  at least one number $k\ge 3$.
\end{lemma}
\begin{proof}
  The property we need about the nerve of the
  ordinary category $\mathrm{h}\mathscr{C}$ is
  that any map
  $\bd\Delta^{k}\to\mathrm{N}(\mathrm{h}\mathscr{C})$
  extends to a map
  $\Delta^{k}\to\mathrm{N}(\mathrm{h}\mathscr{C})$,
  if $k\ge 3$.\footnote{Indeed, the information of a
    $3$-simplex in
    $\mathrm{N}(\mathrm{h}\mathscr{C})$ is a
    sequence of three composable morphisms
    $f,g,h$, while the data of
    $\partial
    \Delta^{3}\to\mathrm{N}(\mathrm{h}\mathscr{C})$
    is the four sequences
    \begin{itemize}
    \item $f,g$ on the $\set{0,1,2}$ face,
    \item $f,gh$ on the $\set{0,1,3}$ face,
    \item $fg,h$ on the $\set{0,2,3}$ face,
    \item $g,h$ on the $\set{1,2,3}$ face.
    \end{itemize}
    One can simply take $f,g$ from the
    $\set{0,1,2}$ face and $h$ from the
    $\set{1,2,3}$ face, and prove that the three
    simplex $(f,g,h)$ fills the map
    $\partial \Delta^{3}\to
    \mathrm{N}(\mathrm{h}\mathscr{C})$; the
    verification uses the fact that the faces of
    the constituents of a map
    $\partial \Delta^{3}$ satisfy various
    equations amongst each other.} In particular, if
  \eqref{eq:natural-map} is a trivial Kan
  fibration, then any map
  $\bd\Delta^{k}\to \mathscr{C}(Y)$ extends to a
  map $\Delta^{k}\to \mathscr{C}(Y)$, for
  $k\ge 3$. We will use the
  existence of a non-trivial class
  $f\in \pi_{k}(\mathrm{Cont}(Y))$ to produce a
  boundary $\bd \Delta^{k}\to \mathscr{C}(Y)$
  which cannot be filled, thereby proving the
  lemma.

  It is more convenient to work in
  $\bar{\mathscr{C}}(Y)$ where the identity
  $0$-simplex is allowed; this is without
  loss of generality, since the space of
  Moore paths joining $\id$ to $\id$ is
  naturally identified with the space of
  Moore paths joining $\varphi_{t}$ to
  $\varphi_{t}$, via pointwise composition
  (using the group structure of
  $\mathrm{Cont}(Y)$).
  
  First of all, by standard adjunction
  identifications, we can think of $f$ as a
  $k-2$ cube in the space of Moore paths in
  the space of based loops of
  $\mathrm{Cont}(Y)$. In particular, this
  induces a $k-1$ simplex $\Phi$ in
  $\bar{\mathscr{C}}(Y)$ given by the
  formula:
  \begin{equation*}
    \Phi(x_{1},\dots,x_{k-2})=f(x_{1},\dots,x_{k-2})\#_{\tau(x_{1})+\dots+\tau(x_{k-2})}\id.
  \end{equation*}
  
  Now consider
  $\Gamma(f)=\partial \Delta^{k}\to
  \bar{\mathscr{C}}(Y)$ which agrees with $f$ on the
  $[1\dots k]$ face, and agrees with $\id$ on each
  other face. Then:
  
  \textbf{Claim.} {\itshape If $\Gamma(f)$ can be
    filled to a $k$-simplex in $\bar{\mathscr{C}}(Y)$,
    then $f$ represents the trivial element in
    $\pi_{k}(\mathrm{Cont}(Y))$.}

  If so, then there exists $\Psi$ so that:
  \begin{equation*}
    \Psi(x_{1},\dots,x_{k-1})=\id\#_{\tau(x_{1})} f(x_{2},\dots,x_{k-1})\#_{\tau(x_{2})+\dots+\tau(x_{k-1})}\id
  \end{equation*}
  for $x_{1}$ large enough, while
  $\Psi(0,x_{1},\dots,x_{k-1})=\Psi^{02\dots
    k}(x_{2},\dots,x_{k-1})=\id$, and so
  $\eta\in [0,\infty)\mapsto
  \Psi(\eta,x_{1},\dots,x_{k-1})$ acts as a
  homotopy between $\id$ and $f$. In general
  $\Psi(\eta,x_{1},\dots,x_{k-1})$ is some
  non-negative loop based at $\id$; since
  there do not exist $C^{0}$ small
  non-constant and non-negative loops of
  contactomorphisms,\footnote{The
    non-existence of $C^{0}$ small
    non-negative and non-constant loops of
    contactomorphisms follows from the
    existence of generating functions and
    action selectors for Legendrian isotopies
    in $1$-jet spaces due to
    \cite{chekanov-funkt-analiz-1996} see
    also \cite{colin_sandon}.} the $t=1$
  endpoint of the isotopies must remain fixed
  at $\id$. Thus the homotopy over
  $\eta\in [0,\infty)$ proves that $f$ was
  the zero element in
  $\pi_{k}(\mathrm{Cont}(Y))$.
\end{proof}

\begin{remark}
  In \cite[Theorem 1]{casals-spacil-JTA-2016}
  it is shown that
  $i:U(n)\to \mathrm{Cont}(S^{2n-1})$ is
  injective on homotopy groups. Thus Lemma
  \ref{lemma:natural-not} is not vacuous.
\end{remark}

% 2.2
\subsection{Equivariant Floer cohomology over
  characteristic two}
\label{sec:case-p-2}

We split the proof of Theorem
\ref{theorem:main-infinity} into the two cases
$p=2$ and $p\ge 3$. There are differences between
the cases:
\begin{itemize}
\item in the case $G=\Z/2\Z$, we consider $EG$ as
  $S^{\infty}\subset \R^{\infty}$ with the
  antipodal $\Z/2\Z$-action, and use the shift
  self-symmetry of $\R^{\infty}:$
  $$\tau:(\eta_{0},\eta_{1},\dots)\mapsto(0,\eta_{0},\eta_{1},\dots);$$
\item in the case $G=\Z/p\Z$ with $p\ge 3$, we
  consider $EG$ as $S^{\infty}\subset \C^{\infty}$
  with the $\Z/p\Z$-action inherited from the
  circle action on $\C^{\infty}$, and we use the
  shift self-symmetry of $\C^{\infty}$ instead of
  $\R^{\infty}$. This leads to various differences
  in the recipe for constructing equivariant Floer
  complexes.
\item in the case $p=2$, so $\mathbf{k}=\Z/2\Z$,
  we do not need to discuss orientations in Floer
  theory, whereas the case $p\ge 3$ requires such
  discussion.
\end{itemize}
In the remainder of this subsection
\S\ref{sec:case-p-2} we focus solely on the case
$p=2$.

\subsubsection{Reminder of the Borel equivariant
  cohomology complex}
\label{sec:borel-construction}

Our approach to defining equivariant Floer
cohomology complexes follows
\cite{seidel-smith-GAFA-2010,seidel-eq-pop}. Their
strategy is to emulate the Borel construction
\cite{borel-book-PU-press-1960} in the language of
Morse theory in such a way that it carries over to
Floer theory. We briefly recall the idea in the
context of Morse theory: given a compact manifold
$M$ with a $\Z/2\Z$-action, one considers
pseudogradients on $M\times_{\Z/2\Z} S^{\infty}$
of a specific type (some care is needed due to the
non-compactness of $S^{\infty}$).  Appropriately
counting the flow lines of such a pseudogradient
defines the equivariant Morse differential.

One thinks of a pseudogradient $P$ on
$M\times_{\Z/2\Z} S^{\infty}$ as a
$\Z/2\Z$-equivariant vector field on
$M\times S^{\infty}$. It is convenient to fix a
standard pseudogradient $V$ on $S^{\infty}$, and
then only consider $P$ which are related to $V$
under the projection
$M\times S^{\infty}\to S^{\infty}$. Such vector
fields can be thought of as $S^{\infty}$-families
of vector fields on $M$ satisfying a certain
$\Z/2\Z$-equivariance. In the Floer theoretic
setting, one similarly considers
$S^{\infty}$-families of auxiliary data as the
necessary input to obtain a chain complex.

\subsubsection{Auxiliary category of Floer data}
\label{sec:auxil-categ-floer}

As explained in
\S\ref{sec:survey-infty-categories}, we do not
directly define an infinity functor
$\mathscr{C}(Y)\to
\mathrm{N}_{\mathrm{dg}}(\mathrm{Ch})$, but rather
construct a diagram of the form:
\begin{equation}\label{eq:forgetful-diagram}
  \begin{tikzcd}
    {\mathscr{D}^{*}(W)}
    \arrow[d,"{\mathrm{forget}}",swap]
    \arrow[r,"{\mathrm{CF}_{\mathrm{eq}}}"]
    &{\mathrm{N}_{\mathrm{dg}}(\mathrm{Ch})}\\
    {\mathscr{C}(Y)}
  \end{tikzcd}
\end{equation}
where $\mathscr{D}^{*}(W)$ is a larger
$\infty$-category whose objects consist of all the
necessary data needed to define the Floer complex;
for similar set-up, we refer the reader to
\cite[Eqn.\,7.0.1]{pardon-GT-2016}. We will first
define $\mathscr{D}(W)$ and then define
$\mathscr{D}^{*}(W)\subset \mathscr{D}(W)$ as a
subcategory of ``regular'' simplices. The goal of
the present section is to construct
$\mathscr{D}(W)$ and the forgetful map. In
\S\ref{sec:defin-infty-funct}, we show that the
forgetful map in \eqref{eq:forgetful-diagram} is a
\emph{trivial Kan fibration} (see
\cite[Eqn.\,7.5.3]{pardon-GT-2016}). Precomposing
$\mathrm{CF}_{\mathrm{eq}}$ with a section of this
fibration induces an $\infty$-functor
$\mathscr{C}(Y)\to
\mathrm{N}_{\mathrm{dg}}(\mathrm{Ch})$; this is
the $\infty$-functor satisfying Theorem
\ref{theorem:main-infinity}.

The definition is in terms of \emph{Borel data},
namely pairs $(\psi_{\eta,t},J)$
where:\footnote{We refer the reader to
  \cite[pp.\,1468]{seidel-smith-GAFA-2010} for a
  similar set-up.}
\begin{itemize}
\item $J$ is a $G$-invariant, $\omega$-tame and
  Liouville-equivariant-at-$\infty$ almost complex
  structure on $W$,
\item $\psi_{\eta,t}$ is a family of Hamiltonian
  isotopies of $W$ depending on a parameter
  $\eta\in S^{\infty}$;
\end{itemize}
Borel data is supposed to satisfy the following:
\begin{enumerate}[label=(B\arabic*)]
\item\label{item:B-equiv} (\emph{equivariance})
  $g^{-1}\psi_{g\eta,t}g=\psi_{\eta,t}$,
\item (\emph{shift invariance})
  $\psi_{\tau(\eta),t}=\psi_{\eta,t}$, where
  $\tau:\R^{\infty}\to \R^{\infty}$ shifts by $1$,
\item\label{item:B-polar} (\emph{polar})
  $\psi_{\eta,t}$ is independent of $\eta$ in
  fixed neighborhoods of the poles,
\item\label{item:B-IR} (\emph{ideal restriction})
  the Hamiltonian generator $X_{\eta,t}$ of
  $t\mapsto \psi_{\eta,t}$ is $\eta$-independent
  and Liouville-equivariant outside of a compact
  set.
\end{enumerate}
Consequently there is a contact isotopy arising as
the ideal restriction, denoted
$\varphi_{t}=\mathrm{IR}(\psi_{\eta,t})$.

As in Definition \ref{definition:moore-path},
one can define a \emph{Moore
  paths} of Borel data to be a smooth map:\footnote{The axioms
  \ref{item:B-polar} and \ref{item:B-IR}
  involve open neighborhoods, and these
  neighborhoods are assumed to be uniform
  along any smooth family.}
\begin{equation*}
  s\in \R\mapsto (\psi_{s,\eta,t},J_{s})
\end{equation*}
together with a ``width'' $w\in [0,\infty)$
such that $\psi_{s,\eta,t},J_{s}$ is
$s$-independent on the intervals $[w,\infty)$
and $(-\infty,0]$; such a path is said to be
\emph{non-negative} provided the ideal restriction lies in the space of non-negative paths $\mathrm{NN}(Y)$.
\begin{definition}\label{definition:mathfrakF}
  The space of non-negative Moore paths of
  such Borel data is denoted $\mathfrak{F}(W)$,
  and admits an ideal restriction map to
  $\mathrm{NN}(Y)$.
\end{definition}

Then an $n$-simplex
$\Phi\in \mathscr{D}_{n}(W)$ is defined exactly as in Definition \ref{definition:simplices} except:
\begin{itemize}
\item zero simplices $\mathscr{D}_{0}(W)$ are
  required to have ideal restrictions in
  $\mathscr{C}_{0}(Y)$.
\item $\mathrm{NN}(Y)$ is replaced by $\mathfrak{F}(W)$.
\end{itemize}
The face and degeneracy maps are exactly as
in Definition \ref{definition:faces} and
\ref{def:degeneracies}. The same proof of the
``horn-filling axioms'' given in Theorem
\ref{theorem:C-is-infty-category} works here
to show $\mathscr{D}(W)$ is an
$\infty$-category. Moreover, there is an
obvious ``ideal restriction''
$\infty$-functor
$\mathscr{D}(W)\to \mathscr{C}(Y)$.

Inside of $\mathscr{D}(W)$ we will define a
subcategory $\mathscr{D}^{\ast}(W)$ consisting of
those simplices which achieve \emph{regularity},
i.e., which render a certain (countable)
collection of moduli spaces transverse. We defer
the definition of $\mathscr{D}^{\ast}(W)$ and the
proof that it is an $\infty$-category to the end
of the next section \S\ref{sec:moduli-spaces}.

\subsubsection{Moduli spaces}
\label{sec:moduli-spaces}

First of all, we fix a pseudogradient $V$ on
$\mathbb{R}P^{\infty}$. In projective coordinates,
the time $s$ flow of $V$ is given by:
\begin{equation*} [x_{0}:x_{1}:x_{2}:\dots]\mapsto
  [x_{0}:e^{s}x_{1}:e^{2s}x_{2}:\dots].
\end{equation*}
This pseudogradient on $\mathbb{R}P^{\infty}$ has
one critical point $v_{k}$ of each index $k$, and
we refer to these as \emph{poles}. The lift of $V$
to the double cover $S^{\infty}$ will also be
denoted by $V$, and it has two critical points
$v_{k,\pm}$ of each index, which we also call
poles. Let $\mathscr{P}(v_{k,\pm})$ denote the
space of parametrized flow lines of $-V$ on
$S^{\infty}$ joining $v_{k,\pm}$ to
$v_{0,+}$. These form open manifolds of the
expected dimension, namely $k$, and can be
compactified to manifolds with corners in a
reasonable way; see, e.g.,
\cite[\S3.1]{seidel-eq-pop}.

Each triple $(\Phi,k,\pm)$, where $\Phi$ is an
$n$-simplex in $\mathscr{D}(W)$, $k=0,1,2,\dots$,
and $\pm$ is a sign, determines a moduli space
$\mathscr{M}(\Phi,k,\pm)$ of triples
$(x,\pi,u)$ of the following type:
\begin{itemize}
\item
  $x\in [0,\infty)^{n-1}$,
\item $\pi\in \mathscr{P}(v_{k,\pm})$, and,
\item $u$ solves a suitable version of Floer's
  equation on $\R\times \R/\Z$.
\end{itemize}

For technical use, we also fix a smooth function
$f:\R\to \R$ such that:
\begin{itemize}
\item $f'(t)$ is $1$-periodic and non-negative,
\item $f'(t)$ vanishes in a neighborhood of $0$,
  and,
\item $f(0)=0$ and $f(1)=1$.
\end{itemize}
and a $[0,1]$-valued cut-off function $\rho$
satisfying:
\begin{itemize}
\item $\rho'(t)\le 0$,
\item $\rho'(t)$ is supported in the interior of
  the region where $f=1$,
\item $\rho(0)=1$ and $\rho(1)=0$;
\end{itemize}
\begin{figure}[h]
  \centering
  \begin{tikzpicture}
    \draw[black!50!white] (-1,-1) grid (2,2);
    \draw[line width=.6pt] (-1,-1)
    to[out=0,in=180] +(1,1) to[out=0,in=180]
    +(2,2) to[out=0,in=180] +(3,3) ;
    \begin{scope}[shift={(4,0)}]
      \draw[black!50!white] (-1,-1) rectangle
      (2,2) (-1,0)--(1,0); \draw[line width=.6pt]
      (-1,1)--(0.5,1)to[out=0,in=180](1,0)--(2,0);
    \end{scope}
  \end{tikzpicture}
  \caption{Graph of auxiliary functions
    $f:\R\to \R$ and $\rho:\R\to \R$. We require
    the derivative $\rho'(t)$ to be supported in
    the interior of the interval where $f=1$.}
  \label{fig:graph-of-f}
\end{figure}
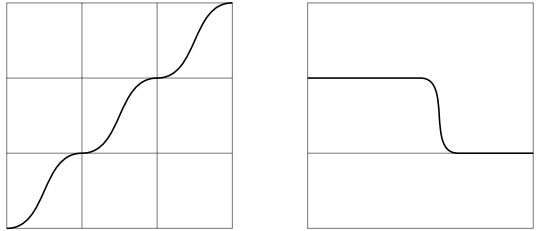
For similar use of such auxiliary functions
$f,\rho$ in setting up the continuation equation,
we refer the reader to
\cite{cant-sh-barcode,cant-hedicke-kilgore-arXiv-2023,cant-arXiv-2024}.\footnote{In
  order for our constructions to be canonical, in
  the sense that any two readers of the paper will
  agree on all constructions, it is necessary that
  the readers pick the same cut-off functions
  $f,\rho$. To make this choice canonical, we
  should give a precise formula for the cut-off
  functions $f$ and $\rho$. Such formulas can be
  found in any advanced calculus text.}

The role of $f$ will be the following: if
$\psi_{\eta,t}$ is a Hamiltonian isotopy,
extended to all times $t\in \R$ by the rule
$\psi_{\eta,t+1}=\psi_{\eta,t}\psi_{\eta,1}$,
then we will consider the
time-reparametrization $\psi_{\eta,f(t)}$.

Let us first consider
$\mathscr{M}(\Phi,k,\pm)$ in the case when
$n=0$. Let $X_{\eta,t}$ be the generator of
$\psi_{\eta,f(t)}$, and consider the
following equation for a pair $(\pi,u)$:
\begin{equation}\label{eq:zero-simplex-eqn-mod-2}
  \left\{
    \begin{aligned}
      &u:\R\times \R/\Z\to W\text{ and }\pi:\R\to S^{\infty},\\
      &\pi'(s)=-V(\pi(s))\text{ with }\pi(-\infty)=v_{k,\pm}\text{ and }\pi(+\infty)=v_{0,+},\\
      &\bd_{s}u+J(u)(\bd_{t}u-X_{\pi(s),t}(u))=0,\\
      &\text{the integral of }\omega(\bd_{s}u,\bd_{t}u-X_{\pi(s),t}(u))\text{ over the cylinder is finite}.
    \end{aligned}
  \right.
\end{equation}
See Figure \ref{fig:zero-simplex-eqn-mod-2} for an
illustration:
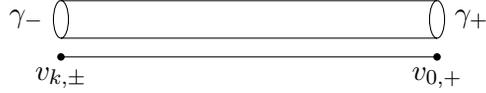
\begin{figure}[h]
  \centering
  \begin{tikzpicture}[scale=.5]
    \draw (0,0) circle (0.2 and 0.5)
    (0,0.5)--+(10,0) (0,-0.5)--+(10,0) (10,0)
    circle (0.2 and 0.5); \draw
    (0,-1)node[fill,circle,inner
    sep=1pt]{}node[below]{$v_{k,\pm}$}--+(10,0)node[fill,circle,inner
    sep=1pt]{}node[below]{$v_{0,+}$}; \node at
    (10.2,0) [right]{$\gamma_{+}$}; \node at
    (-0.2,0) [left]{$\gamma_{-}$};
  \end{tikzpicture}
  \caption{Illustration of a solution to
    \eqref{eq:zero-simplex-eqn-mod-2}. With our
    cohomological conventions, $\gamma_{+}$ will
    be considered as the ``input.''}
  \label{fig:zero-simplex-eqn-mod-2}
\end{figure}

Let us call Borel data \emph{non-degenerate}
provided the time-1 map $\psi_{\eta,1}$ has
only non-degenerate fixed points when $\eta$
is a pole. In this case, it follows from
\ref{item:B-polar} that the solutions $u$ are
asymptotic to stationary solutions over
non-degenerate orbits of $X_{v_{k,\pm},t}$
and $X_{v_{0,+},t}$ at the two ends.

At this stage, we impose our first regularity
assumption for an $n$-simplex $\Phi$:
\begin{enumerate}[label=($\ast\arabic{*}$)]
\item\label{item:ast-non-degenerate} The
  vertices of $\Phi$ are non-degenerate Borel
  data.
\end{enumerate}
This ensures the asymptotics of solutions to
\eqref{eq:zero-simplex-eqn-mod-2} behave well
for all $0$-simplices.

We next define $\mathscr{M}(\phi,k,\pm)$ for
a single non-negative Moore path
$\phi\in \mathfrak{F}(W)$, expressed in the
form $(\psi_{\eta,\sigma,t},J_{\sigma})$ with
$\sigma\in \R$. Write $X_{\eta,s,t}$ and
$Y_{\eta,s,t}$ for the generators of
$\psi_{\eta,-s,f(t)}$ with respect to $t$ and
$s$, and (abusing notation) let us relabel $J_{s}=J_{-s}$.

Then $\mathscr{M}(\phi,k,\pm)$ is defined to be
the solutions $(\pi,u)$ to:
\begin{equation}\label{eq:single-moore-path-eqn-mod-2}
  \left\{
    \begin{aligned}
      &u:\R\times \R/\Z\to W\text{ and }\pi:\R\to S^{\infty},\\
      &\pi'(s)=-V(\pi(s))\text{ with }\pi(-\infty)=v_{k,\pm}\text{ and }\pi(+\infty)=v_{0,+},\\
      &(\bd_{s}u-\rho(t)Y_{\pi(s),s,t}(u))+J_{s}(u)(\bd_{t}u-X_{\pi(s),s,t}(u))=0,\\
      &\text{the integral of }\omega(\bd_{s}u-\rho(t)Y_{\pi(s),s,t}(u),\bd_{t}u-X_{\pi(s),s,t}(u))\text{ is finite}.
    \end{aligned}
  \right.
\end{equation}

Finally, we define
$\mathscr{M}(\Phi^{v_{0}\dots v_{m}},k,\pm)$,
$n\ge 1$, as the moduli space of triples
$(x,\pi,u)$ such that:
\begin{equation}\label{eq:mega-moduli}
  \left\{
    \begin{aligned}
      &x\in [0,\infty)^{m-1},\text{ }\pi:\R\to S^{\infty},\text{ and }u:\R\times \R/\Z\to W,\\
      &(\pi,u)\in \mathscr{M}(\Phi^{v_{0}\dots v_{m}}(x),k,\pm),
    \end{aligned}
  \right.
\end{equation}
where
$\mathscr{M}(\Phi^{v_{0}\dots v_{m}}(x),k,\pm)$ is
defined in
\eqref{eq:single-moore-path-eqn-mod-2}. Let
us note that \ref{item:ast-non-degenerate}
implies that each solution $u$ is asymptotic
to non-degenerate orbits at its two ends.

\begin{remark}\label{remark:energy-estimates}
  As is typical in Floer theory, it is
  necessary to show these moduli spaces have
  certain compactness properties; this boils down
  to establishing a priori estimates on
  certain energy integrals; see, e.g.,
  \cite[\S2.2.6]{cant-arXiv-2024} and
  \cite[\S2.2.5]{cant-hedicke-kilgore-arXiv-2023}
  for some related discussion. We will not
  discuss this aspect of the theory further,
  as it is by now well-established.
\end{remark}

This leads to the second regularity requirement:
\begin{enumerate}[label=($\ast\arabic{*}$),resume]
\item\label{item:ast-transversality} the
  moduli space of continuation cylinders
  $\mathscr{M}(\Phi^{v_{0}\dots
    v_{m}},k,\pm)$ is cut transversally (in
  the parametric sense).
\end{enumerate}
See, e.g., \cite[\S3.2]{seidel-eq-pop} for related
discussion. We define
$\mathscr{D}^{*}(W)\subset \mathscr{D}(W)$:
\begin{definition}\label{definition:regular-n-simplex-D}
  An $n$-simplex $\Phi$ is regular provided
  \ref{item:ast-non-degenerate} holds and
  \ref{item:ast-transversality} holds for all subcubes $\set{v_{0},\dots,v_{m}}\subset \set{0,\dots,n}$. The collection of regular simplices
  is denoted $\mathscr{D}^{*}(W)$
\end{definition}
\begin{proposition}\label{proposition:degeneracy-closed}
  The subset of regular simplices is closed under taking degeneracy maps.
\end{proposition}
\begin{proof}
  Let us consider the $i$th degeneracy $s_{i}\Phi$. Each subcube of $s_{i}\Phi$ is equal to one of the subcubes of $\Phi$, perhaps with two adjacent variables compressed together via $\mathrm{smax}(x_{i},x_{i+1})$, or perhaps with an additional $\id\#_{\tau(x_{1})}(-)$ or $(-)\#_{\tau(x_{n})}\id$; see Definition \ref{def:degeneracies}.

  The equations associated to $\varphi$ and
  $\id\#_{\tau}\varphi$ are naturally
  identified (due to the obvious translation
  symmetry of the set-up). It follows that
  the equations associated to $s_{i}\Phi$ and
  its subcubes are pulled back from analogous
  equations for $\Phi$ along a fibration
  $\pi:[0,\infty)^{k}\to [0,\infty)^{k'},$
  where either:
  \begin{itemize}
  \item $k'=k$ (in which case the fibration
    is an isomorphism) or 
  \item $k'=k+1$ (in which
    case $\pi$ is defined using
    $\mathrm{smax}$, or is a projection map in
    the edge cases $s_{0}$ and $s_{n}$).
  \end{itemize}
  This implies that $s_{i}\Phi$ and its faces
  satisfy \ref{item:ast-transversality} are
  surjective, since the we assume $\Phi$
  satisfies \ref{item:ast-transversality};
  the other item
  \ref{item:ast-non-degenerate} is more
  obvious.
\end{proof}

\begin{proposition}\label{proposition:Dstar-is-infinity}
  The subset of regular simplices is a
  sub-$\infty$-category.
\end{proposition}
\begin{proof}
  This follows from standard results asserting that
  ``transversality is generic;'' indeed, the horn
  filling problem can be solved in
  $\mathscr{D}(W)$, and simplices can be perturbed
  so as to lie in $\mathscr{D}^{*}(W)$ (the
  perturbation can be made relative the horn if
  all of the faces in the horn are already in
  $\mathscr{D}^{*}(W)$).
  
  We comment on how exactly we achieve
  transversality for the moduli space
  $\mathscr{M}(\Phi,k,\pm)$. It is a
  parametric moduli space over
  $[0,\infty)^{n-1}$, and solutions ``break''
  into solutions of the lower parametric
  families as one coordinate $x_{i}$
  approaches $\infty$ or $0$. One
  ``recursively'' achieves transversality (in
  a manner similar to
  \cite{seidel-book-2008}). We only perturb
  in a compactly supported way in $W$ to
  avoid influencing the ideal
  restriction. Since the perturbation is
  allowed to depend on the location in the
  parameter space $[0,\infty)^{n-1}$, the
  problem boils down to one in
  \cite[pp.\,1468]{seidel-smith-GAFA-2010}
  and \cite[\S4]{seidel-eq-pop} where
  transversality for families over the flow
  line spaces $\mathscr{P}(v_{k,\pm})$ is
  considered, (one also appeals to standard
  techniques such as those in
  \cite[\S8]{mcduff-salamon-book-2012}).
\end{proof}

\subsubsection{Definition of the equivariant chain
  complex}
\label{sec:defin-equiv-chain}

Let $\Phi=(\psi_{\eta,t},J)$ be a single Borel datum in
$\mathscr{D}_{0}^{*}(W)$. Following
\cite{seidel-smith-GAFA-2010,seidel-eq-pop}, we
define:
\begin{equation*}
  \mathrm{CF}_{\mathrm{eq}}(\psi_{\eta,t},J):=\mathrm{CF}(\psi_{t})\otimes \mathbf{k}[[x]],
\end{equation*}
where $\psi_{t}=\psi_{v_{0,+},t}$ and where
$\mathrm{CF}(\psi_{t})$ is the $\mathbf{k}$-vector
space generated by contractible $1$-periodic
orbits of $\psi_{t}$. The differential is defined
as a power series:
\begin{equation*}
  d_{eq}=\sum x^{k}d_{k}=\sum x^{k}(d_{k,+}+d_{k,-})
\end{equation*}
where $d_{k,\pm}$ counts solutions in the family
$\mathscr{M}(\Phi,k,\pm)$, as follows:
\begin{equation*}
  \begin{aligned}
    d_{k,+}(\gamma_{+})&:=\sum \set{u(-\infty):(\pi,u)\in \mathscr{M}_{1}(\Phi,k,+)/\R\text{ and }u(+\infty)=\gamma_{+}},\\
    d_{k,-}(\gamma_{+})&:=\sum \set{g(u(-\infty)):(\pi,u)\in \mathscr{M}_{1}(\Phi,k,-)/\R\text{ and }u(+\infty)=\gamma_{+}}.
  \end{aligned}
\end{equation*}
Note that the $d_{k,-}$ term twists the orbit by
$g$, which is necessary for the count to make
sense. Note as well that $d_{0,+}$ is the usual
Floer differential and $d_{0,-}=0$.

The fact that $d_{eq}^{2}=0$ is essentially an
argument in family Floer cohomology, similar to
some discussion in \cite[Lemma
42]{brocic-cant-arXiv-2025}, and is explained in
\cite{seidel-smith-GAFA-2010} and
\cite{seidel-eq-pop}. We briefly review the
argument:

\begin{proof}[Proof that the equivariant
  differential squares to zero]
  As is usual in Floer theory, we prove
  $d_{eq}^{2}=0$ by considering the non-compact
  ends of the 1-dimensional moduli spaces
  $\mathscr{M}_{2}(\Phi,k,\pm)/\R$. Here an
  \emph{end} is a proper embedding of $[0,\infty)$
  into the moduli space, understood via the
  unbounded sequences $(\ell_{n},\pi_{n},u_{n})$
  which remain in the end. Of course, in this case
  $\Phi$ is a single choice of Borel data so
  $\ell_{n}\in \mathscr{M}(\Delta^{n})$ is
  entirely irrelevant. We consider two kinds of
  ends $(\pi_{n},u_{n})$:
  \begin{enumerate}[label=(\roman*)]
  \item\label{item:equivariant-diff-1}
    $[\pi_{n}]\in \mathscr{P}(v_{k,\pm})/\R$
    converges (but $u_{n}$ does not converge),
  \item\label{item:equivariant-diff-2} $[\pi_{n}]$
    breaks into a limit in
    $\mathscr{P}(v_{l_{1},\pm})/\R\times
    \mathscr{P}(v_{l_{2},\pm})/\R$ where
    $l_{1}+l_{2}=k$ and both $l_{i}>0$.
  \end{enumerate}
  As in
  \cite{seidel-smith-GAFA-2010,seidel-eq-pop} the
  fact that the pseudogradient is invariant under
  the self-similarity map and the involution
  implies that:
  \begin{equation*}
    \set{\text{flows lines from }v_{l_{1},\epsilon_{1}}\text{ to }v_{l_{2},\epsilon_{2}}}\simeq \mathscr{P}(v_{l_{2}-l_{1},\epsilon_{1}\epsilon_{2}}),
  \end{equation*}
  and it is with respect to this identification
  that only spaces of flow lines starting at
  $v_{0,+}$ appear in the formula
  \ref{item:equivariant-diff-2}.

  The linearization of the parametric equation is
  a linear map:
  \begin{equation*}
    T\mathscr{P}(v_{k,\pm})\times W^{1,p}(u_{n}^{*}TW)\to L^{p}(u_{n}^{*}TW),
  \end{equation*}
  and the restriction to $W^{1,p}$ is a
  Cauchy-Riemann operator $D_{u_{n}}$ whose index
  depends on the asymptotic orbits
  $\gamma_{-},\gamma_{+}$ and the homology class
  of $[u]$ relative $\gamma_{-},\gamma_{+}$. We
  record the formula for its index:
  \begin{equation}\label{eq:index-formula}
    \mathrm{index}(D_{u_{n}})=2c_{1}^{\mathfrak{s}}(u)+\mathrm{CZ}_{\mathfrak{s}}(\gamma_{+})-\mathrm{CZ}_{\mathfrak{s}}(\gamma_{-}),
  \end{equation}
  where:
  \begin{itemize}
  \item $\mathfrak{s}$ is any generic section of
    $\mathrm{\det}_{\C}(TW)$ non-vanishing along
    $\gamma_{\pm}$,
  \item $c_{1}^{\mathfrak{s}}(u)$ is the signed
    intersection number
    $\mathrm{PD}(\mathfrak{s}^{-1}(0))\cap [u]$,
  \item the Conley-Zehnder indices
    $\mathrm{CZ}_{\mathfrak{s}}(\gamma_{\pm})$ are
    computed using trivializations of
    $TW|_{\gamma_{\pm}}$ which render
    $\mathfrak{s}$ constant;
  \end{itemize}
  see, e.g., \cite{salamon-notes-1997}; we also
  refer the reader to \cite{cant-thesis-2022} and
  the references therein for further details on
  this formula. In particular, along any sequence
  in $\mathscr{M}_{2}(\Phi,k,\pm)$ it holds that
  $\mathrm{index}(D_{u_{n}})=2-k$. In case
  \ref{item:equivariant-diff-1}, one picks
  representatives:
  $$(\pi_{n},u_{n})\in
  \mathscr{M}_{2}(\Phi,k,\pm)/\R$$ so that
  $\pi_{n}$ converges (after passing to a
  subsequence). Then $u_{n}$ converges on compact
  subsets by standard Floer theory to a solution
  $u_{\infty}$ in $\mathscr{M}_{d}(\Phi,k,\pm)$
  with $d\le 2$. By \ref{item:ast-transversality}
  it holds that:
  \begin{equation}\label{eq:most-are-empty}
    \dots=\mathscr{M}_{-1}(\Phi,k,\pm)=\mathscr{M}_{0}(\Phi,k,\pm)=0,
  \end{equation}
  if $k\ne 0$, while if $k=0$ then
  $\mathscr{M}_{0}(\Phi,0,+)$ consists of only the
  stationary cylinders. Note that if $k=0$, then
  $\pi_{n}$ is irrelevant, and the argument boils
  down to Floer's original argument and shows
  $d_{0}^{2}=0$. Assume $k>0$ for the rest of the
  proof.

  If $u_{n}$ does not converge uniformly, then it
  must break off non-stationary Floer differential
  cylinders (solutions to the $s$-independent
  equation) at one of the ends. Each Floer differential
  cylinder must have a positive index
  \eqref{eq:index-formula}, by
  \ref{item:ast-non-degenerate} and
  \ref{item:ast-transversality}. It then follows
  from \eqref{eq:most-are-empty} that only one
  non-stationary Floer cylinder of index $1$
  breaks off, and the other limit $u_{\infty}$ is
  in $\mathscr{M}_{1}(\Phi,k,\pm)$. For each
  $k\ge 1$, the count of these breakings
  (interpreted in the usual Floer theoretic sense)
  gives:
  \begin{equation*}
    d_{0}d_{k}+d_{k}d_{0}.
  \end{equation*}
  Standard Floer theory gluing proves each
  hypothetical breaking in the fiber product
  inside of
  $\mathscr{M}_{1}/\R\times \mathscr{M}_{1}/\R$
  actually does appear as the end of the moduli
  space $\mathscr{M}_{2}/\R$. In general we will
  suppress discussion of gluing.
  
  Turning now to the other non-compact ends
  \ref{item:equivariant-diff-2}. These breakings
  only occur when $k\ge 2$. In this case, if
  $[\pi_{n}]$ breaks into $[\pi_{-}],[\pi_{+}]$,
  we claim that $[\pi_{n},u_{n}]$ must converge
  (after a subsequence) into a broken
  configuration $[\pi_{-},u_{-}]$ and
  $[\pi_{+},u_{+}]$ in
  $\mathscr{M}_{1}/\R\times
  \mathscr{M}_{1}/\R$. This is the only
  possibility compatible with
  \ref{item:equivariant-diff-2} and the index
  formula \eqref{eq:index-formula}. There is some
  small technicality regarding how the breakings
  are interpreted, when the self-similarity and
  involution maps are taken into account; see
  \cite[pp.\,972]{seidel-eq-pop} for related
  discussion. It suffices to say that this
  technicality explains that the breakings of type
  \ref{item:equivariant-diff-2} correspond to the
  sum of compositions $d_{l_{1}}d_{l_{2}}$ where
  $l_{1}+l_{2}=k$ and $l_{1},l_{2}>0$, extending
  the case $d_{k}d_{0}+d_{0}d_{k}$ to the
  intermediate values. In total, one shows:
  \begin{equation*}
    \sum_{l_{1}+l_{2}=k} d_{l_{1}}d_{l_{2}}=0,\text{ for each $k$}
  \end{equation*}
  which is equivalent to $d_{eq}^{2}=0$.
\end{proof}

\subsubsection{Chain homotopies associated to
  higher simplices}
\label{sec:chain-homot-assoc}

We continue from
\S\ref{sec:infin-categ-eliashb-intro}. Recall that
a map
$\mathscr{D}^{*}(W)\to
\mathrm{N}_{\mathrm{dg}}(\mathrm{Ch}(\mathbf{k}[[x]]))$
is the data of:
\begin{enumerate}[label=(\alph*)]
\item\label{item:chain-homot-stage-1} a
  finitely generated graded\footnote{We need
    to use $\Z/2\Z$-gradings in certain signs
    when the characteristic of the field is
    $p\ge 3$; in the case $p=2$ we can use
    ungraded modules.}
  $\mathbf{k}[[x]]$-module $V_{\Phi}$ with
  differential $d_{\Phi}$ of degree $1$ for
  each zero simplex $\Phi$,
\item\label{item:chain-homot-stage-2} a map
  $\mathfrak{c}_{\Phi}:V_{\Phi|0}\to V_{\Phi|n}$
  of degree $1-n$ for each $n$-simplex $\Phi$.
\end{enumerate}
In the present case where
$\mathrm{char}(\mathbf{k})=2$, these are supposed
to satisfy:
\begin{equation}\label{eq:dg-nerve-repeat}
  \sum_{j=1}^{n-1}\mathfrak{c}_{\Phi|[j\dots n]}\circ \mathfrak{c}_{\Phi|[0\dots j]}+\mathfrak{c}_{\Phi|[0\dots \hat{j}\dots n]}=\mathfrak{c}_{\Phi}\circ d_{\Phi|0}+d_{\Phi|n}\circ \mathfrak{c}_{\Phi},
\end{equation}
and, concerning degeneracies,
\begin{equation}\label{eq:degenerate-nerve-repeat}
  \mathfrak{c}_{s_{i}\Phi}=\left\{
    \begin{aligned}
      &0&&\text{ for each $n$-simplex $\Phi$ with $n\ge 1$},\\
      &\id&&\text{ for each $0$-simplex $\Phi$}.
    \end{aligned}
    \right.
\end{equation}

The first stage \ref{item:chain-homot-stage-1} is
done: each zero simplex $\Phi=(\psi_{\eta,t},J)$
is sent to its equivariant Floer cohomology
complex
$V_{\Phi}=\mathrm{CF}_{\mathrm{eq}}(\psi_{\eta,t},J)$
with differential $d_{\Phi}=d_{eq}$ as defined in
\S\ref{sec:defin-equiv-chain}.

The second stage \ref{item:chain-homot-stage-2} is
satisfied by the following definition:
\begin{equation*}
  \mathfrak{c}_{\Phi}=\sum x^{k}\mathfrak{c}_{k}=\sum x^{k}(\mathfrak{c}_{k,+}+\mathfrak{c}_{k,-})
\end{equation*}
where $\mathfrak{c}_{k,\pm}$ are defined by:
\begin{equation*}
  \begin{aligned}
    \mathfrak{c}_{k,+}(\gamma_{+})&=\sum \set{u(-\infty):(x,\pi,u)\in \mathscr{M}_{0}(\Phi^{0\dots n},k,+)\text{ with }u(+\infty)=\gamma_{+}},\\
    \mathfrak{c}_{k,-}(\gamma_{+})&=\sum \set{g(u(-\infty)):(x,\pi,u)\in \mathscr{M}^{0\dots n}_{0}(\Phi,k,-)\text{ with }u(+\infty)=\gamma_{+}};
  \end{aligned}
\end{equation*}
here the ``$0$'' subscript in
$\mathscr{M}_{0}$ signifies the dimension of
the moduli space. It remains only to verify
the equations \eqref{eq:dg-nerve-repeat} and
\eqref{eq:degenerate-nerve-repeat}. To simplify the notation, let us denote by: $\mathscr{M}_{d}(\Phi)=\mathscr{M}_{d}(\Phi^{0\dots n},k,\pm)$.

The equation
\eqref{eq:dg-nerve-repeat} is proved by analyzing the
$1$-dimensional component
$\mathscr{M}_{1}(\Phi)$. The moduli
space admits a map:
$$x:\mathscr{M}_{1}(\Phi)\to
[0,\infty)^{n-1},$$ which records the
underlying cubical coordinates; this is used
to assign algebraic interpretations to the
non-compact ends:
\begin{enumerate}[label=(\roman*)]
\item\label{item:type-of-end-1}
  $\mathfrak{c}_{\Phi|[j\dots n]}\circ
  \mathfrak{c}_{\Phi|[0\dots j]}$: ends where
  one of the cubical coordinates $x_{j}$ diverges to $\infty$,
\item\label{item:type-of-end-2}
  $\mathfrak{c}_{\Phi|[0\dots\hat{j}\dots n]}$:
  ends where one of the cubical coordinates $x_{j}$ converges to $0$,
\item\label{item:type-of-end-3}
  $\mathfrak{c}_{\Phi}\circ d_{\Phi|0}$ and
  $d|_{\Phi|n}\circ c_{\Phi}$: ends where the
  cubical coordinate $x$ converges in the
  interior $(0,\infty)^{n-1}$.
\end{enumerate}
In the remainder of this section we make this
informal sketch more precise.

\begin{proof}[Proof of equation \eqref{eq:dg-nerve-repeat}]
  The moduli space $\mathscr{M}(\Phi,k,\pm)$
  defined in \eqref{eq:mega-moduli} admits a
  ``forgetful'' map to
  $\mathscr{P}(v_{k,\pm})\times [0,\infty)^{n-1}$.

  Referring to the Fredholm index in
  \eqref{eq:index-formula}, it holds that:
  \begin{equation}\label{eq:index-along-end}
    \mathrm{index}(u)+k+n-1=d\text{ for }(x,\pi,u)\in \mathscr{M}_{d}(\Phi,k,\pm),
  \end{equation}
  As in \ref{item:type-of-end-1} through
  \ref{item:type-of-end-3}, let us study the
  parameter map
  $x:\mathscr{M}_{1}\to [0,\infty)^{n-1}$. We begin
  with:
  \begin{claim}
    Along any end
    $(x_{\nu},\pi_{\nu},u_{\nu})\in
    \mathscr{M}_{1}(\Phi,k,\pm)$ such that
    the $x_{\nu,i}$ converges to $0$ but the
    other coordinates converge in the
    interior $(0,\infty)$, and $\pi_{n}$
    converges, some subsequence will converge
    to a solution
    $(x_{\infty},\pi_{\infty},u_{\infty})$
    for the restriction
    $\mathscr{M}_{0}(\Phi|_{[0\dots
      \hat{i}\dots n]},k,\pm)$.
  \end{claim}
  \begin{proof}[Proof of claim]
    If follows from \ref{item:set-face-zero}
    that the PDE which $u_{n}$ solves
    converges to the one for the restricted
    simplex; and then, by standard
    compactness results, $u_{n}$ also
    converges (after passing to a
    subsequence) in the stated sense,
    yielding the claim. The dimension drops
    by $1$ by inspecting
    \eqref{eq:index-along-end}.
  \end{proof}

  These ends lead to the terms of the form
  $\mathfrak{c}_{\Phi|[0\dots\hat{i}\dots
    n]}$ appearing in
  \eqref{eq:dg-nerve-repeat}. All such
  contributions do appear as ends of
  one-dimensional moduli spaces by standard
  Floer theoretic ``gluing'' (here we do not
  actually glue together two solutions, but
  one still appeals to the inverse function
  theorem; see
  \cite[\S4]{brocic-cant-JFPTA-2024} for
  related discussion).
  
  The same argument used for the claim works
  if $x_{i}\to 0$ holds for
  multiple values of $i$, (i.e., all
  coordinates converge in $[0,\infty)$ and
  $\pi_{n}$ converges). Then the index will
  drop by the number of coordinates which
  converged to $0$. For generic data, we can
  therefore assume that only one coordinate
  ever converges to $0$, provided the other
  coordinates converge in $[0,\infty)$ and
  $\pi_{n}$ converges.

  Let us now analyze the case when one of the
  coordinates, say $x_{i}$, diverges to $\infty$, but
  the other coordinates converge in $[0,\infty)$.

  Then the property \ref{item:break}, stating that:
  \begin{equation*}
    \Phi^{0\dots n}(x)=\Phi^{0\dots i}(x_{1},\dots,x_{i-1})\#_{\tau(x_{i})}\Phi^{i\dots n}(x_{i+1},\dots,x_{n-1})
  \end{equation*}
  enables us to extract two limits. The first
  $u_{+}$ is just the limit of the original
  sequence on compact subsets. The second
  $u_{-}$ is the limit of retranslations by
  $\tau(x_{i})$. Moreover, using the same
  parametrizations, $\pi$ also converges to
  some pair of flow lines $\pi_{-},\pi_{+}$
  of indices $k_{-},k_{+}$ on $S^{\infty}$
  (after passing to a subsequence). By
  compactness, the other cubical coordinates
  converge and we write
  $x_{-}=(x_{1},\dots,x_{i-1})$ and
  $x_{+}=(x_{i+1},\dots,x_{n})$ for these
  limits. We claim:
  \begin{itemize}
  \item
    $(u_{-},x_{-},\pi_{-})\in
    \mathscr{M}_{d_{-}}(\Phi|_{[i\dots
      n]},k_{-},\pm)$,
  \item
    $(u_{+},x_{+},\pi_{+})\in
    \mathscr{M}_{d_{+}}(\Phi|_{[0\dots
      i]},k_{+},\pm)$,
  \end{itemize}
  for appropriate $\pm$ signs modulo applying the
  involution and using the shift self-symmetry of
  $S^{\infty}$, as in
  \cite{seidel-smith-GAFA-2010,seidel-eq-pop}. Here
  $d_{+}+d_{-}\le 0$. The only possibility is
  $d_{+}=d_{-}=0$, and $k_{-}+k_{+}=k$ (rather
  than $k_{-}+k_{+}\le k$). This implies that
  $\pi_{n}$ actually converges in the generalized
  Morse sense to the configuration
  $\pi_{-},\pi_{+}$ we already found.

  Therefore, the pair
  $(u_{-},x_{-},\pi_{-}),(u_{+},x_{+},\pi_{+})$
  contributes to the term:
  \begin{equation*}
    \mathfrak{c}_{\Phi|[i\dots n]}\mathfrak{c}_{\Phi|[0\dots i]};
  \end{equation*}
  all such contributions actually do appear as
  ends of one-dimensional moduli spaces by
  standard Floer theoretic gluing for parametric
  families of continuation cylinders, similarly to
  \cite[pp.\,972]{seidel-eq-pop}.

  Essentially the same argument implies that
  one cannot have two coordinates which
  diverge to infinity, without contradicting
  regularity.

  All of the other ends are of the form where
  $x$ converges in $(0,\infty)^{n-1}$. If
  $\pi_{n}$ converges to $(\pi_{-},\pi_{+})$ with
  positive indices $k_{-},k_{+}$ in the Morse
  theoretic sense, then one concludes the terms of
  the form:
  \begin{equation*}
    \mathfrak{c}_{k_{-}}d_{k_{+}}+d_{k_{-}}\mathfrak{c}_{k_{+}}.
  \end{equation*}
  In all other cases permitted by regularity,
  $\pi_{n}$ converges in
  $\mathscr{P}(v_{k,\pm})/\R$, and one concludes
  the terms of the form
  $\mathfrak{c}_{0}d_{k}+d_{k}\mathfrak{c}_{0}$. Summing
  together all of these possible ends gives the
  desired relation \eqref{eq:dg-nerve-repeat}.
\end{proof}

\begin{proof}[Proof of the degeneracy equation \eqref{eq:degenerate-nerve-repeat}]
  The case of a degenerate $1$-simplex is
  just the well-known fact that the
  continuation map associated to stationary
  continuation data is equal to the identity
  map on chain level.
  
  For degenerate $n$-simplices with $n>1$,
  the argument is essentially the same as the
  one given in the proof of Proposition
  \ref{proposition:degeneracy-closed}. The
  top cube of $s_{i}\Phi$ is pulled back from
  the top cube of $\Phi$ by a submersion:
  $$\pi:[0,\infty)^{n}\to
  [0,\infty)^{n-1},$$ defined by an explicit
  formula. The moduli spaces associated to
  different points in the same fiber of $\pi$
  are canonically identified, and the
  linearizations at solutions above
  $x\in [0,\infty)^{n}$ are surjective if the
  linearizations of above $\pi(x)$ are
  surjective. In particular, there is an
  empty zero dimensional component of the
  moduli space above $[0,\infty)^{n}$
  associated to $s_{i}\Phi$, and so it holds
  that $\mathfrak{c}_{s_{i}\Phi}=0$, as desired.
\end{proof}

\subsubsection{Definition of the $\infty$-functor}
\label{sec:defin-infty-funct}

Most of the arguments needed in this subsection
were already explained in
\S\ref{sec:survey-infty-categories}. The main
technical result of this subsection is:
\begin{lemma}\label{lemma:is-a-trivial-kan-fibration}
  The forgetful map
  $\mathscr{D}^{*}(W)\to \mathscr{C}(Y)$ is a
  trivial Kan fibration.
\end{lemma}
\begin{proof}
  One shows that
  $\mathscr{D}(W)\to \mathscr{C}(Y)$ is a trivial
  Kan fibration by direct construction. Using a
  $G$-invariant radial function $r$, one can
  cut-off the generating Hamiltonians functions
  from the symplectization end to the filling and
  explicitly construct extensions; e.g., one can
  use a formula of the form $f(r)H$ where $f$
  vanishes outside the symplectization
  end. Extensions constructed in this way will not
  be regular, but they can be perturbed on the
  compact part to be made regular.
\end{proof}

Therefore, by Proposition
\ref{proposition:trivial-kan-fibration}, there is
a contractible Kan complex of sections
$\Sigma:\mathscr{C}(Y)\to \mathscr{D}^{*}(W)$ of
this forgetful map. We define:
\begin{equation*}
  \mathrm{CF}_{\mathrm{eq}}\circ \Sigma:\mathscr{C}(Y)\to \mathrm{N}_{\mathrm{dg}}\mathrm{Ch}(\mathbf{k}[[x]]),
\end{equation*}
as the $\infty$-functor satisfying Theorem
\ref{theorem:main-infinity}. While this depends on
a choice of section, this choice is
``contractible'' in the $\infty$-category sense,
and any two choices give equivalent functors (in
an appropriate sense).

It is perhaps reassuring to observe the following:
for any
$(\psi_{\eta,t},J)\in \mathscr{D}_{0}^{*}(W)$,
there is a section $\Sigma$ such that
$\Sigma(\varphi_{t})=(\psi_{\eta,t},J)$, where
$\varphi_{t}$ is the ideal restriction of
$\psi_{\eta,t}$. Indeed, this follows from the
lifting property for the diagram:
\begin{equation*}
  \begin{tikzcd}
    {\Delta^{0}}\arrow[d,"{}"]\arrow[r,"{}"] &{\mathscr{D}^{*}(W)}\arrow[d,"{}"]\\
    {\mathscr{C}(Y)}\arrow[ru,dashed,"{\Sigma}"]\arrow[r,"{\id}"]
    &{\mathscr{C}(Y).}
  \end{tikzcd}
\end{equation*}
This means that, given any desired regular
extension of $\varphi_{t}$ to Borel data
$(\psi_{\eta,t},J)$, one can assume that the
output of
$\mathrm{CF}_{\mathrm{eq}}\circ
\Sigma(\varphi_{t})$ agrees with the equivariant
Floer complex
$\mathrm{CF}_{\mathrm{eq}}(\psi_{\eta,t},J)$ from
\S\ref{sec:defin-equiv-chain} ``on the nose.'' In
any case, the chain level constructions we will
appeal to are insensitive to the precise choice of
$\Sigma$.

% 2.3
\subsection{The case when $p\ge 3$}
\label{sec:case-p-3}

For the rest of this section we fix a prime
\( p\ge 3 \) and let \( G = \mathbf{k}=\Z/p\Z \).

The main contents of the construction are similar
to the \( \Z/2\Z \)-case. The principal difference
lies in the underlying Morse model for \( BG. \)
As noted at the start of \S\ref{sec:case-p-2}, in
the case $p\ge 3$, we work on
\( S^{\infty}\subset \C^{\infty} \) with the
diagonal \( G \subset U(1) \) action, and use the
shift operator \( \tau \) associated to the
standard \( \C \)-valued coordinates. The usual
perfect Morse function on \( \C P^{\infty} \)
pulls back to a \( U(1) \) invariant Morse--Bott
function on \( S^{\infty} \) given by:
\begin{equation}\label{eq:morse-bott-function}
  f(z_0,z_1,\ldots) = \sum_{n = 0}^{\infty} n\left|z_n\right|^{2}.
\end{equation}
Well chosen \( G \)-equivariant perturbations turn
this into a Morse function on \( BG \) (see
\cite[\S 4]{shelukhin-zhao-JSG-2021} for
discussion of this construction, indeed we are
happy to work with the same model). Moreover, this
may be done compatibly with the action of
\( \tau \).

Ultimately we are concerned with the
pseudogradient as in \S\ref{sec:moduli-spaces}. We
take the pseudogradient $B$ on
$\mathbb{C}P^{\infty}$ whose time-$s$ flow is
given by:
\begin{equation*} [z_{0}:z_{1}:z_{2}:\dots]\mapsto
  [z_{0}:e^{s}z_{1}:e^{2s}z_{2}:\dots].
\end{equation*}
Fixing a connection on the line bundle
$S^{\infty}\to \mathbb{C}P^{\infty}$, we obtain a
unique lift to a vector field $B$ on $S^{\infty}$
taking values in the horizontal distribution of the
connection; this vector field is Morse-Bott with
Lyapunov function given by
\eqref{eq:morse-bott-function}. One ``Morsifies''
$B$ by adding a vector field $F$ tangent to the
fibers of the map
$S^{\infty}\to \mathbb{C}P^{\infty}$. One
considers $V=B+F$ as the $p>2$ analogue of the
vector field $V$ considered in
\S\ref{sec:moduli-spaces}. For convenience, one
can assume that the connection is flat in a
neighborhood of the poles of
$\mathbb{C}P^{\infty}$.

One picks $F$ and the connection so that:
\begin{itemize}
\item the restriction of $F$ to the fiber above a critical point of $B$ is as in Figure \ref{fig:mod-p-equivariant};
\item $F$ and the connection are $\Z/p\Z$ equivariant;
\item $F$ and the connection is invariant under the shift map:
  $$\tau:(z_{0},z_{1},z_{2},\dots)\mapsto
  (0,z_{0},z_{1},\dots).$$
\item the \emph{total vector field} $V=B+F$ is
  Morse-Smale;
\end{itemize}
We leave the construction of such $F$ and such
connection to the reader, and refer to
\cite{shelukhin-zhao-JSG-2021} for further
discussion.

\begin{figure}[h]
  \centering
  \begin{tikzpicture}
    \foreach \x in {0,120,240} {
      \draw[postaction={decorate,decoration={markings,mark=at
          position 0.5 with {\arrow{>}}}}, thick]
      (\x:1) arc ({\x}:{\x+60}:1);
      \draw[postaction={decorate,decoration={markings,mark=at
          position 0.5 with {\arrow{>}}}}, thick]
      (\x:1) arc ({\x}:{\x-60}:1);
      \node[circle,fill,inner sep=1pt] at (\x:1)
      {}; } \foreach \x in {0,120,240} {
      \node[circle,draw,fill=white,inner sep=1pt]
      at (\x+60:1) {}; }
    \path (0:1)node[right]{$v_{0,e}$}
    (120:1)node[above left]{$v_{0,g}$}
    (240:1)node[below left]{$v_{0,g^{2}}$};
    \path (60:1)node[above right]{$v_{1,e}$}
    (180:1)node[left]{$v_{1,g}$} (300:1)node[below
    right]{$v_{1,g^{2}}$};
  \end{tikzpicture}
  \caption{Restriction of $F$ to the critical
    circle has $2p$ critical points (shown for
    $p=3$). Here $v_{1,e}$ is required to be a
    positive pole.}
  \label{fig:mod-p-equivariant}
\end{figure}

This \( V \) then has exactly \( p \) zeros of
index \( n \) for each \( n \geq 0, \) which are
related to each other by the \( G \) action, and
\( \tau \) acts by shifting the index by
2. Choosing distinguished critical points
\( v_{0,e}, v_{1,e}, \) of index 0 and 1
respectively, we have a natural labeling of the
rest \( v_{n,g} \) by the data of their index
\( n, \) and the element \( g \in G \) which sends
\( v_{n,g} \) to the appropriate shift of
\( v_{0,e} \) or \( v_{1,e}. \) Coherent
orientations of the finite dimensional spaces of
flow lines completes the Morse setup; we describe
this in greater detail in the sequel. To keep
things shorter we fix such a model on $BG$ once
and for all as input to the rest of the
construction.\footnote{Unlike the case of
  $B\Z/2\Z=\mathbb{R}P^{\infty}$, the Morse-Smale
  gradient on $BG$ is not canonically defined, as
  we have written it. This is a minor loss of
  canonicity in our construction, and it seems the
  optimal resolution would be to encode the
  remaining ambiguity as auxiliary data to be
  included in the category $\mathscr{D}(W)$. For
  reasons of space, we leave the details of this
  optimization to the reader.}

Having fixed a Morse model on $BG$, the category
of auxiliary data $\mathscr{D}(W)$, and moduli
spaces are essentially as in \S\ref{sec:case-p-2},
with the following modifications:
\begin{itemize}
\item In addition to conditions \ref{item:B-equiv}
  through \ref{item:B-IR} above, auxiliary data
  must be identical near \( v_{0,e},v_{1,e}. \)
\item Associated to a \( k \)-simplex of Borel
  data \( \Phi \) we have moduli spaces:
  \[ \mathscr{M}(\Phi,k,g,i) \] where
  \( i \in \{0,1\} \) identifies the input point
  (otherwise as in \S\ref{sec:moduli-spaces}).
\end{itemize}

Fixing a Borel datum \( \Phi=(\psi_{\eta,t},J), \)
the equivariant chain complex now has underlying
\( \mathbf{k}[[x]] \) module:
\[ \mathrm{CF}_{eq}(\Phi) := \mathrm{CF}(\psi_{t})
  \otimes_{\mathbf{k}} \left(\mathbf{k}[[x]]
    \oplus \mathbf{k}[[x]][1]\right).  \] Here
\( \mathrm{CF}(\psi_{t})\) is the direct sum of
the \( \mathbf{k} \)-orientation lines
$\mathfrak{o}(\gamma)$ of contractible 1-periodic
orbits $\gamma$ of
\( \psi_{t}=\psi_{(1,0,0,\dots),t} \). This module
carries a differential, given by:
\begin{align}\label{eqn:p3-differential}
  d_{eq}(\gamma\otimes 1) &= \sum_{k \geq 0} x^{k}d^{2k}_{0}(\gamma) \otimes 1 + x^{k}d^{2k+1}_{0}(\gamma)\otimes \theta\\
  d_{eq}(\gamma \otimes \theta) &= \sum_{k \geq 0} x^{k}d^{2k+1}_{1}(\gamma)\otimes \theta + x^{2k+2}d^{2k+2}_{1}(\gamma)\otimes 1,
\end{align}
where $1,\theta$ denote the generators of
$\mathbf{k}[[x]] \oplus \mathbf{k}[[x]][1]$, and:
\[ d^{k}_{i}(\gamma) = \sum_{g \in G}\sum_{u\in
    \mathscr{M}(\Phi,k,g,i)} \mathfrak{o}(u). \]
The inner sum is over all
$u\in \mathscr{M}(\Phi,k,g,i)$ which are rigid up
to translation and which are asymptotic to
$\gamma$ at their input end. The symbol
$\mathfrak{o}(u)$ can be considered in two ways:
\begin{enumerate}
\item $\mathfrak{o}(u)=\pm \gamma'$, where
  $\gamma'$ is the output asymptotic of $u$, and
  the sign $\pm$ is determined by comparing the
  orientation lines of $\mathscr{M}(\psi,k,g,i)$
  with a system of coherent orientations, as in
  \cite{floer-hofer-math-z-1993};
\item
  $\mathfrak{o}(u)\in
  \mathrm{Hom}(\mathfrak{o}(\gamma),\mathfrak{o}(\gamma'))$
  is a canonical isomorphism constructed by the
  linear gluing theory of
  \cite{floer-hofer-math-z-1993} (see, e.g.,
  \cite{abouzaid-EMS-2015,pardon-GT-2016}).
\end{enumerate}

That \( d_{eq} \) squares to 0 follows from an
analysis of the moduli spaces identical to that of
\ref{sec:defin-equiv-chain} above, replacing the
\( \Z/2 \) counting of that section with the usual
orientation line formalism (see
\ref{sec:asymp-ops-orient-lines}).

Chain homotopies from higher simplices, and the
definition of the \( \infty \)-functor are also
built as above, mutatis mutandis.

In the next section, we digress and explain the
set-up of orientations; the reader who is
comfortable with the framework of
\cite[pp.\,290]{abouzaid-EMS-2015}, and/or
\cite[\S C.13.2]{pardon-GT-2016} can safely skip
to \S\ref{sec:equivariant-pss-map}.

% 2.4
\subsection{Gradings, asymptotic operators, and
  orientation lines}
\label{sec:asymp-ops-orient-lines}

The essential ingredient in the construction is
\cite[Theorem 10]{floer-hofer-math-z-1993} and
their kernel gluing lemma. We will also appeal to
\cite[\S C.13.1]{pardon-GT-2016} for orientations
on ``spaces of flow lines on $\Delta^{n}$,'' and
\cite[\S A]{shelukhin-zhao-JSG-2021} for
discussion of the orientations on the spaces of
Morse flow lines on $BG$, for $G=\Z/p\Z$.

\subsubsection{Asymptotic operators}
\label{sec:asymptotic-operators}

An \emph{asymptotic operator} is a first order
differential operator $\mathfrak{A}$, acting on
$\R^{2n}$-valued functions on $\R/\Z$, of the
form:
\begin{equation*}
  \mathfrak{A}=-J\bd_{t}-S(t)
\end{equation*}
where $t\in \R/\Z\mapsto S(t)$ is a smooth loop of
symmetric matrices and $J$ is the standard complex
structure on $\R^{2n}$. Such operators are
elliptic and self-adjoint. If $\mathfrak{A}$
induces an isomorphism
$C^{\infty}(\R/\Z,\R^{2n})\to
C^{\infty}(\R/\Z,\R^{2n})$ then we say
$\mathfrak{A}$ is \emph{non-degenerate}.

\subsubsection{Cauchy-Riemann operators with
  asymptotics}
\label{sec:cauchy-riem-oper}

Given non-degenerate asymptotic operators
$\mathfrak{A}_{-},\mathfrak{A}_{+}$, we denote by
$C(\mathfrak{A}_{-},\mathfrak{A}_{+})$ the space
of operators on the cylinder $\R\times \R/\Z$ of
the form:
\begin{equation*}
  \bd_{s}+J\bd_{t}+S(s,t).
\end{equation*}
such that $-J\bd_{t}-S(s,t)$ converges to
$\mathfrak{A}_{\pm}$ as $s\to\pm\infty$. The
convergence should be in a suitable topology whose
details are not so relevant to the present
discussion.

Then $C(\mathfrak{A}_{-},\mathfrak{A}_{+})$ is a
convex space of Fredholm operators $D$, see
\cite{salamon-notes-1997}, all of which have the
same Fredholm index
$i(\mathfrak{A}_{-},\mathfrak{A}_{+})$. Moreover
there is a continuous real-line bundle:
\begin{equation*}
  \mathrm{det}(\mathfrak{A}_{-},\mathfrak{A}_{+})\to C(\mathfrak{A}_{-},\mathfrak{A}_{+})
\end{equation*}
whose restriction to any stratum with fixed kernel
dimension agrees with:
\begin{equation*}
  \det(\ker D\oplus (\coker D)^{\vee}),
\end{equation*}
The fact that these bundles patch together across
the different strata is a fundamental property of
Fredholm operators, and is discussed in
\cite{mcduff-salamon-book-2012}.

\subsubsection{Orientation lines}
\label{sec:orientation-lines}

The real line bundle
$\det(\mathfrak{A}_{-},\mathfrak{A}_{+})$ admits
two orientations $O,\bar{O}$ (since the base space
is contractible). Let us denote by
$\mathfrak{o}(\mathfrak{A}_{-},\mathfrak{A}_{+})$
the \emph{orientation line} generated by
$O,\bar{O}$ modulo the relation that
$O+\bar{O}=0$. This object
$\mathfrak{o}(\mathfrak{A}_{-},\mathfrak{A}_{+})$
is a free $\Z$-module of rank $1$.

\subsubsection{Automatic isomorphism}
\label{sec:autom-isom}

In \cite{salamon-notes-1997} it is shown that
$\bd_{s}-\mathfrak{A}$ is an isomorphism for any
non-degenerate asymptotic operator
$\mathfrak{A}$. Thus
$\mathfrak{o}(\mathfrak{A},\mathfrak{A})$ is
\emph{canonically} identified with $\mathbf{k}$,
because the determinant line at
$D=\bd_{s}-\mathfrak{A}$ is the determinant line
of zero, which is canonically oriented.

\subsubsection{Kernel gluing}
\label{sec:kernel-gluing}

The kernel gluing theorem of
\cite{floer-hofer-math-z-1993} gives isomorphisms:
\begin{itemize}
\item
  $\mathfrak{o}(\mathfrak{A}_{0},\mathfrak{A}_{1})
  \otimes\mathfrak{o}(\mathfrak{A}_{1},\mathfrak{A}_{2})
  \to\mathfrak{o}(\mathfrak{A}_{0},\mathfrak{A}_{2})$
\end{itemize}
which are associative, and such that
``multiplication'' by the canonical positive
generator of
$\mathfrak{o}(\mathfrak{A},\mathfrak{A})$ acts by
the identity.

Moreover, the pairing:
\begin{equation*}
  \mathfrak{o}(\mathfrak{A}_{0},\mathfrak{A}_{1})
  \otimes \mathfrak{o}(\mathfrak{A}_{1},\mathfrak{A}_{0})
  \to \mathfrak{o}(\mathfrak{A}_{0},\mathfrak{A}_{0})
  \simeq \Z
\end{equation*}
canonically identifies
$\mathfrak{o}(\mathfrak{A}_{0},\mathfrak{A}_{1})\to
\mathfrak{o}(\mathfrak{A}_{1},\mathfrak{A}_{0})$,
as follows: one sends a generator of
$\mathfrak{o}(\mathfrak{A}_{0},\mathfrak{A}_{1})$
to the unique generator of
$\mathfrak{o}(\mathfrak{A}_{1},\mathfrak{A}_{0})$
which pairs with it to produce $1$. Associativity
proves the composition of the canonical
identifications
$\mathfrak{o}(\mathfrak{A}_{0},\mathfrak{A}_{1})\to
\mathfrak{o}(\mathfrak{A}_{1},\mathfrak{A}_{0})\to
\mathfrak{o}(\mathfrak{A}_{0},\mathfrak{A}_{1})$
is the identity map.

\subsubsection{Twisting identifications}
\label{sec:twist-ident}

Given $\mathfrak{A}$ and a loop
$\varphi:\R/\Z\to \mathrm{U}(n)$ consider the
conjugation:
\begin{equation*}
  \varphi_{*}\mathfrak{A}:=\varphi\mathfrak{A}\varphi^{-1}.
\end{equation*}
Multiplication of kernel and cokernel elements by
$\varphi$ yields an identification:
\begin{equation}\label{eq:twist-ident}
  \mathfrak{o}(\mathfrak{A}_{0},\mathfrak{A}_{1})
  \to \mathfrak{o}(\varphi_{*}\mathfrak{A}_{0},\varphi_{*}\mathfrak{A}_{1}).
\end{equation}
Unpacking the construction from
\cite{floer-hofer-math-z-1993} shows
\eqref{eq:twist-ident} is compatible with the
kernel gluing maps from \S\ref{sec:kernel-gluing},
in the sense that the diagram commutes:
\begin{equation*}
  \begin{tikzcd}
    \mathfrak{o}(\mathfrak{A}_{0},\mathfrak{A}_{1})\otimes
    \mathfrak{o}(\mathfrak{A}_{1},\mathfrak{A}_{2})\arrow[r]\arrow[d]
    &\mathfrak{o}(\varphi_{*}\mathfrak{A}_{0},\varphi_{*}\mathfrak{A}_{1})
    \otimes \mathfrak{o}(\varphi_{*}\mathfrak{A}_{1},\varphi_{*}\mathfrak{A}_{2})\arrow[d]\\
    \mathfrak{o}(\mathfrak{A}_{0},\mathfrak{A}_{2})\arrow[r]
    &\mathfrak{o}(\varphi_{*}\mathfrak{A}_{0},\varphi_{*}\mathfrak{A}_{2})
  \end{tikzcd}
\end{equation*}
where the vertical maps are the kernel gluing
maps, and the horizontal maps are the twisting
identifications.

The twisting identification between
$\mathfrak{o}(\mathfrak{A},\mathfrak{A})$ and
$\mathfrak{o}(\varphi_{*}\mathfrak{A},\varphi_{*}\mathfrak{A})$
preserves the canonical generators from
\S\ref{sec:autom-isom}, since the empty basis is
conjugated to the empty basis.

\subsubsection{Parallel transport}
\label{sec:parall-transp}

Suppose that
$\mathfrak{A}_{0,s},\mathfrak{A}_{1,s}$,
$s\in[0,1]$, are families of asymptotic operators
which remain non-degenerate. Then there is a
\emph{parallel transport map}:
\begin{equation}\label{eq:parallel-transport-map}
  \mathfrak{o}(\mathfrak{A}_{0,0},\mathfrak{A}_{1,0})\to   \mathfrak{o}(\mathfrak{A}_{0,1},\mathfrak{A}_{1,1}),
\end{equation}
defined by considering any continuous family
$D_{s}\in
C(\mathfrak{A}_{0,s},\mathfrak{A}_{1,s})$ and the
associated determinant line bundle over this
family.

\subsubsection{Reference asymptotic operator}
\label{sec:reference-operator}

Let $\mathfrak{R}=-J\bd_{t}-\delta$ for
$\delta\in(0,2\pi)$. This acts as a reference
operator.

It will be important for us to consider the
\emph{complex linear} orientation:
\begin{equation*}
  \mathfrak{o}(\varphi_{*}\mathfrak{R},\mathfrak{R})\to \Z,
\end{equation*}
The space of complex linear $D$ is an affine
subspace, so the resulting orientation of
$\mathrm{det}(\varphi_{*}\mathfrak{R},\mathfrak{R})$
is independent of the choice of complex linear
$D$.

Consequently, we obtain canonical isomorphisms:
\begin{equation}\label{eq:twisting-functor}
  \mathfrak{o}(\mathfrak{A},\mathfrak{R})
  \to \mathfrak{o}(\varphi_{*}\mathfrak{A},\varphi_{*}\mathfrak{R})
  \to \mathfrak{o}(\varphi_{*}\mathfrak{A},\mathfrak{R}),
\end{equation}
where in the latter step we multiply by the
generator of
$\mathfrak{o}(\varphi_{*}\mathfrak{R},\mathfrak{R})$
corresponding to the complex linear orientation.

\subsubsection{Linearization procedure}
\label{sec:line-proc}

Given a non-degenerate orbit $\gamma$ of a
Hamiltonian system in a symplectic manifold $W$:
\begin{equation*}
  \gamma(t)-X_{t}(\gamma(t))=0.
\end{equation*}
Let us consider travelling symplectic coordinates
$e_{t}:B(\epsilon)\to W$, where $B(\epsilon)$ is
the standard symplectic ball of capacity
$\epsilon$, such that $e_{t}(0)=\gamma(t)$. Then
nearby loops can be expressed as $e_{t}(\eta(t))$,
and the non-linear functional is represented by:
\begin{equation*}
  \eta'(t)+\d e_{t}(\eta(t))^{-1}\left[\partial_{t}e_{t}(\eta(t))-X_{t}(e_{t}(\eta(t))\right].
\end{equation*}
The linearization of this is equal to
$J\mathfrak{A}_{e}$ where $\mathfrak{A}_{e}$ is an
asymptotic operator.

Changing $e$ amounts to conjugating
$J\mathfrak{A}_{e}$ by a linear symplectic map
$\Psi(t)$,
\begin{equation*}
  J^{-1}\Psi(t)J\mathfrak{A}_{e'}\Psi(t)^{-1}=\mathfrak{A}_{e}.
\end{equation*}
Because $\mathrm{Sp}(2n)$ deformation retracts to
$\mathrm{U}(n)$, one obtains a canonical parallel
transport isomorphism:
\begin{equation}\label{eq:pt1}
  \mathfrak{o}(\mathfrak{R},\varphi_{*}\mathfrak{A}_{e'})\to \mathfrak{o}(\mathfrak{R},\mathfrak{A}_{e}).
\end{equation}
This can be precomposed with the morphism
constructed in \S\ref{sec:reference-operator}:
\begin{equation}\label{eq:pt2}
  \mathfrak{o}(\mathfrak{R},\mathfrak{A}_{e'})\to \mathfrak{o}(\mathfrak{R},\varphi_{*}\mathfrak{A}_{e'}).
\end{equation}

\begin{claim}
  The assignment
  $e\mapsto
  \mathfrak{o}(\mathfrak{R},\mathfrak{A}_{e})$,
  with the morphisms obtained by composing
  \eqref{eq:pt2} and \eqref{eq:pt1} defines a
  functor from the indiscrete groupoid spanned by
  choices of symplectic
  coordinates.\hfill$\square$
\end{claim}

\begin{definition}
  The \emph{orientation line}
  $\mathfrak{o}(\gamma)$ is the limit of
  $e\mapsto
  \mathfrak{o}(\mathfrak{R},\mathfrak{A}_{e})$.
\end{definition}

A similar linearization procedure can be done at
any solution of the Floer-type PDEs considered in
this text (without varying parameters), and one
obtains canonical identifications:
\begin{equation*}
  \mathfrak{o}(D_{u,e})\to \mathfrak{o}(\mathfrak{A}_{-,e},\mathfrak{A}_{+,e})\to \mathfrak{o}(\gamma_{-})\otimes \mathfrak{o}(\gamma_{+}),
\end{equation*}
where $e$ is a travelling family of symplectic
coordinates defined along the domain of $u$ which
restrict to families of coordinates at the
asymptotic ends $\gamma_{-},\gamma_{+}$ (in
general, one tensors the orientation lines of all
asymptotics).

\begin{claim}
  The identifications
  $\mathfrak{o}(D_{u,e})\to
  \mathfrak{o}(\gamma_{-})\otimes
  \mathfrak{o}(\gamma_{+})$ are compatible with
  the conjugation isomorphisms arising from
  changing $e$.\hfill$\square$
\end{claim}

\subsubsection{Note on parametric moduli spaces}
\label{sec:note-param-moduli}

Recall that in the construction of the equivariant
cohomology infinity-functor, we count solutions
$(x,\pi,u)$ of parametric moduli spaces where:
\begin{itemize}
\item $x\in [0,\infty)^{n-1}$,
\item $\pi$ is a flow line on $S^{\infty}$ for a
  $G$- and $\tau$-invariant pseudogradient,
\item $u$ solves Floer's continuation cylinder
  equation for data which depends on $\ell$ and
  $\pi$;
\end{itemize}
see \eqref{eq:mega-moduli} from
\S\ref{sec:moduli-spaces} for the precise details.

Let us denote by $D_{x,\pi,u}$ the
linearization associated to the parametric
problem. Parallel transport from $D_{x,\pi,u}$
to $D_{u}$ in the space of Fredholm operators
produces an identification :
\begin{equation*}
  \mathfrak{o}(D_{x,\pi,u})\to \mathfrak{o}(T_{x})\otimes \mathfrak{o}(T_{\pi})\otimes \mathfrak{o}(D_{u})=\mathfrak{o}(T_{x})\otimes \mathfrak{o}(T_{\pi})\otimes \mathfrak{o}(\gamma_{-})\otimes \mathfrak{o}(\gamma_{+}),
\end{equation*}
where $T_{x},T_{\pi}$ are the
(finite-dimensional) tangent spaces to:
\begin{itemize}
\item the cube $[0,\infty)^{n-1}$ at $x$, and,
\item the space of flow lines on $BG$ with the
  same asymptotics as $\pi$.
\end{itemize}
We comment here that \cite[\S
C.13.1]{pardon-GT-2016} considers orientations on
$T_{x}$, in a slightly different context, while
\cite[\S A]{shelukhin-zhao-JSG-2021} considers
orientations on $T_{\pi}$.

\subsubsection{The definition of the
  infinity-functor}
\label{sec:defin-infin-funct}

Fix the Morse model on $BG$ as described in
\S\ref{sec:case-p-3}. Given a zero-simplex
$(\varphi_{\eta,t},J)$ in the $\infty$-category of
regular Borel data $\mathscr{D}(W)$, we define:
\begin{equation*}
  \mathrm{CF}_{\mathrm{eq}}(\varphi_{\eta,t},J)=\bigoplus \mathfrak{o}(\gamma)\otimes\mathfrak{o}(y)\otimes \mathbf{k}[[x]],
\end{equation*}
where the direct sum is over pairs $\gamma, y$
where:
\begin{itemize}
\item $y$ is a critical point in the fiber of $BG$
  over the minimum on $\mathbb{C}P^{\infty}$;
\item $\gamma$ is an orbit of the system
  $H_{y,t}=H_{t}$ over the distinguished lift of
  $y$; see Figure \ref{fig:mod-p-equivariant} for
  details on distinguished lifts $v_{0,e}$ and
  $v_{1,e}$.
\end{itemize}
Here $\mathfrak{o}(y)$ is the orientation line for
the unstable manifold of $y$. Standard Morse
theory and the self-similarity transformations
give isomorphisms:
\begin{equation*}
  \mathfrak{o}(T_{\pi})\to \mathfrak{o}(y_{-})\otimes \mathfrak{o}(x^{k}y_{+})\simeq \mathfrak{o}(y_{-})\otimes \mathfrak{o}(y_{+})\otimes \mathfrak{o}(x^{k}),
\end{equation*}
where we abuse notation and let $x^{k}$ denote the
critical point of index $2k$ on the base space
$\mathbb{C}P^{\infty}$, and let $x^{k}y$ denote
the ``shift'' of $y$ by the self-similarity map
$k$ times.

The count of rigid-up-to-translation trajectories
can be considered as:
\begin{equation*}
  d_{eq}:\mathrm{CF}_{\mathrm{eq}}(\varphi_{\eta,t},J,\mathbf{k})\to   \mathrm{CF}_{\mathrm{eq}}(\varphi_{\eta,t},J,\mathbf{k})
\end{equation*}
provided we pick orientations for
$\mathfrak{o}(x^{k})$ (and provided we follow the
overall ``twisting asymptotics by the group
action'' scheme of
\cite{seidel-smith-GAFA-2010,seidel-eq-pop,shelukhin-zhao-JSG-2021},
similarly to in \S\ref{sec:case-p-2}).

\begin{claim}
  If one picks the complex linear orientations of
  $\mathfrak{o}(x^{k})$, then it holds that
  $d_{eq}^{2}=0$.
\end{claim}
\begin{proof}
  This is a simplified version of what is shown in
  \cite[\S A]{shelukhin-zhao-JSG-2021}.
\end{proof}

This complex produces the first half of an
infinity-functor. The second half concerns the
morphisms associated to $n$-simplices $\Phi$ for
$n\ge 1$. Counting the rigid elements
$(x,\pi,u)$ produces a map:
\begin{equation*}
  \mathfrak{c}_{\Phi}:\mathrm{CF}_{\mathrm{eq}}(\Phi|_{0})\to   \mathrm{CF}_{\mathrm{eq}}(\Phi|_{n}),
\end{equation*}
provided one also picks orientations of the
$n-1$ cube $[0,\infty)^{n-1}$. Then:
\begin{claim}\label{claim:infty-holds}
  If one orients the $j$-dimensional cubes in the
  standard way, so that $\bd_{1},\dots,\bd_{j}$ is
  an oriented basis, for $j<n$, it holds that:
  \begin{equation*}
    \sum_{j=1}^{n-1}(-1)^{j}(\mathfrak{c}_{\Phi|[j\dots n]}\circ \mathfrak{c}_{\Phi|[0\dots j]}-\mathfrak{c}_{\Phi|[0\dots\hat{j} \dots n]})=\mathfrak{c}_{\Phi}\circ d_{\Phi|0}+(-1)^{n}d_{\Phi|n}\circ \mathfrak{c}_{\Phi}.
  \end{equation*}
\end{claim}
\begin{proof}
  This follows from the boundary
  classification and orientation properties
  established by Pardon in \cite[\S
  C]{pardon-GT-2016}.
\end{proof}

This completes our description of the
infinity-functor in the case $p\ge 3$.

% 3

\section{PSS isomorphism and localization in Morse
  theory}
\label{sec:equivariant-pss-map}

In this section we prove Theorem
\ref{theorem:pss}, and part
\ref{theorem:main-infinity-2} of Theorem
\ref{theorem:main-infinity}.

Let \( W \) be an aspherical \( G \) filling of a
closed contact manifold $Y$, for some prime cyclic group
\( G \) of order \( p \). Throughout we use
the field $\mathbf{k}=\Z/p\Z$, which we generally
suppress from the notation.

The usual PSS map in symplectic cohomology, as
constructed in
\cite{frauenfelder-schlenk-israeljm-2007}, yields
an isomorphism:
\begin{equation*}
  \mathrm{HM}^{*}(W) \cong \mathrm{HF}(R_{\epsilon t})
\end{equation*}
where the left hand side is the Morse cohomology
of \( W \), and the right hand side is the Floer
cohomology of any linear-at-infinity Hamiltonian
whose slope $\epsilon$ with respect to a fixed
Reeb flow $R$ on \( \bd W \) is sufficiently small
and positive. We'll implement this in our
equivariant setup, and lift the construction to
the chain level statement stated in Theorem
\ref{theorem:pss}.

The strategy to prove Theorem \ref{theorem:pss} is
four-fold:
\begin{enumerate}
\item define the Morse equivariant cohomology
  complex $\mathrm{CM}_{\mathrm{eq}}(P_{\eta})$
  for suitable ``Borel families'' of
  pseudogradients $P_{\eta}$ on $W$; see
  \S\ref{sec:equivariant-morse-complex};
\item Prove a localization result in
  \S\ref{sec:loc-in-eq-morse}, where fixed
  points of the $G$-action in $W$ are shown
  to imply the equivariant cohomology
  $\mathrm{HM}_{\mathrm{eq}}(P_{\eta})$ is
  not $x$-torsion;
\item construct the $\infty$-categories
  $\mathscr{P}(Y),\mathscr{N}(Y)$ and the $\infty$-functors
  appearing in the statements of Theorems
  \ref{theorem:pss} and \ref{theorem:negative-pss}; see
  \S\ref{sec:equiv-pss-constr};
\item prove the resulting $\mathrm{PSS}$
  chain maps, are quasi-isomorphisms,
  provided the Floer data is a small Reeb
  flow; this is achieved by constructing
  chain-homotopy inverses to the
  $\mathrm{PSS}$ maps, as in
  \cite{frauenfelder-schlenk-israeljm-2007};
  see \S\ref{sec:low-slope-pss}.
\end{enumerate}

\subsection{The equivariant Morse complex}
\label{sec:equivariant-morse-complex}

In this section, we are concerned with defining
the $G$-equivariant Morse complex
$\mathrm{CM}_{\mathrm{eq}}(\pm P_{\eta})$ and
developing its properties. As above, $G=\Z/p\Z$,
and we work over a base field $\mathbf{k}=\Z/p\Z$.

\subsubsection{Morse--Borel data}
\label{sec:morse-borel-data}

The auxiliary data is analogous
\S\ref{sec:auxil-categ-floer}.
\begin{definition}
  A \emph{Morse--Borel datum} is a family of
  vector fields \( P_{\eta} \) on $W$
  parameterized by $\eta\in EG=S_{\infty}$ such
  that:
  \begin{enumerate}[label=(M\arabic*)]
  \item\label{item:M-equivariant}
    $g_{*}P_{\eta}=P_{g\eta}$, for all $g\in G$,
  \item\label{item:M-liouville} $P_{\eta}$
    agrees with $aZ$, where $Z$ is the
    Liouville vector field and $a\in \R$, outside a compact
    set,\footnote{We use the Liouville flow
      here only because it is a canonical
      outwards pointing vector field; any
      other outwards pointing vector field
      would suffice.}
  \item\label{item:M-self-similarity}
    $P_{\tau(\eta)}=P_{\eta}$, where $\tau$ is the
    self-similarity map,
  \item\label{item:M-near-poles} $P_{\eta}$ is
    independent of $\eta$ on neighborhoods of the
    critical points of the model $G$-invariant
    Morse pseudogradient $V$ of $S^{\infty}=EG$.
  \item\label{item:M-morse} $P_{\eta}$ is a Morse
    pseudogradient, whenever $\eta\in EG$ is a
    zero of $V$.
  \end{enumerate}
\end{definition}
Here the model $G$-invariant Morse pseudogradient
on $EG$ is as described in \S\ref{sec:case-p-2} in
the case $p=2$, or in \S\ref{sec:case-p-3} in the
case $p\ge 3$.

\begin{remark}
  Note there are two types of Morse--Borel
  data, those which point outwards
  ($P_{\eta}=aZ$ for $a>0$ in the end) versus those
  which point inwards ($P_{\eta}=aZ$ with $a<0$ in the
  end). We say the former are of \emph{type
    $+$}, while the latter are of \emph{type
    $-$}.
\end{remark}

\subsubsection{Spaces of flow lines}
\label{sec:spaces-flow-lines}

Associated to such Morse--Borel data
\( P_{\eta} \) one can form the moduli spaces
$\mathscr{M}(P_{\eta})$ of flowlines
\( (q,\pi): \R \rightarrow W \times S^{\infty} \)
satisfying:
\begin{equation*}
  \left\{
    \begin{aligned}
      &q'(s) = -P_{\pi(s)}(q(s)),&&\lim_{s \rightarrow \pm \infty} q(s) \text{ exists,}\\
      &\pi'(s) = -V(\pi(s)).      
    \end{aligned}
  \right.
\end{equation*}
Any such flow line has well-defined asymptotics
$(q_{\pm},\eta_{\pm})$.

\begin{definition}
  If the moduli spaces $\mathscr{M}(P_{\eta})$ are smooth
  manifolds of the expected dimension given by:
  \begin{equation*}
    \mathrm{index}(q_{-})-\mathrm{index}(q_{+})+\mathrm{index}(\eta_{-})-\mathrm{index}(\eta_{+}),
  \end{equation*}
  then we say that $P_{\eta}$ is \emph{regular}
  Morse--Borel data.
\end{definition}

The existence and genericity of regular data
follow from the usual Morse-theoretic
transversality arguments.

\subsubsection{Definition of the Borel-equivariant
  Morse complex}
\label{sec:defin-borel-equiv}

The equivariant Morse complex associated to
Morse--Borel data $P_{\eta}$ is defined to be the
direct sum:
\begin{itemize}
\item $\mathrm{CM}_{\mathrm{eq}}(P_{\eta}):=\bigoplus
  \mathfrak{o}(q)\otimes \mathfrak{o}(y)\otimes
  \mathbf{k}[[x]],$
\end{itemize}
similarly to the Floer case (see
\S\ref{sec:defin-equiv-chain} and
\S\ref{sec:defin-infin-funct}), where the direct
sum runs over pairs $(q,y)$ such that:
\begin{itemize}
\item $y$ is a critical point of the
  pseudogradient on $BG$ lying above the pole
  $[1:0:\dots]$ in projective space,\footnote{With
    the correct interpretation of ``projective
    space'' (as either $\mathbb{R}P^{\infty}$ or
    $\mathbb{C}P^{\infty}$), this handles both
    cases $p=2$ and $p\ge 3$; in the former case,
    there is a unique choice of $y$, while in the
    latter case, one can arrange the construction
    so there are two choices of $y$.}
\item $q$ is a critical point of $P_{\eta}$, where
  $\eta$ is the distinguished lift of $y$ from
  $BG$ to $EG$. The word ``distinguished'' is to be
  understood vis-a-vis Figure
  \ref{fig:mod-p-equivariant} from
  \S\ref{sec:case-p-3}: the distinguished lifts
  were denoted $v_{0,e},v_{1,e}$.
\end{itemize}
The differential $d_{eq}$ on this complex is
defined by counting the $1$-dimensional components
in $\mathscr{M}(P_{\eta})$ and interpreting the count as
follows: if $\pi(s)$ is asymptotic to
$g\tau^{k}(\eta_{-})$ and $\eta_{+}$, and $q(s)$
is asymptotic to $gq_{-},q_{+}$, then $(q,\pi)$ is
registered as contributing towards the
coefficient:
\begin{equation*}
  \mathfrak{o}(q_{+})\otimes \mathfrak{o}(y_{+})\mapsto x^{k}\mathfrak{o}(q_{-})\otimes \mathfrak{o}(y_{-}),
\end{equation*}
where $y_{\pm}$ are the projections of
$\eta_{\pm}$ to $BG$. This is all in a manner
similar to \S\ref{sec:defin-infin-funct}. The
differential squares to zero by the same arguments
used in \cite{seidel-smith-GAFA-2010,
  seidel-eq-pop, shelukhin-zhao-JSG-2021} and in
\S\ref{sec:case-p-2}, \S\ref{sec:case-p-3}, and
\S\ref{sec:asymp-ops-orient-lines}.

\subsubsection{Definition of the Morse continuation map}
\label{sec:defin-morse-continuation map}

Consider the space of Moore paths $\mathfrak{M}$ consisting
of pairs of:
\begin{itemize}
\item a real number
$w\in [0,\infty)$ and
\item a smooth path
$s\in \R\mapsto P_{\eta,s}$ of vector fields
on $W$ (parametrized by $\eta$) such that
$\partial_{s}P_{\eta,s}=0$ for $s$ outside
$[0,w]$,
\end{itemize}
such that \ref{item:M-equivariant}, and \ref{item:M-self-similarity} are satisfied for all $s$.

\begin{definition}\label{definition:remains-at-zero}
  It will be important in the sequel to consider families $P_{\eta,s}$ such that $P_{\eta,s}=0$ holds identically. In this case we say that $P_{\eta,s}$ \emph{remains at $0$}.
\end{definition}

If we fix two endpoints, say, $P_{\eta}$ and
$P_{\eta}'$, then we let
$\mathfrak{M}(P_{\eta},P_{\eta}')$ be the
subspace of Moore paths whose endpoints are
$P_{\eta}$ and $P_{\eta}'$. Each of these morphism spaces in $\mathfrak{M}$ is a convex space.

\begin{definition}\label{definition:admissible-MMP}
  A path $\phi\in \mathfrak{M}$ is said to be
  \emph{admissible} for the Morse continuation map
  provided that the endpoints $P_{\eta,0}$ and $P_{\eta,w}$ of $\phi$ are valid Morse--Borel data, and either:
  \begin{itemize}
  \item $P_{\eta,0},P_{\eta,w}$ are both of type $+$ or both of type $-$,
  \item $P_{\eta,0}$ is of type $-$ and
    $P_{\eta,w}$ is of type $+$.
  \end{itemize}
  In other words, paths from type $+$ to
  type $-$ are not admissible.
\end{definition}

Associated to an admissible path $\phi\in \mathfrak{M}$, expressed as $P_{\eta,s}$,
one can form the moduli space
$\mathscr{M}(\phi)$ of solutions $(q,\pi)$ to
the following equation:
\begin{equation}
  \label{eq:morse-cont-map}
  \left\{
    \begin{aligned}
      &q'(s)=-X_{-s,\pi(s)}(q(s)),&&\lim_{s\to\pm\infty}q(s)\text{ exists},\\
      &\pi'(s)=-V(\pi(s)).\\ 
    \end{aligned}
  \right.
\end{equation}
Definition \ref{definition:admissible-MMP} is
used to prove an a priori bound for solutions
of \eqref{eq:morse-cont-map}: any sequence
$q_{k}$ of solutions to
\eqref{eq:morse-cont-map} must remain in a
uniform compact set.

By counting the rigid elements (as in
\S\ref{sec:defin-borel-equiv}) associated to
the moduli spaces $\mathscr{M}(\phi)$ for
generic admissible paths, one defines chain maps between
the equivariant Morse complexes:
\begin{equation}\label{eq:morse-continuation-map}
  \mathfrak{c}_{\phi}:\mathrm{CM}_{\mathrm{eq}}(P_{\eta})\to \mathrm{CM}_{\mathrm{eq}}(P'_{\eta}).
\end{equation}
The map is independent of the choice of
morphism $\phi$ from $P_{\eta}$ to $P_{\eta}'$ in
$\mathfrak{M}$, up to chain
homotopy. Moreover, standard gluing arguments
then prove that continuation maps commute
with each other up to chain homotopy, and
consequently one shows all chain complexes
$\mathrm{CF}_{\mathrm{eq}}(P_{\eta})$ of the
same type (either type $-$ or type $+$) are
chain homotopy equivalent (via continuation
maps). Note that the map in the statement of
Theorem \ref{theorem:commutative-square} is a
special case of
\eqref{eq:morse-continuation-map}, when
$P_{\eta},P_{\eta}'$ are of different types.

The direct limit of the homology groups
$\mathrm{HM}_{\mathrm{eq}}(P_{\eta})$ over
the diagram of choices $P_{\eta}$ of the same
type (connected by continuation maps) yields
two invariants:
\begin{itemize}
\item type $+$: \emph{the equivariant cohomology} $\mathrm{HM}_{\mathrm{eq}}(W)$,
\item type $-$: \emph{the relative equivariant cohomology} $\mathrm{HM}_{\mathrm{eq}}(W,\partial W)$,
\end{itemize}
and the morphisms \eqref{eq:morse-continuation-map} in $\mathfrak{M}$ connecting type $-$ to type $+$ yield a map:
\begin{equation*}
  \mathrm{HM}_{\mathrm{eq}}(W,\partial W)\to \mathrm{HM}_{\mathrm{eq}}(W).
\end{equation*}

\subsubsection{The unit element in equivariant
  Morse cohomology}
\label{sec:unit-element}

Let us comment on the construction of a class
$1\in \mathrm{HM}_{\mathrm{eq}}(W)$. Let
$P_{\eta}$ be Morse--Borel data of type
$+$. Each generating pair
$(q,y)\in
\mathrm{CM}_{\mathrm{eq}}(P_{\eta})$ where
$\mathrm{index}(q)=\mathrm{index}(y)=0$
admits canonical isomorphisms:
\begin{equation*}
  \Z\simeq \mathfrak{o}(q)\text{ and }\Z\simeq\mathfrak{o}(y),
\end{equation*}
because the orientation line of a zero-dimensional
space has a canonical generator (corresponding to
the empty basis). Taking their tensor product with
$1\in \mathbf{k}[[x]]$, and direct summing over
such pairs $(q,y)$ produces an element in the
complex $\mathrm{CM}_{\mathrm{eq}}(P_{\eta})$, see
\S\ref{sec:defin-borel-equiv}, denoted by the
symbol $1$.

\begin{lemma}
  The element $1$ satisfies $d_{eq}(1)=0$, and is
  preserved under the continuation maps of \S\ref{sec:defin-morse-continuation map}.
\end{lemma}
\begin{proof}
  The proof is based on two standard ideas in
  Morse theory:
  \begin{itemize}
  \item the sum of all local minima is a cycle,
  \item the continuation map from one Morse
    complex to another sends the sum of all local
    minima to the sum of all local minima.
  \end{itemize}
  The details are left to the reader.
\end{proof}

\subsubsection{Equivariant pseudogradients}
\label{sec:equiv-pseud}

Let us now focus on a special class of
Borel--Morse data of type $+$, namely, data $P=P_{\eta}$ which
is globally independent of $\eta$, and which
satisfy the following additional axiom:
\begin{enumerate}[label=(EP)]
\item\label{item:EP} around each critical point
  $q$ in the fixed submanifold $F\subset W$ of the
  $G=\Z/p\Z$ action, there is a coordinate chart
  $B^{2r}\times B^{2n-2r}$ so that $F$ is
  identified with $B^{2r}\times \set{0}$, and
  $g\in G$ acts by
  $(z_{1},z_{2})\mapsto (z_{1},g z_{2})$ for some
  unitary representation of $G$ on $\C^{n-r}$, and
  so that $P$ appears as a direct sum
  $P_{1}+P_{2}$ of two linear vector fields, and
  $P_{2}$ points radially outwards.
\end{enumerate}
\begin{lemma}
  There exist vector fields $P$ satisfying
  \ref{item:EP} which define regular Borel--Morse
  data of type $+$.
\end{lemma}
\begin{proof}
  See \cite[\S2.2]{seidel-smith-GAFA-2010}. The
  idea is to pick $P$ first on the fixed
  submanifold $F\subset W$, and then extend it to
  a tubular neighborhood of $F$ by requiring it
  points outwards (from the perspective of
  Lyapunov functions, we add a positive definite
  quadratic function on each fiber of the normal
  bundle $NF$). Then we extend $P$ equivariantly
  to the rest of $W$ (generically), bearing in
  mind it should agree with $Z$ in the convex end.

  If the choice of extension from $NF$ to $W$ is
  done sufficiently generically, then all flow
  lines asymptotic to a critical point in
  $W\setminus F$ will be cut transversally.

  Moreover, if the initial choice of $P$ on $F$ is
  done sufficiently generically, then all flow
  lines contained entirely in $F$ will be cut
  transversally \emph{inside of $F$}. However, the
  assumption that $P$ points outwards along fibers
  of $NF$ ensures that such flow lines are also
  cut transversally when considered in $W$ (there
  is no change in the index difference).

  The cases covered in the two preceding
  paragraphs cover all possible flow lines, and
  the proof is complete.
\end{proof}

\subsection{Localization in equivariant Morse
  theory}
\label{sec:loc-in-eq-morse}

\emph{Localization} refers to the phenomenon that
the ring-theoretic localization of the equivariant
cohomology ring of a $G$-space at certain elements
recovers the regular cohomology of the fixed point
set. This statement can be made quite
sophisticated (we refer the reader to
\cite{atiyah-bott-moment-map} for a refined
result). In our paper, we will require only a
small input from this theory (for Theorem
\ref{theorem:main-infinity} part
\ref{theorem:main-infinity-2}):
\begin{proposition}\label{proposition:unit-not-torsion}
  If the action of $G=\Z/p\Z$ on a $G$-filling $W$
  has at least one fixed point, then the unit
  element $1\in \mathrm{HM}_{\mathrm{eq}}(W)$ is
  not $x$-torsion.
\end{proposition}
The group $\mathrm{HM}_{\mathrm{eq}}(W)$ is a
finitely generated $\mathbf{k}[[x]]$-module,
so it admits a torsion submodule and a free
quotient; the proposition guarantees the free
quotient is non-zero (in the presence of at least
one fixed point). More refined localization
results identify exactly the rank of this
free quotient.

The somewhat involved Morse-theoretic arguments of
\cite[\S 2]{seidel-smith-GAFA-2010}, which are
carried out in the case of a closed manifold with
a \( \Z/2 \)-action, prove something stronger. We
will take a simpler approach\footnote{Another
  reasonable approach is to essentially translate
  each step of \cite{atiyah-bott-moment-map} into
  Morse theoretic terms. This requires four
  ingredients: \emph{embedding functoriality}
  (i.e., pull-back maps), \emph{the relative
    cohomology of a pair}, \emph{the Thom
    isomorphism}, and the \emph{ring structure on
    (equivariant) cohomology}. The interpretation
  of these steps within Morse theory is more or
  less standard. The recipe in
  \cite{atiyah-bott-moment-map} then gives a
  refined localization result.} than that of
\cite{seidel-smith-GAFA-2010}.

\subsubsection{Proof of Proposition
  \ref{proposition:unit-not-torsion}}
\label{sec:proof-prop-not-torsion}

We assume $p\ge 3$, as the case $p=2$ is simpler
(and is essentially established in
\cite{seidel-smith-GAFA-2010}).

Without loss of generality, we pick $P_{\eta}=P$
satisfying \ref{item:EP}. Since we assume there is
at least one fixed point, the fixed point
submanifold $F$ is non-empty, and, by construction
of $P$, there is at least one index $0$ critical
point $q_{0}$ in the fixed point submanifold.

Let us introduce symbols $v_{0},v_{1}$ as the
index $0$ and index $1$ critical points on $BG$;
all other critical points are of the form
$\tau^{k}(v_{0})$ or $\tau^{k}(v_{1})$, and each
such critical point has $p$ lifts to critical
points on $EG$ (shown in Figure
\ref{fig:mod-p-equivariant} using notation
$v_{0,g}$ and $v_{1,g}$ for the lifts of
$v_{0},v_{1}$).

Recall from \S\ref{sec:defin-borel-equiv} that:
\begin{equation}\label{eq:morse-complex-eq}
  \mathrm{CM}_{\mathrm{eq}}(P) := \bigoplus_{q}(\mathfrak{o}(v_{0})\otimes \mathfrak{o}(q)\otimes \mathbf{k}[[x]])\oplus (\mathfrak{o}(v_{1})\otimes \mathfrak{o}(q)\otimes \mathbf{k}[[x]]),
\end{equation}
where the sum is over all critical points $q$
(zeros of $P$). Let us consider the projection map
onto the summand:
\begin{equation*}
  Q(P,q_{0}):=(\mathfrak{o}(v_{0})\otimes \mathfrak{o}(q_{0})\otimes \mathbf{k}[[x]])\oplus (\mathfrak{o}(v_{1})\otimes \mathfrak{o}(q_{0})\otimes \mathbf{k}[[x]])
\end{equation*}
associated to the chosen index $0$ critical point
$q_{0}\in F$. This map is a quotient by a
subcomplex because any flow line $(q(s),\pi(s))$
which ends with $q(-\infty)=q_{0}$ must also start
with $q(+\infty)=q_{0}$.

The induced differential on $Q(P,q_{0})$
yields a complex which computes the
equivariant Morse cohomology of a point
$\mathrm{HM}_{\mathrm{eq}}(\mathrm{pt})$
(this is obvious, since the quotient
differential on $Q(P,q_{0})$ only considers
flow lines $(q(s),\pi(s))$ for which
$q(s)=q_{0}$ while $\pi(s)$ is allowed to
vary arbitrarily on $EG$).\footnote{The
  quotient map
  $\mathrm{CM}_{\mathrm{eq}}(P) \to Q(P,
  q_{0})$ is the Morse-theoretic realization
  of the cohomological pullback
  $H^{*}_{G}(W) \to H^{*}_{G}(\set{q_0})$
  induced by the inclusion $\set{q_0}\to W$.}
It is also clear from the definitions that
this quotient map sends
$1\in \mathrm{HM}_{\mathrm{eq}}(W)$ to
$1\in
\mathrm{HM}_{\mathrm{eq}}(\mathrm{pt})$. We
now invoke the following lemma:
\begin{lemma}
  The element
  $1\in \mathrm{HM}_{\mathrm{eq}}(\mathrm{pt})$ is
  not $x$-torsion; here it is important to assume
  that $G=\mathbf{k}=\Z/p\Z$.
\end{lemma}
\begin{proof}
  This is obvious provided one knows that
  $\mathrm{HM}_{\mathrm{eq}}(q_{0})$ is the
  ordinary Borel $G$-equivariant cohomology
  $H^{*}_{G}(q_{0},\mathbf{k})$. However, it also
  follows from a direct computation of the
  differentials for the special pseudogradient
  $BG$; see, e.g.,
  \cite[{\S4}]{shelukhin-zhao-JSG-2021}.
\end{proof}

This completes the proof of Proposition
\ref{proposition:unit-not-torsion}.\hfill$\square$

\subsection{The equivariant PSS construction}
\label{sec:equiv-pss-constr}

The goal is to give a precise description of the
$\infty$-categories $\mathscr{P}(Y),\mathscr{N}(Y)$ introduced in
\S\ref{sec:pss-introduction}, and to construct the
$\infty$-functors:
\begin{equation*}
  \mathscr{P}(Y)\to \mathrm{N}_{\mathrm{dg}}\mathrm{Ch}(\mathbf{k}[[x]])\quad \mathscr{N}(Y)\to \mathrm{N}_{\mathrm{dg}}\mathrm{Ch}(\mathbf{k}[[x]])
\end{equation*}
with the properties stated in Theorems \ref{theorem:pss} and \ref{theorem:negative-pss}.

The aspects of the construction involving the
negative PSS category $\mathscr{N}(Y)$ are
entirely analogous to the aspects involving
the positive PSS category $\mathscr{P}(Y)$,
and for this reason we focus our discussion
on the case of $\mathscr{P}(Y)$.

\subsubsection{The PSS category}
\label{sec:PSS-category}

In this section we define the PSS category
$\mathscr{P}(Y)$. Recall from
\S\ref{sec:pss-introduction} and Remark
\ref{remark:barC-also} the $\infty$-category
$\bar{\mathscr{C}}(Y)$. 

\begin{definition}
  An $n$-simplex in $\mathscr{P}_{n}(Y)$ is a pair $(j,\Phi)$ where $-1\le j \le n$, and $\Phi$ is an $n$-simplex in $\bar{\mathscr{C}}(Y)$ such that:
  \begin{itemize}
  \item the vertices $j+1,\dots,n$ lie in $\mathscr{C}(Y)$ (the non-discriminant part),
  \item the subsimplex spanned by vertices $0,\dots,j$ remains at $\id$ in the sense of Definition \ref{definition:floer-remains-at}.
  \end{itemize}
  In particular, the simplices with $j=-1$ form a copy of $\mathscr{C}(Y)$.
\end{definition}
\begin{definition}\label{definition:floer-remains-at}
  A $j$-simplex $\Phi$ in
  $\bar{\mathscr{C}}(Y)$ \emph{remains at $\id$} provided all of
  its subcubes are mapped into the subset of
  $\mathrm{NN}(Y)$ consisting of the constant
  paths at $\id$ (of any length).
\end{definition}

\begin{definition}[Face and degeneracy maps]
  If $(j,\Phi)\in \mathscr{P}_{n}(Y)$, then the $i$th face map is given by:
  \begin{equation*}
    d_{i}(j,\Phi)=\left\{
      \begin{aligned}
        &(j,d_{i}\Phi)&&\text{ if }j<i,\\
        &(j-1,d_{i}\Phi)&&\text{ if }j \ge i.
      \end{aligned}
    \right.
  \end{equation*}
  The $i$th degeneracy map is given by:
  \begin{equation*}
    s_{i}(j,\Phi)=\left\{
      \begin{aligned}
        &(j,s_{i}\Phi)&&\text{ if }j<i,\\
        &(j+1,s_{i}\Phi)&&\text{ if }j\ge i.
      \end{aligned}
    \right.
  \end{equation*}    
\end{definition}

\begin{lemma}
  The collection of $n$-simplices
  $\mathscr{P}_{n}(Y)$, with the above face
  and degeneracy maps, defines an
  $\infty$-category.
\end{lemma}
\begin{proof}
  The face and degeneracy maps are closely
  related to the \emph{join operation}; see
  \cite[\S1.2.8]{lurie-HTT-2009}. It is an
  easy combinatorial verification to track
  the $j$ variable and bootstrap the known
  $\infty$-category structure on
  $\bar{\mathscr{C}}(Y)$ to conclude
  $\mathscr{P}(Y)$ is also an
  $\infty$-category.
\end{proof}

\begin{remark}
  The negative PSS category $\mathscr{N}(Y)$
  is defined analogously, except now
  the subsimplex spanned by the last vertices
  $n-j,\dots,n$ is required to be remain at $\id$,
  and the first $0,\dots,n-j-1$ vertices are
  required to lie in $\mathscr{C}_{0}(Y)$.
\end{remark}

\subsubsection{The category of Morse--Floer data}
\label{sec:category-morse-floer}

As with the definition of the $\infty$-functor in \S\ref{sec:equiv-floer-compl}, we will actually define a canonical diagram:
\begin{equation}\label{eq:diagram-QPNdg}
  \begin{tikzcd}
    \mathscr{Q}^{*}(W)\arrow[r]\arrow[d]& \mathrm{N}_{\mathrm{dg}}\mathrm{Ch}(\mathbf{k}[[x]]),\\
    \mathscr{P}(Y)
  \end{tikzcd}
\end{equation}
and prove that the vertical arrow is a
trivial Kan fibration. The category
$\mathscr{Q}^{*}(W)$ contains auxiliary
information; most importantly, the fibers of
$\mathscr{Q}^{*}(W)$ over
$\id\in \mathscr{P}(Y)$ represent choices of
regular Morse--Borel data.

Similarly to the discussion in
\S\ref{sec:auxil-categ-floer}, we will first
define a larger category $\mathscr{Q}(W)$ and
then define $\mathscr{Q}^{*}(W)$ as the
subcategory of ``regular'' simplices.

\begin{definition}\label{definition:MF}
  Let us define $\mathfrak{MF}$ to be the space of triples $(w,\phi,\varphi)$ where:
  \begin{itemize}
  \item $(w,\phi)\in \mathfrak{M}$, as in \S\ref{sec:defin-morse-continuation map}, (a Moore path of Morse--Borel data),
  \item $(w,\varphi)\in \mathfrak{F}$, as in \S\ref{sec:auxil-categ-floer}, (a Moore path of Floer--Borel data),
  \end{itemize}
  satisfying either:
  \begin{enumerate}[label=(\roman*)]
  \item $\phi$ remains at $0$ (Definition \ref{definition:remains-at-zero}), and $\varphi$ has non-degenerate endpoints,
  \item $\phi$ has non-degenerate endpoints
    of type $+$, and $\varphi$ remains at
    $\id$ (Definition
    \ref{definition:floer-remains-at}),
  \item $\phi$ joins type $+$ data to the $0$ vector field, and $\varphi$ joins $\id$ to non-degenerate data.
  \end{enumerate}
  The space $\mathfrak{MF}$ should be thought of as the space of Moore paths in the space of hybrid Morse-Floer data. The concatenation operations $\#_{\tau}$ (insertion of a stationary path) is well-defined between any two paths in $\mathfrak{MF}$ with composable endpoints ((iii) composed with (i) is again of type (iii), etc).
\end{definition}
\begin{definition}
  A zero simplex in $\mathscr{Q}_{n}(W)$ is
  either a Morse--Borel datum $P_{\eta}$ of
  type $+$, as in
  \S\ref{sec:morse-borel-data}, or a
  Floer--Borel data
  $(\varphi_{\eta,t},J)$ as in
  \S\ref{sec:auxil-categ-floer}.
  
  An $n$-simplex $\Psi$ in $\mathscr{Q}_{n}(W)$, for $n\ge 1$, is a collection of zero simplices $\Psi^{i}\in \mathscr{Q}_{0}(W)$ for $i=0,\dots,n$, and a collection smooth cubes:
  \begin{equation*}
    \Psi^{v_{0}\dots v_{m}}:[0,\infty)^{m-1}\to \mathfrak{MF},
  \end{equation*}
  for all $0\le v_{0}<\dots<v_{m}\le n$ and
  $m>0$, such that, as in Definition \ref{definition:simplices},
  \begin{enumerate}
  \item the endpoints of
    $\Psi^{v_{0}\dots v_{m}}(x)$ are
    $\Psi^{v_{0}}$ and $\Psi^{v_{m}}$, for
    all $x$; this is understood in the
    appropriate sense that we ignore the
    component which is set to $0$ or $\id$
    in Definition \ref{definition:MF},
  \item \label{item:break} as $x_{j}\to \infty$,
    \begin{equation*}
      \Psi^{v_{0}\dots v_{m}}(x)=\Psi^{v_{0}\dots v_{j}}(x_{1},\dots,x_{j-1})\#_{\tau(x_{j})} \Psi^{v_{j}\dots v_{m}}(x_{j+1},\dots,x_{m-1}),
    \end{equation*}
  \item \label{item:set-face-zero} when $x_{j}=0$,
    \begin{equation*}
      \Psi^{v_{0}\dots v_{m}}(x)=\Psi^{v_{0}\dots \hat{v_{j}}\dots v_{m}}(x_{1},\dots,x_{j-1},x_{j+1},\dots,x_{m-1}),   
    \end{equation*}
  \item \label{item:microcollaring} on the
    face $x_{j}=0$,
    $\partial_{x_{j}}\Psi^{v}(x)$ vanishes to
    infinite order.
  \end{enumerate}
\end{definition}

The definition of the simplicial structure on
$\mathscr{Q}(W)$ and the verification that it
is an infinity category proceeds exactly the
same as in
\S\ref{sec:infin-categ-contact-isotopies}.

\subsubsection{The subcategory of regular Morse--Floer data}
\label{sec:subc-regul-morse}

The next step is to define the category
$\mathscr{Q}^{*}(W)$ as the subcategory of
simplices which render a certain collection
of countably many moduli spaces tranverse.

To each simplex we will associate the moduli spaces for each subcollection of vertices $\set{v_{0}<\dots<v_{m}}$. The primary case of interest is when $v_{0}$ is Morse type and $v_{m}$ is Floer type, as these are associated to the hybrid moduli spaces. Let us therefore focus on this case.

From each element
$(w,\phi,\varphi)\in \mathfrak{MF}$ joining a
Morse type vertex to a Floer type vertex, one
can extract:
\begin{itemize}
\item Hamiltonian vector fields $Y_{\eta,s,t}$ and $X_{\eta,s,t}$, as in \S\ref{sec:moduli-spaces},
\item Morse pseudogradients $P_{\eta,s}$, as in \S\ref{sec:defin-morse-continuation map},
\end{itemize}
defined for $s\in \R$ and $t\in [0,1]$. Using these, we associate to $w,\phi,\varphi$ the moduli space $\mathscr{M}(w,\phi,\varphi)$ of solutions $(\pi,q,u)$ to:
\begin{equation}\label{eq:hybrid-morse-floer}
  \left\{
    \begin{aligned}
      &\pi:\R\to BG,\quad q:\R\to W,\quad u:\R\times \R/\Z\to W,\\
      &\pi'(s)=-V(\pi(s)),\\
      &q'(s)=-P_{-s,\pi(s)}(q(s)),\quad \lim_{s\to\infty} q(s)\text{ exists},\\
      &(\bd_{s}u-\rho(t)Y_{\pi(s),s,t}(u))+J_{s}(u)(\bd_{t}u-X_{\pi(s),s,t}(u))=0,\\
      &\text{the integral of }\omega(\bd_{s}u-\rho(t)Y_{\pi(s),s,t}(u),\bd_{t}u-X_{\pi(s),s,t}(u))\text{ is finite},\\ 
      &u(+\infty)=q(-\infty).
    \end{aligned}
  \right.
\end{equation}
This is considered as a hybrid between
\eqref{eq:morse-cont-map} and
\eqref{eq:single-moore-path-eqn-mod-2}; see
Figure \ref{fig:morse-floer}.\footnote{The
  negative PSS version of this equation, used
  for Theorem \ref{theorem:negative-pss}, is
  a reflected version, with $u(-\infty)=q(\infty)$.}

Given $\Psi\in \mathscr{Q}_{n}(W)$, and any $\set{v_{0}<\dots<v_{m}}\subset \set{0,\dots,n}$ we consider the cubes:
\begin{equation*}
  \Psi^{v_{0}\dots v_{m}}:[0,\infty)^{m-1}\to \mathfrak{MF}.
\end{equation*}
This is associated with the parametric moduli space of tuples $(x,\pi,q,u)$:
\begin{equation}\label{eq:parametric-depending-on-cube}
  \left\{
    \begin{aligned}
      &x\in [0,\infty)^{m-1},\\
      &\Psi^{v_{0}\dots v_{m}}(x)=(w,\phi,\varphi),\\
      &(\pi,q,u)\in \mathscr{M}(w,\phi,\varphi)\text{ as in }\eqref{eq:hybrid-morse-floer}.
    \end{aligned}
  \right.
\end{equation}
There are two edge cases; if all vertices $v_{0},\dots,v_{m}$ are of Morse type, then $\varphi$ is just a constant path at $\id$, so we ignore $u$ and only ask that:
\begin{itemize}
\item $(\pi,q)$ solves equation \eqref{eq:morse-cont-map}.
\end{itemize}
On the other hand, if $v_{0},\dots,v_{m}$ are all Floer type vertices, then we ignore the Morse part and only ask that $(\pi,u)$ solves:
\begin{equation}\label{eq:just-u-part}
  \left\{
    \begin{aligned}
      &\pi:\R\to BG\quad u:\R\times \R/\Z\to W,\\
      &\pi'(s)=-V(\pi(s)),\\
      &(\bd_{s}u-\rho(t)Y_{\pi(s),s,t}(u))+J_{s}(u)(\bd_{t}u-X_{\pi(s),s,t}(u))=0,\\
      &\text{the integral of }\omega(\bd_{s}u-\rho(t)Y_{\pi(s),s,t}(u),\bd_{t}u-X_{\pi(s),s,t}(u))\text{ is finite}.
    \end{aligned}
  \right.  
\end{equation}

\begin{definition}\label{definition:regular-PSS}
  If the parametric moduli spaces \eqref{eq:parametric-depending-on-cube}, or \eqref{eq:morse-cont-map} or \eqref{eq:just-u-part}, are transverse for each subcube determined by $\Psi\in \mathscr{Q}_{n}(W)$, then we say $\Psi$ is a regular $n$-simplex. The collection of all regular simplices is denoted $\mathscr{Q}^{*}(W)$.
\end{definition}

\begin{lemma}\label{lemma:Qast-closed}
  The regular simplices $\mathscr{Q}^{*}(W)$ form a sub $\infty$-category.
\end{lemma}
\begin{proof}
  It follows, essentially by construction, that
  $\mathscr{Q}^{*}(W)$ is closed under face
  maps (since we require transversality for
  all subcubes in Definition
  \ref{definition:regular-PSS}). In fact,
  $\mathscr{Q}^{*}(W)$ is also closed under
  degeneracy maps, as follows from the same
  argument used in Proposition
  \ref{proposition:degeneracy-closed}.
  Briefly: the actual equations which are
  associated to a degenerate simplex are the
  same equations which appear for the underlying
  non-degenerate simplex, but the parameter
  space has been enlarged with additional
  degrees of freedom. This implies that
  regularity is preserved under taking
  degeneracies.

  The fact that $\mathscr{Q}^{*}(W)$
  satisfies the horn filling property follows
  from Lemma \ref{lemma:triv-kan-fibr-PSS}
  below.
\end{proof}

\begin{figure}[h]
  \centering
  \begin{tikzpicture}[yscale=.5]
    \draw (0,0)--(4,0)node[draw,circle,inner sep=1pt,fill]{}node[right]{$\text{input}$};
    \draw (-4,1)--(-2,1)to[out=0,in=90](0,0)to[out=-90,in=0](-2,-1)--(-4,-1);
    \draw (-4,0) circle (0.2 and 1); \node at (-4.2,0)[left]{$\text{output}$};
    \draw[<->] (0.0,0.4)-- node[above]{$w(\phi)$} (1.4,0.4);
    \node[draw,circle,inner sep=1pt,fill] (A) at (1.4,0){} node[below]at(A){$s=0$};
    \node[draw,circle,inner sep=1pt,fill] at (0,0){};
  \end{tikzpicture}
  \caption{Illustration of the hybrid Morse--Floer equation \eqref{eq:hybrid-morse-floer}.}
  \label{fig:morse-floer}
\end{figure}

\subsubsection{The trivial Kan fibration property}
\label{sec:triv-kan-fibr}

\begin{lemma}\label{lemma:triv-kan-fibr-PSS}
  There is a commutative square of simplicial maps:
  \begin{equation*}
    \begin{tikzcd}
      \mathscr{D}^{*}(W)\arrow[d]\arrow[r,"\subset"]&\mathscr{Q}^{*}(W)\arrow[d]\\
      \mathscr{C}(Y)\arrow[r,"\subset"]&\mathscr{P}(Y)
    \end{tikzcd}
  \end{equation*}
  and the vertical morphisms are trivial Kan fibrations.
\end{lemma}
\begin{proof}
  This boils down to the following problem: given a $[0,\infty)^{\infty}\to \mathrm{NN}(Y)$ arising from a simplex in $\mathscr{P}(Y)$, and a lift of the boundary of the cube to $\mathfrak{MF}$, can one extend the lift to the interior? The Floer part can be handled in the same way as Lemma \ref{lemma:is-a-trivial-kan-fibration}. The Morse part is handled straightforwardly by convexity in the space of Moore paths of vector fields.
\end{proof}

\subsubsection{Definition of the PSS infinity functor}
\label{sec:defin-infin-funct-PSS}

In this section we complete the first part of Theorem \ref{theorem:pss} (and, by analogy, also Theorem \ref{theorem:negative-pss}), and construct the $\infty$-functor $\mathscr{Q}^{*}(W)\to \mathrm{N}_{\mathrm{dg}}\mathrm{Ch}(\mathbf{k}[[x]])$; the passage to a functor defined on $\mathscr{P}(Y)$ (by choosing sections of a trivial Kan fibration) follows the same lines as \S\ref{sec:defin-infty-funct}.

Each $n$-simplex $\Psi\in \mathscr{Q}^{*}(W)$ yields a collection of subcubes; let us label the important ones for the purposes of verifying Definition \ref{definition:infinity-functor}:
\begin{equation}
  \label{eq:important-subcubes}
  \text{$\Psi^{0\dots n}$, $\Psi^{0\dots \hat{i}\dots n}$, $\Psi^{0\dots i}$, and, $\Psi^{i\dots n}$, for $i=1,\dots,n-1$}
\end{equation}
Each of these subcubes determines a parametric moduli space, which is assumed to be regular (as we start with $\Psi\in \mathscr{Q}^{*}(W)$).

The overall strategy is as follows:
\begin{enumerate}[label=(\alph*)]
\item\label{item:pss-strategy-1} vertices $\Psi$ of
  Morse type are sent to
  $\mathrm{CM}_{\mathrm{eq}}(\Psi)$, while
  vertices of Floer type are sent to
  $\mathrm{CF}_{\mathrm{eq}}(\Psi)$; to unify
  notation, let us denote this chain complex
  by $\mathrm{C}_{\mathrm{eq}}(\Psi)$;
\item\label{item:pss-strategy-2}
  $n$-simplices $\Psi$ with $n\ge 1$
  between zero-simplices $\Psi|_{0}$ and
  $\Psi|_{n}$ are sent to the map
  $\mathfrak{c}_{\sigma}:\mathrm{C}_{\mathrm{eq}}(\Psi|_{0})\to
  \mathrm{C}_{\mathrm{eq}}(\Psi|_{n})$
  obtained by counting the rigid solutions of
  \eqref{eq:parametric-depending-on-cube} for
  the top subcube (corresponding to the full
  set of vertices $0,\dots,n$) in an
  appropriate fashion (recalled below);
\end{enumerate}
Then, to prove the $\infty$-functor relations, one argues as follows:
\begin{enumerate}[label=(\alph*),resume]
\item\label{item:pss-strategy-3} the relation
  between the one-dimensional components of
  \eqref{eq:parametric-depending-on-cube},
  for the top subcube, and the rigid elements
  in the moduli spaces associated to the
  lower dimensional subcubes listed in
  \eqref{eq:important-subcubes} is used to show that
  the first part of Definition
  \ref{definition:infinity-functor} is
  satisfied;
\item\label{item:pss-strategy-4} the rigid
  components of the moduli space associated
  to the top subcube of a degenerate
  $n$-simplex are shown to either yield the
  identity chain map (if $n=1$) or to the $0$
  map (if $n>1$); the latter case is proved
  by exhibiting a degree of freedom in the
  moduli space, similarly to Lemma
  \ref{lemma:Qast-closed}.
\end{enumerate}

There is not much left to explain for
\ref{item:pss-strategy-1}; the complexes are
as defined in \S\ref{sec:defin-equiv-chain},
\S\ref{sec:case-p-3},
\S\ref{sec:defin-infin-funct} (Floer), and
\S\ref{sec:defin-borel-equiv} (Morse).

The counting scheme used for
\ref{item:pss-strategy-2} follows mostly the
same lines as \S\ref{sec:defin-infin-funct}
and \S\ref{sec:defin-borel-equiv}. We only explain the case when there are both types of vertices; the cases when the equation is entirely Morse or entirely Floer follow simpler lines and are left to the reader.

Each solution $(x,\pi,q,u)$ determines an
linearized operator:
\begin{equation*}
  D_{\pi,q,u}:T_{x}\oplus T_{\pi}\oplus T_{q}\oplus W^{1,p}(u^{*}TW)\to L^{p}(u^{*}TW)\oplus TW_{u(\infty)},
\end{equation*}
where $T_{x},T_{\pi},T_{q}$ are shorthands for the
tangent spaces at $x,\pi$ and $q$. The map to
the second factor is the linearization of the constraint $u(\infty)=q(w)$; if
$(\delta x,\delta \pi,\delta q,\delta u)$ is
a variation, the map to $TW_{u(\infty)}$ is
given by:
\begin{equation*}
  q'(w(x))dw_{x}(\delta x)+\delta q(w(x))-\delta u(\infty),
\end{equation*}
while the map to the first factor is defined
by a standard linearization procedure, as
discussed in \S\ref{sec:defin-equiv-chain} and \S\ref{sec:note-param-moduli}.

This linearized operator is surjective, by
definition of regularity, and the kernel of
this operator is the tangent space of the
moduli space at $(x,\pi,q,u)$.

Thus, in the case of the zero-dimensional moduli spaces, one
can use the orientability of the moduli space
and the parallel transport maps for Fredholm
determinants to produce generators of:
\begin{equation*}
  \mathfrak{o}(T_{x})\otimes\mathfrak{o}(T_{\pi})\otimes \mathfrak{o}(T_{q})\otimes \mathfrak{o}(D_{u})\otimes \mathfrak{o}(TW_{u(\infty)}),
\end{equation*}
and then use the canonical gluing isomorphisms from \S\ref{sec:line-proc} to identify:
\begin{equation*}
  \mathfrak{o}(D_{u})\simeq \mathfrak{o}(\gamma_{-})\text{ and }\mathfrak{o}(T_{\pi})\simeq \mathfrak{o}(y_{-})\otimes \mathfrak{o}(y_{+}),
\end{equation*}
and consequently, as in \S\ref{sec:defin-infin-funct}, we have a canonical identification:
\begin{equation*}
  \mathfrak{o}(T_{\pi})\otimes \mathfrak{o}(T_{q})\otimes \mathfrak{o}(D_{u})\otimes \mathfrak{o}(TW_{u(\infty)})\simeq \mathfrak{o}(\gamma_{-})\otimes \mathfrak{o}(y_{-})\otimes \mathfrak{o}(q_{+})\otimes \mathfrak{o}(y_{+}),
\end{equation*}
where $\mathfrak{o}(q_{+})$ is the orientation line of the \emph{unstable manifold} at $q_{+}=q(+\infty)$ (note that $T_{q}$ is naturally identified with the stable manifold at $q_{+}$).

In this fashion, the rigid solutions in the moduli space associated to the top subcube of $\sigma$ are packaged into a linear map:
\begin{equation*}
  \mathfrak{c}_{\sigma}:\bigoplus \mathfrak{o}(q_{+})\otimes \mathfrak{o}(y_{+})\otimes \mathbf{k}[[x]]\to   \bigoplus \mathfrak{o}(\gamma_{-})\otimes \mathfrak{o}(y_{-})\otimes \mathbf{k}[[x]].
\end{equation*}
This completes our discussion of \ref{item:pss-strategy-2}. The discussion of \ref{item:pss-strategy-3} is analogous to \S\ref{sec:chain-homot-assoc} and \S\ref{sec:defin-infin-funct}; we only briefly comment on the novel aspects introduced by the hybrid equation \eqref{eq:parametric-depending-on-cube}.

Let us consider the case of an end in the 1-dimensional moduli space where $x_{i}\to \infty$; there are two cases to consider, depending on whether the vertex $v_{i}$ is of Morse type or Floer type. In either case, $\Psi(x)$ breaks into:
\begin{itemize}
\item $\Psi^{0\dots i}(x_{1},\dots,x_{i-1})\#_{\tau(x_{i})}\Psi^{i\dots n}(x_{i+1},\dots,x_{n-1})$,
\end{itemize}
and the only difference is which piece of $\mathfrak{MF}$; if $v_{i}$ is Morse type, then:
\begin{itemize}
\item $\Psi^{0\dots i}$ takes values in the Morse-Morse part of $\mathfrak{MF}$,
\item $\Psi^{i\dots n}$ takes values in the Morse-Floer part of $\mathfrak{MF}$,
\end{itemize}
otherwise, if $v_{i}$ is of Floer type, then:
\begin{itemize}
\item $\Psi^{0\dots i}$ takes values in the Morse-Floer part of $\mathfrak{MF}$,
\item $\Psi^{i\dots n}$ takes values in the Floer-Floer part of $\mathfrak{MF}$.
\end{itemize}
These is the expected breaking, and terms with
this type of breaking contribute the term
$\mathfrak{c}_{\sigma|[i\dots
  n]}\circ \mathfrak{c}_{\sigma|[0\dots i]}$ which
appears in \eqref{eq:dg-nerve}.

Let us explain in a bit more detail the
degenerating sequence in the case when
$v_{i}$ is of Morse type. A sequence of
solutions $(x_{k},\pi_{k},q_{k},u_{k})$ will
converge, after passing to a subsequence, and
on compact subsets of the domain (either $\R$
or $\R\times \R/\Z$), to an equation for the
subcube $\Psi^{0\dots i}$. Note that the
Floer part converges to a constant, as it
solves the holomorphic curve equation and has
finite energy. To catch the other limit, one
needs to translate the solutions by the
amount:
\begin{equation*}
  w(\Psi^{0\dots i}(x_{1},\dots,x_{i-1}))+\tau(x_{i}),
\end{equation*}
where $w$ stands for the width of a Moore
path. After this translation is done, the
sequence of solutions will converge on
compact subsets to a solution in the equation
for the subcube $\Psi^{i\dots n}$ (after passing to a subsequence).

The rest of the details of \ref{item:pss-strategy-3} follows similar lines to \S\ref{sec:chain-homot-assoc}, and the necessary deferal to the orientation line gluing arguments as in \S\ref{sec:defin-infin-funct}.

Finally, \ref{item:pss-strategy-4},
concerning the counts associated to
degenerate simplices, follows easily from the
discussion in Lemma \ref{lemma:Qast-closed}
and the same argument used in
\S\ref{sec:chain-homot-assoc}. Namely, when
one degenerates, the parameter space is
simply enlarged by one dimension, without
changing the actual underlying solutions;
this precludes the formation of rigid
solutions. The case of a degenerate
$1$-simplex is an edge case, but it is standard and handled, as in, say, \cite{hofer-salamon-1995}.

This completes our construction of the functor $\mathscr{Q}^{*}(Y)\to \mathrm{N}_{\mathrm{dg}}\mathrm{Ch}(\mathbf{k}[[x]])$. The passage to $\mathscr{P}(Y)$ follows the strategy of \S\ref{sec:defin-infty-funct} and \S\ref{sec:survey-infty-categories}, namely by picking sections of the trivial Kan fibration from \S\ref{sec:triv-kan-fibr}.

\subsection{The low-slope PSS isomorphism}
\label{sec:low-slope-pss}

We explain why the maps:
\begin{equation*}
  \mathrm{PSS}:\mathrm{CM}_{\mathrm{eq}}(P_{\eta})\to \mathrm{CF}_{\mathrm{eq}}(R_{\epsilon t})
\end{equation*}
associated to specific $1$-simplices
$\varphi_{s,t}=R_{s\epsilon t}$ joining $\id$
to $R_{\epsilon t}$, for small positive
slopes $\epsilon$, are chain-homotopy
equivalences. The analogous statement for the
negative PSS map of Theorem
\ref{theorem:negative-pss} is proved in the
exact same way, and so we only focus on the
``positive'' PSS map. The argument we will
use follows
\cite[pp.\,24-25]{frauenfelder-schlenk-israeljm-2007};
the idea is to construct a chain-homotopy
inverse. Due to the similarity with
\cite{piunikhin-salamon-schwarz-1996,
  frauenfelder-schlenk-israeljm-2007}, we
only sketch the results.

\subsubsection{Geometric set up for the inverse of
  PSS}
\label{sec:geometric-set-up-PSS}

Let $r$ be a radial coordinate on the convex end
of $W$, so that $\Omega=\set{r\le 1}$ is an
aspherical symplectic domain with contact-type
boundary, and so that $r\ge 1$ is identified with
the positive half of the symplectization of
$\partial W$. We suppose that $r$ is
$G$-invariant. Then $X_{r}$ has an ideal
restriction to $\partial W$ as a Reeb flow $R$
which lifts a Reeb flow on $Y=\partial W/G$.

Let $h:\R\to \R$ be a convex cut-off function so
that:
\begin{itemize}
\item $h(r)=r$ for $r\ge 2$,
\item $h'(r)>0$ for $r>1$,
\item $r=3/2$ on $\Omega=\set{r\le 1}$,
\end{itemize}
and let $H=\epsilon h(r)$. We assume $\epsilon$ is
small enough that the only orbits of $X_{H}$ are
the constant orbits in $\Omega$. For a choice of
almost complex structure $J$ as in
\S\ref{sec:auxil-categ-floer}, the pair $(H,J)$ is
equivariant Borel data (certainly not regular)
whose ideal restriction is $R_{\epsilon t}$.

We also pick a vector field $P$ which agrees with
the Liouville vector field $Z$ in the convex end and is a
regular equivariant Morse--Borel data as in
\S\ref{sec:equiv-pseud}.

\subsubsection{Reverse PSS cylinders}
\label{sec:reverse-pss-cylinders}

Let $(\psi_{\eta,t},J)$ be regular Borel data
whose ideal restriction is $R_{\epsilon t}$. As in
\S\ref{sec:moduli-spaces}, let us denote by
$X_{\eta,t}$ the generator of $\psi_{\eta,f(t)}$,
where $f$ is the cut-off function from Figure
\ref{fig:graph-of-f}. We may assume, without loss
of generality, that $X_{\eta,t}=f'(t)X_{H}$
outside $\Omega$ (i.e., we assume
$\psi_{\eta,t}=R_{\epsilon t}$ outside
$\Omega$). We use this $f(t)$ only to be
consistent with the set-up of
\S\ref{sec:moduli-spaces}, rather than for any
purpose internal to this section.

Introduce the continuation data:
\begin{equation*}
  X_{\eta,s,t}:=(1-f(s))f'(t)X_{H}+f(s)X_{\eta,t}+E_{s,t},
\end{equation*}
for $s\in [0,1]$, extended to the rest of the
line by $s$-independence, where:
\begin{itemize}
\item $E_{s,t}$ is a $C^{\infty}$-small
  Hamiltonian pertubation, supported in $\Omega$.
\end{itemize}
Introduce the moduli space $\mathscr{R}$ of
\emph{reverse PSS continuation cylinders}:
\begin{equation}
  \label{eq:moduli-space-inverse-of-PSS}
  \left\{
    \begin{aligned}
      &u:\R\times \R/\Z\to W,\ q:(-\infty,0]\to W,\text{ and }\pi:\R\to EG,\\
      &\pi'(s)=-V(\pi(s)),\ q'(s)=-P(q(s)),\\
      &\bd_{s}u+J(u)(\bd_{t}u-X_{\pi(s),s,t}(u))=0,\\
      &u(-\infty)=q(0).
    \end{aligned}
  \right.
\end{equation}
This solution is asymptotic at the input to a
Hamiltonian orbit, and has a flow line of $-P$
connected at the output end.

Importantly, because $X_{\pi(s),s,t}(u)$ is
independent of $s$, outside of $\Omega$, the
finite energy solutions obey the necessary a
priori energy estimates needed for the usual Floer
theory compactness results.

We only consider those solutions
$(u,q,\pi)\in \mathscr{R}$ so that:
\begin{itemize}
\item $u$ has finite energy, which implies the
  left asymptotic is a removable singularity
  $u(-\infty)\in \Omega$ (see \cite[Figure
  3]{frauenfelder-schlenk-israeljm-2007});\footnote{This
    point is slightly subtle since $f'(t)X_{H}$
    has every point in the domain $\Omega$ as a
    constant $1$-periodic orbit, and has no other
    orbits. For any sequence $s_{n}\to-\infty$,
    there is a subsequence of $u(s_{n},t)$ which
    converges to a point $q$. Then, by some
    maximum principle, e.g., the one of
    \cite[Lemma 7.2]{abouzaid-seidel-GT-2010}, we
    conclude that $u$ takes values entirely in
    $\Omega$. It follows that $u$ is a genuine
    holomorphic curve on the half-cylinder
    $s\le 0$; having established this, the
    analysis follows the usual PSS arguments of
    \cite{piunikhin-salamon-schwarz-1996}.}
\item $q(-\infty)$ converges to a zero of
  $P$ (i.e., we do not consider solutions where
  $q(-\infty)$ drifts off in the non-compact end);
\end{itemize}

By counting the rigid such solutions in
$\mathscr{R}$ in the same way as in
\S\ref{sec:defin-infin-funct-PSS} (for generic choice
of the perturbation term $E_{s,t}$), we obtain a
map:
\begin{equation}\label{eq:inverse-of-PSS}
  \mathrm{CF}_{\mathrm{eq}}(\psi_{\eta,t},J)\to \mathrm{CM}_{\mathrm{eq}}(P);
\end{equation}
consideration of the $1$-dimensional part
of $\mathscr{R}$ proves this is a chain map.

\subsubsection{The chain homotopy inverse}
\label{sec:chain-homot-inverse}

We now explain why \eqref{eq:inverse-of-PSS} is
the chain homotopy inverse of the PSS map of
\S\ref{sec:defin-infin-funct-PSS}.

To define the PSS map, one ultimately uses the equation
\eqref{eq:hybrid-morse-floer}. In the case of
the $1$-simplex $s\mapsto R_{\epsilon s t}$, one needs
to pick a path of Borel data
$s\mapsto (\psi_{\eta,s,t},J)$ whose ideal
restriction is $R_{\epsilon s t}$. As in
\S\ref{sec:reverse-pss-cylinders}, we pick
$\psi_{\eta,s,t}$ agreeing with
$R_{\epsilon s t}$ outside of $\Omega$.

Unpacking \eqref{eq:hybrid-morse-floer} leads to
a family of vector fields $Y_{\eta,s,t}$ and
$X_{\eta,s,t}$ so that:
\begin{enumerate}[label=(\roman*)]
\item\label{item:curv-1} $Y_{\eta,s,t}=0$ outside
  of $s\in [-w(\varphi),0]$,
\item\label{item:curv-2} $X_{\eta,s,t}=X_{\eta,t}$
  for $s\le -w(\varphi)$, where $X_{\eta,t}$ is as in
  \S\ref{sec:reverse-pss-cylinders},
\item\label{item:curv-3} $X_{\eta,s,t}=0$ for
  $s\ge 0$,
\item\label{item:curv-4} both
  $X_{\eta,s,t},Y_{\eta,s,t}$ are $\eta$
  independent and lie in the line spanned by the
  Reeb vector field $R$, outside of $\Omega$,
\item\label{item:curv-5} the curvature of the
  Hamiltonian connection\footnote{see
    \cite[{\S3.2}]{brocic-cant-arXiv-2025} for an
    introduction to Hamiltonian connections} on
  the complement of $\Omega$ determined by
  $X_{s,t},Y_{s,t}$ is non-positive.
\end{enumerate}

One counts the rigid solutions of the parametric
moduli space where $\eta$ is constrained to equal
$\pi(s)$ for flow lines $\pi:\R\to EG$, as usual,
to define the chain map. This PSS chain map can be
post-composed with \eqref{eq:inverse-of-PSS}. By
standard Floer theory, the composition is chain
homotopic to the count of rigid cylinders solving
Floer's equation for glued Hamiltonian connections
(with flow lines attached at both ends, in the PSS
sense).

\begin{figure}[h]
  \centering
  \begin{tikzpicture}
    \draw (1.5,0.5)arc(90:-90:0.5)--(-1.5,-0.5)arc(-90:-270:0.5)--cycle;
    \draw (1.5,0) circle (0.1 and 0.5) (-1.5,0) circle (0.1 and 0.5);
    
    \draw (2,0)to[out=10,in=190](4,0)node[fill,circle,inner sep=1pt]{} (-2,0)to[out=190,in=10](-4,0)node[fill,circle,inner sep=1pt]{};
  \end{tikzpicture}
  \caption{Cartoon of the composition of
    PSS followed by \eqref{eq:inverse-of-PSS}}
  \label{fig:compo-pss}
\end{figure}
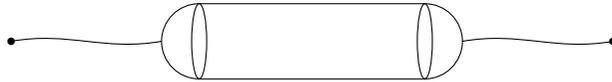

The glued Hamiltonian connections are
similarly determined by vector fields
$X_{R,\eta,s,t},Y_{R,\eta,s,t}$, where $R$ is a
gluing parameter. These glued connections still satisfy
\ref{item:curv-4} and \ref{item:curv-5}. In
addition, they satisfy the property that:
\begin{enumerate}[resume,label=(\roman*)]
\item\label{item:curv-6}
  $X_{R,\eta,s,t}=f'(t)X_{H}$ for $s\le -s_{R}$,
  and $X_{R,\eta,s,t}=0$ for $s\ge s_{R}$, for
  some $s_{R}$.
\end{enumerate}
The key idea is that the space of connections on
the cylinder (depending on the auxiliary parameter
$\eta$) satisfying \ref{item:curv-4},
\ref{item:curv-5} and \ref{item:curv-6} is a
convex space.\footnote{this uses \ref{item:curv-4}
  in a crucial way, otherwise having non-positive
  curvature is a non-linear condition; for related
  discussion, see
  \cite[\S3.4.2]{brocic-cant-arXiv-2025}} Thus we
can deform the glued data
$X_{R,\eta,s,t},Y_{R,\eta,s,t}$ in a one parameter
family until it agrees with the following:
\begin{itemize}
\item $X_{s,t}=(1-f(s))X_{H}$ on $s\in [0,1]$, and
  extended by $s$-independence.
\end{itemize}
Indeed, we can simply travel in a straight line
through the space of affine Hamiltonian
connections.

The solutions for this final Hamiltonian
connection are as follows:
\begin{enumerate}
\item a cylinder $u$ solving
  $\partial_{s}
  u+J(u)(\partial_{t}u-X_{s,t}(u))=0$,
\item a flow line of $-P$ on the interval
  $[1,\infty)$ starting at the removable
  singularity $u(+\infty)$,
\item a flow line of $P$ on the interval
  $(-\infty,0]$ ending at the removable
  singularity $u(-\infty)$,
\item an underyling flow line of $\pi:\R\to EG$,
  which does not influence the equation because we
  use equivariant data.
\end{enumerate}
By the maximum principle, e.g., \cite[Lemma
7.2]{abouzaid-seidel-GT-2010}, $u$ must lie
entirely in the domain $\Omega$, and so must be a
holomorphic sphere, and therefore must be
constant. Thus all we are doing is counting the
flow lines of $P$. The index zero requirement
implies the only solutions we will count are the
constant solutions, and so the resulting map is
the identity map
$\mathrm{CM}_{\mathrm{eq}}(P)\to
\mathrm{CM}_{\mathrm{eq}}(P)$.

Thus, by the usual chain homotopy argument, we
conclude the composition of the PSS map with
\eqref{eq:inverse-of-PSS} is chain homotopic to
the identity map, as desired.

Thus we have proved \eqref{eq:inverse-of-PSS} is
the left inverse of the PSS map. A similar
argument to the one given in
\cite{piunikhin-salamon-schwarz-1996,
  frauenfelder-schlenk-israeljm-2007} proves that
\eqref{eq:inverse-of-PSS} is also the right
inverse. We leave the details of this half of the
argument to the reader. This completes the proof
of Theorem \ref{theorem:pss}. \hfill$\square$

\subsection{Proof of Theorem \ref{theorem:commutative-square}}
\label{sec:proof-theorem-c-square}

The argument involves similar analysis as the
proof that \eqref{eq:inverse-of-PSS} is the
right inverse of PSS; namely it involves
nodal gluing between PSS cylinders whose
removable singularities are coincident. The
details are similarly left to the reader. For
further discussion, we refer the reader to
\cite[\S2.3.3]{cant-hedicke-kilgore-arXiv-2023}.\hfill$\square$

% 4

\section{Local Floer cohomology as the cone of a
  simple crossing}
\label{sec:local-floer-cohom}

\emph{Local Floer cohomology} is associated to a suitable set $C$ of
orbits. What is important is that Floer cylinders
which start and end in $C$ occupy a certain
minimal amount of energy (in a way which persists
under perturbations). This is the perspective
developed in \cite{floer-comm-math-phys-1989}. Some special cases are:
\begin{itemize}
\item $C$ is a single isolated fixed point; see \cite{ginzburg-gurel-JSG-2010, zhao-JSG-2019,
    shelukhin-zhao-JSG-2021,
    shelukhin-on-HZ-annals-2022};
\item $C$ is an $S^{1}$ family of orbits; see, e.g.,
  \cite{cieliebak-et-al-mathZ-1996,
    mclean-AGT-2012}.
\end{itemize}
In our setting, the families of orbits for which we
will develop a local Floer cohomology are
associated to crossings with the discriminant.

\subsection{Extenders}
\label{sec:extenders}

To any compact contact manifold $Y$ we will associate
a class of ``simple'' morphisms
$\mathscr{E}\to \mathscr{C}_{1}(Y)$ and a
map:
$$\mathrm{CF}_{\mathrm{loc}}:\mathscr{E}\to
\mathrm{Ch}(\mathbf{k})$$ by applying local
Floer cohomology theory in the symplectization
$SY$.

Define an \emph{extender}, denoted
$\psi_{s,t}\in \mathscr{E}$, to be a Hamiltonian
isotopy on $SY$ with $s,t$-generating Hamiltonians
$K_{s,t},H_{s,t}$ satisfying:
\begin{enumerate}[label=(E\arabic*)]
\item\label{item:extender-1} $H_{s,t}=H_{0,t}$ and
  $K_{s,t}=0$ on the negative end,
\item\label{item:extender-2} $X_{H_{0,t}}$ is
  Liouville equivariant everywhere,
\item\label{item:extender-3} $X_{H_{s,t}}$ is
  Liouville equivariant on the positive end,
\item\label{item:extender-4} $K_{s,1}$ is
  non-negative,
\item\label{item:extender-5} $X_{H_{1,t}}$
  generates a Hamiltonian diffeomorphism with a
  compact set of fixed points, all of which have
  negative action,
\item\label{item:extender-6} there is at most a
  single action value attained in
  \ref{item:extender-5}.
\end{enumerate}
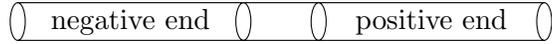
\begin{figure}[h]
  \centering
  \begin{tikzpicture}[scale=0.5]
    \draw (-2,0) coordinate (A) circle (0.2 and
    0.5) (4,0) coordinate (B) circle (0.2 and 0.5)
    (6,0) coordinate (C) circle (0.2 and 0.5)
    (12,0) coordinate (D) circle (0.2 and 0.5)
    (-2,0.5)--(12,0.5) (-2,-0.5)--(12,-0.5);
    \path (A)--node{negative end}(B)
  (C)--node{positive end}(D);
\end{tikzpicture}
\caption{The domain of an extender is the
  symplectization $SY$ decomposed into a negative
  end, a positive end, and a compact part in between.}
\label{fig:domain-of-extender}
\end{figure}

There is an ideal restriction
$\mathscr{E}\to \mathscr{C}_{1}(Y)$ given by the
flow generated by $H_{s,t}$ on the positive
end.

Given an extender $\psi_{s,t}\in \mathscr{E}$, and
a subset $\kappa\subset \pi_{0}(\Lambda Y)$ of free
homotopy classes of loops, we will define in
\S\ref{sec:local-floer-cohomology-of-an-extender}
the \emph{local Floer chain homotopy
  type}\footnote{Here a chain homotopy type is a
  chain complex up to chain homotopy equivalence.}
$\mathrm{CF}_{\mathrm{loc}}(\psi_{s,t};\kappa)$
for $H_{1,t}$ following
\cite{floer-comm-math-phys-1989}. Then:

\begin{theorem}\label{theorem:main-local-hf}
  Consider the infinity functor constructed in
  \S\ref{sec:equiv-floer-compl}:
  $$\mathrm{CF}_{\mathrm{eq}}:\mathscr{C}(Y)\to
  \mathrm{N}_{\mathrm{dg}}\mathrm{Ch}(\mathbf{k}[[x]]),$$
  assuming the hypotheses of Theorem
  \ref{theorem:main-infinity}. For
  $\psi_{s,t}\in \mathscr{E}$ with ideal
  restriction
  $\varphi_{s,t}\in \mathscr{C}_{1}(Y)$, the cone
  of the continuation map associated to
  $\varphi_{s,t}$ lies in the chain homotopy class
  $\mathrm{CF}_{\mathrm{loc}}(\psi_{s,t};\kappa)$
  where $\kappa$ is the collection of
  $W$-contractible orbits in $Y$ (namely, those
  orbits in $Y$ which lift to orbits in
  $\partial W$ which are contractible in $W$).
\end{theorem}
With this result in mind,
we say an extender $\psi_{s,t}$ is
\emph{trivial relative $W$} if $X_{H_{1,t}}$
has no $W$-contractible $1$-periodic orbits.

The proof of Theorem \ref{theorem:main-local-hf}
is given in
\S\ref{sec:proof-theorem-main-local-hf}. In order
to apply the theorem, it is necessary to show that
there is an abundance of such extenders:
\begin{theorem}\label{theorem:composition}
  For any $\varphi\in \mathscr{C}_{1}(Y)$
  which is a strictly positive Moore path,
  there exists an $n$-simplex $\Phi$ in
  $\mathscr{C}_{n}(Y)$ whose $[0,n]$ edge is
  $\sigma$, and whose $[i,i+1]$ edge is one of
  two types:
  \begin{enumerate}[label=(\roman*)]
  \item\label{type-extenders} the ideal
    restriction of an extender, or,
  \item\label{type-isomorphisms} an isomorphism in
    $\mathscr{C}_{1}(Y)$, i.e., projects to an
    isomorphism in the homotopy category,
  \end{enumerate}
  If $\varphi\in \mathrm{NN}(Y)$ never intersects the
  $W$-contractible discriminant, then one can pick
  the extenders in \ref{type-extenders} to be
  trivial relative $W$, and the continuation map
  associated to $\varphi$ is a
  quasi-isomorphism.
\end{theorem}
This result together with Theorem
\ref{theorem:main-local-hf} implies part
\ref{theorem:main-infinity-1} of Theorem
\ref{theorem:main-infinity}. It also implies,
together with \S\ref{sec:comp-proof-thm-decomp},
Theorem \ref{theorem:zig-zag}.

\subsection{Decompositions of continuation data}
\label{sec:decomp-cont-data}

The goal of this section is to prove Theorem
\ref{theorem:composition} on the decomposability
of 1-simplices into simpler pieces.

\subsubsection{Decompositions}
\label{sec:decompositions}

We say that a $1$-simplex $\sigma$ is a
\emph{composition} of simplices
$\sigma_{1}\dots\sigma_{n}$ provided there is an
$n$-simplex whose $[i-1,i]$ edge equals
$\sigma_{i}$ and whose $[0,n]$ edge equals
$\sigma$. Then $[\sigma]$ is equal to the
composition of $[\sigma_{1}],\dots,[\sigma_{n}]$
in the homotopy category
$\mathrm{h}\mathscr{C}$. We refer to this process
as \emph{decomposition of} $\sigma$.

Our first observation is rather basic, but useful
nonetheless:
\begin{lemma}\label{lemma:when-two-1-simplices-can-be-decomposed}
  Let $\varphi,\varphi'\in \mathscr{C}(Y)$ be two 1-simplices with the
  same source and target. The following are
  equivalent:
  \begin{enumerate}[label=(\alph*)]
  \item $\varphi,\varphi'$ lie in the same path component of $\mathrm{NN}(Y)$ with fixed endpoints,
  \item $[\varphi']=[\varphi]$ in the homotopy
    category $\mathrm{h}\mathscr{C}(Y)$;
  \item $\varphi'$ can be decomposed into the
    sequence $\varphi_{1}=\id$, $\varphi_{2}=\varphi$;
  \item $\varphi'$ can be decomposed into the
    sequence $\varphi_{1}=\varphi$, $\varphi_{2}=\id$;
  \end{enumerate}
  here $\id$ is the degenerate 1-simplex (the
  identity element).
\end{lemma}
\begin{proof}
  The equivalence of (b) through (d) can be found
  in \cite[\S1.2.3]{lurie-HTT-2009}. The
  equivalence of (a) and (d) follows from
  \S\ref{sec:equivalent-morphisms}.
\end{proof}

Mapping cones also interact nicely with decompositions:
\begin{lemma}
  Let $\mathscr{C}$ be an $\infty$-category and
  let
  $F:\mathscr{C}\to
  \mathrm{N}_{\mathrm{dg}}\mathrm{Ch}(\mathbf{k}[[x]])$
  be an $\infty$-functor. Suppose that $\sigma$ is
  the composition of $\sigma_{1},\dots,\sigma_{n}$
  and the homology of the mapping cones of the
  chain maps $F(\sigma_{i})$ is $x$-torsion for
  each $i$. Then the homology of the mapping cone
  of $F(\sigma)$ is also $x$-torsion.
\end{lemma}
\begin{proof}
  It suffices to prove the case $n=2$, by
  induction. Each of the three vertices of
  $\sigma$ is mapped by $F$ to a chain complex
  over $\mathbf{k}[[x]]$; let us call these
  $C_{0},C_{1},C_{2}$. The three faces are mapped
  to chain maps $f,g,h$. The equation for the
  $\infty$-functor:
  \begin{equation*}
    h-gf=kd+dk
  \end{equation*}
  where $k:C_{0}[1]\to C_{2}$ is the map
  associated to the 2-simplex. Throughout we use
  gradings modulo two.

  Introduce:
  \begin{equation*}
    \Gamma=C_{2}\oplus C_{0}[1]\quad\Delta=C_{1}\oplus C_{0}[1]\oplus C_{2}\oplus C_{1}[1]
  \end{equation*}
  with morphisms:
  \begin{equation*}
    d_{\Gamma}=\left[
      \begin{smallmatrix}
        d&h\\
        0&d
      \end{smallmatrix}
    \right]\quad d_{\Delta}=\left[
      \begin{smallmatrix}
        d&f&0&1\\
        0&d&0&0\\
        0&0&d&g\\
        0&0&0&d
      \end{smallmatrix}
    \right]\quad p=\left[
      \begin{smallmatrix}
        g&k&1&0\\
        0&1&0&0\\
      \end{smallmatrix}
    \right]\quad i=\left[
      \begin{smallmatrix}
        0&0\\
        0&1\\
        1&k\\
        0&f
      \end{smallmatrix}
    \right],    
  \end{equation*}
  Some care is needed to explain the sign
  conventions, but the basic rule is that the
  ``shifting'' symbol $[1]$ introduces $\pm$
  signs; we leave this subtle point to the
  reader and work modulo 2. We claim:
  \begin{itemize}
  \item $d_{\Gamma}^{2}=d_{\Delta}^{2}=0$,
  \item $p,i$ are chain maps.
  \item $p,i$ are chain homotopy inverses of each other.
  \end{itemize}
  The non-obvious part is that $ip$ is chain homotopic to $1$; for this we have:
  \begin{equation*}
    1-ip=\left[
      \begin{smallmatrix}
        1&0&0&0\\
        0&0&0&0\\
        g&0&0&0\\
        0&f&0&1
      \end{smallmatrix}
    \right]=\left[
      \begin{smallmatrix}
        d&f&0&1\\
        0&d&0&0\\
        0&0&d&g\\
        0&0&0&d
      \end{smallmatrix}
    \right]\left[
      \begin{smallmatrix}
        0&0&0&0\\
        0&0&0&0\\
        0&0&0&0\\
        1&0&0&0
      \end{smallmatrix}
    \right]+\left[
      \begin{smallmatrix}
        0&0&0&0\\
        0&0&0&0\\
        0&0&0&0\\
        1&0&0&0
      \end{smallmatrix}
    \right]\left[
      \begin{smallmatrix}
        d&f&0&1\\
        0&d&0&0\\
        0&0&d&g\\
        0&0&0&d
      \end{smallmatrix}
    \right]
  \end{equation*}
  Thus we have constructed a complex $\Delta$
  which is chain homotopy equivalent to $\Gamma$
  (the cone of $h$). On the other hand, $\Delta$
  has a two term filtration whose associated
  graded complex is isomorphic to the direct sum
  of the cone of $f$ and the cone of $g$. Thus, by
  a simple spectral sequence argument, if the
  cones of $f$ and $g$ are $x$-torsion, then so is
  the cone of $h$.
\end{proof}

Our next lemma gives an explicit formula
representing the composition in the case of $\mathscr{C}(Y)$:
\begin{lemma}\label{lemma:compo-concat}
  Let $\varphi_{1},\dots,\varphi_{n}$ be a
  composable sequence. Writing
  $\varphi_{i}=\varphi_{i,s,t}$ as an element in $\mathrm{NN}(Y)$, the $1$-simplex:
  \begin{equation}\label{eq:composition-formula}
    (\varphi_{n,s,t}\circ \varphi_{n,0,t}^{-1})\circ \dots \circ (\varphi_{1,s,t}\circ \varphi_{1,0,t}^{-1})\circ \varphi_{1,0,t}
  \end{equation}
  represents the composition of
  $\varphi_{1},\dots,\varphi_{n}$.
\end{lemma}
\begin{proof}
  It suffices to prove the case $n=2$, as the
  rest follows by induction. A two-simplex
  is the information of:
  \begin{itemize}
  \item $\Phi^{012}:[0,\infty)\to \mathrm{NN}(Y)$,
  \item $\Phi^{01},\Phi^{12},\Phi^{02}\in \mathrm{NN}(Y)$,
  \end{itemize}
  satisfying certain interrelations; see Definition
  \ref{definition:simplices}. In particular, we are seeking a map:
  \begin{equation*}
    \Phi^{012}:[0,\infty)\to \mathrm{NN}(Y)
  \end{equation*}
  such that
  \begin{itemize}
  \item $\Phi^{012}(0)=(s\mapsto \varphi_{2,s,t}\circ \varphi_{2,0,t}^{-1}\circ \varphi_{1,s,t})$,
  \item $\Phi^{012}(x)=\varphi_{2}\#_{\tau(x)}\varphi_{1}$ for $x$ large enough.
  \end{itemize}
  Finding such $\Phi^{012}$ is a standard
  affair of interpolating between timewise
  composition and concatenation. This can be
  achieved in a way compatible with the
  axioms in Definition
  \ref{definition:simplices}.
\end{proof}

\subsubsection{Isomorphisms in the homotopy
  category}
\label{sec:isom-homot-categ}

We characterize those $1$-simplices which project
to isomorphisms in the homotopy category.
\begin{lemma}\label{lemma:isom-homot-categ}
  The following are equivalent conditions on a
  $1$-simplex $\varphi$:
  \begin{enumerate}[label=(\alph*)]
  \item $[\varphi]$ is an isomorphism in the
    homotopy category;
  \item $\varphi_{s,t}$ satisfies
    $\varphi_{s,1}=\varphi_{0,1}$ for all
    $s\in \R$.
  \end{enumerate}
\end{lemma}
\begin{proof}
  That (b) implies (a) follows from the fact that
  $\varphi_{1-s,t}$ remains a well-defined
  $1$-simplex, and that the composition:
  \begin{equation*}
    \varphi_{1-s,t}\circ \varphi_{1,t}^{-1}\circ \varphi_{s,t}
  \end{equation*}
  is homotopic through non-negative paths to
  the identity morphism based at
  $\varphi_{0,t}$. The implication that (a)
  implies (b) is slightly more subtle, but
  follows from the same reasoning used in
  \S\ref{sec:complexity-mathscrcy} appealing
  to the non-existence of $C^{1}$ small loops
  of contactomorphisms.
\end{proof}

\subsubsection{Time reparametrization trick}
\label{sec:time-repar-trick}

In this subsection, we will use the explicit
formula for the composition
\eqref{eq:composition-formula} to decompose
morphisms into simpler pieces. Let us introduce
the following special class of $1$-simplices,
which have certain benefits vis-à-vis the
construction of extenders.

\begin{definition}\label{definition:M-AM}  
  A 1-simplex $\sigma_{s,t}$ is said to be of type
  {\rm\makeatletter\phantomsection\edef\@currentlabel{(M)}\label{type-M}{\rm(M)}\makeatother}
  provided that there is a positive isotopy
  $\psi_{s}$ and a contact isotopy
  $\varphi_{\tau}$ such that:\footnote{In the
    following $\beta$ is a standard cut-off
    function; one can take, e.g., $\beta(x)=f(x)$
    for $x\in [0,1]$ and $\beta'(x)=0$ for
    $x\not\in [0,1]$, where $f$ is as in Figure
    \ref{fig:graph-of-f}}
  \begin{equation*}
    \sigma_{s,t}=\psi_{s\beta(2t-1)}\circ \varphi_{\beta(2t)}=\left\{
      \begin{aligned}
        \varphi_{\beta(2t)}\text{ for }t\le 1/2\\
        \psi_{s\beta(2t-1)}\circ \varphi_{1}\text{ for }t\ge 1/2\\
      \end{aligned}
    \right.
  \end{equation*}
  If, in addition, $\psi_{s}$ is autonomous,
  we say that $\sigma_{s,t}$ is of type
  {\rm\makeatletter\phantomsection\edef\@currentlabel{(AM)}\label{type-AM}{(AM)}\makeatother}.
  We implicitly assume that endpoints do not
  lie on the discriminant, in order for
  $\sigma_{s,t}$ to be considered as a
  $1$-simplex in $\mathscr{C}(Y)$.
\end{definition}

Not every morphism $\sigma$ can be decomposed into
morphisms of type \ref{type-M} or type
\ref{type-AM}, because of the assumption that
$\psi$ must be positive. However, if we restrict
to those $\sigma$ which are positive then:
\begin{lemma}\label{lemma:can-be-decomposed}
  Any positive morphism $\sigma$ can be decomposed
  into a sequence
  $\sigma_{1},\sigma_{2},\sigma_{3}$ where
  $\sigma_{1}$ and $\sigma_{3}$ are isomorphisms
  and $\sigma_{2}$ is of type \ref{type-M}.
\end{lemma}
\begin{proof}
  Let us say that two $1$-simplices are equivalent $\simeq$ if they differ by right or left multiplication by isomorphisms. Then we claim:  
  \begin{itemize}
  \item $\sigma_{s,t}\simeq \sigma_{s,\beta(2t)}$,
  \item $\sigma_{s,\beta(2t)}\simeq \sigma_{s,\beta(2t-1)}\sigma_{0,\beta(2t-1)}^{-1}\sigma_{0,\beta(2t)}$,
  \item $\sigma_{s,\beta(2t-1)}\sigma_{0,\beta(2t-1)}^{-1}\sigma_{0,\beta(2t)}\simeq \sigma_{s\beta(2t-1),1}\sigma_{0,1}^{-1}\sigma_{0,\beta(2t)}$,
  \end{itemize}
  the crucial observation is that these
  homotopies are relative the sets $t=1$,
  $t=0$, which is what ensures the
  equivalence up to right or left
  muliplcation by isomorphisms (using Lemma
  \ref{lemma:isom-homot-categ}). The final
  morphism is of type~\ref{type-M}, provided
  we set
  $\psi_{s}=\sigma_{s,1}\circ
  \sigma_{0,1}^{-1}$ and
  $\varphi_{\tau}=\sigma_{0,\tau}$ to satisfy
  Definition \ref{definition:M-AM}.
\end{proof}

\subsubsection{Extender ansatz}
\label{sec:extender-ansatz}

Suppose that $\sigma_{s,t}$ is of type
\ref{type-M} for $\varphi_{\tau}$ and $\psi_{s}$
as in Definition \ref{definition:M-AM}. Then the
generator of $\sigma_{s,t}$ on the symplectization
$SY$ is:
\begin{itemize}
\item
  $\partial_{t}\sigma_{s,t}=X'_{s,t}\circ
  \sigma_{s,t}$ and $\lambda(X'_{s,t})=H'_{s,t}$;
\end{itemize}
here the ``prime'' notation signifies that these
are Liouville equivariant\footnote{Recall that
  here we mean these objects commute with the
  Liouville flow on the symplectization $SY$;
  Liouville equivariant Hamiltonian vector fields
  are uniquely determined by the ideal
  restrictions to contact vector fields of $Y$.}
for all $s,t$; we will deform $H_{s,t}'$ in the
subsequent discussion to construct an extender
$H_{s,t}$ satisfying the axioms in
\S\ref{sec:extenders}. We observe:
\begin{equation*}
  H_{s,t}'=2\beta'(2t)F_{\beta(2t)}+2s\beta'(2t-1)S_{s\beta(2t-1)},
\end{equation*}
where $S_{s}$ generates $\psi_{s}$ and $F_{\tau}$
generates $\varphi_{\tau}$.

Define the \emph{extender ansatz} by the formula:
\begin{equation}\label{eq:extender-ansatz}
  H_{s,t}=2\beta'(2t)F_{\beta(2t)}+2s\beta'(2t-1)\gamma(S_{s\beta(2t-1)})
\end{equation}
where $\gamma$ is the convex cut-off function
illustrated in Figure \ref{fig:convex-cut-off},
which is required to vanish in a neighborhood of
$0$.

\begin{figure}[h]
  \centering
  \begin{tikzpicture}
    \draw
    (-2,0)--(0,0)coordinate(X)to[out=0,in=-135](1,0.5)--+(1,1);
    \draw (X)node[below
    right]{$x=0$}node[draw,circle,fill,inner
    sep=1pt]{} +(-.5,-.5)--+(1.5,1.5);
  \end{tikzpicture}
  \caption{Convex cut-off function $\gamma(x)$
    used in the extender ansatz equals $0$ in a
    neighborhood of $x\le 0$ and equals
    $x-\epsilon$ for $x\ge 1$. Assume that
    $\gamma''(x)>0$ if $\gamma'(x)\in (0,1)$.}
  \label{fig:convex-cut-off}
\end{figure}
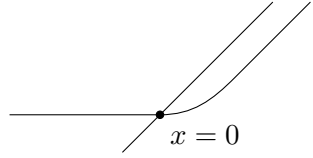

\begin{lemma}  
  The isotopy $\Psi_{s,t}$ generated by
  \eqref{eq:extender-ansatz} satisfies
  \ref{item:extender-1} through
  \ref{item:extender-5}, but maybe not axiom
  \ref{item:extender-6} (in other words, we do not
  show all the orbits of the extender ansatz have
  the same action).
\end{lemma}
\begin{proof}
  Recall that $K_{s,t}$ is determined from
  $H_{s,t}$ by the initial value problem:
  \begin{equation}\label{eq:curvature-IVP}
    \left\{
      \begin{aligned}
        &\partial_{t} K_{s,t}+\omega(X_{H_{s,t}},X_{K_{s,t}})=\partial_{s}H_{s,t},\\
        &K_{s,0}=0;
      \end{aligned}\right.
  \end{equation}
  see, e.g., \cite[pp.\,16]{cant-sh-barcode}. It
  follows that $K_{s,t}=0$ on the negative end
  where $H_{s,t}=H_{0,t}$. Thus
  \ref{item:extender-1} holds. The two axioms
  \ref{item:extender-2} and \ref{item:extender-3}
  are clear.

  We compute:
  \begin{equation*}
    \partial_{s}H_{s,t}=2\beta'(2t-1)[\gamma(S_{s\beta(2t-1)})+s\gamma'(S_{s\beta(2t-1)})D_{s\beta(2t-1)}\beta(2t-1)],
  \end{equation*}
  where $D_{s}=\partial_{s}S_{s}$. After some
  manipulation, this yields:
  \begin{equation*}
    \begin{aligned}
      \partial_{s}H_{s,t}=\partial_{t}(\beta(2t-1)\gamma(S_{s\beta(2t-1)})),
    \end{aligned}
  \end{equation*}
  and thus we can set:
  \begin{equation*}
    K_{s,t}=\beta(2t-1)\gamma(S_{s\beta(2t-1)})
  \end{equation*}
  to solve the IVP \eqref{eq:curvature-IVP}, as
  $\omega(X_{K_{s,t}},X_{H_{s,t}})=0$ (this is
  clear for $t\le 1/2$ as $K_{s,t}$ vanishes
  there, and it holds for $t\ge 1/2$ since then
  both vector fields are proportional to the
  vector field $X_{S_{s}}$). Then, by inspection
  it follows that $K_{s,1}=\gamma(S_{s})$, i.e.,
  \ref{item:extender-4} holds.

  The remaining axiom \ref{item:extender-5} is
  concerned with the actions of the orbits of
  $X_{H_{1,t}}$. If $H_{s,t}$ is given by
  \eqref{eq:extender-ansatz}, then we compute:
  \begin{equation*}
    \int_{0}^{1/2}H_{1,t}(y(t))dt-y^{*}\lambda =0,
  \end{equation*}
  provided that $y(t)=\Psi_{1,t}(y(0))$. This is
  because $H_{1,t}$ is 1-homogeneous with respect
  to the Liouville flow, for $t\le 1/2$. Moreover:
  \begin{equation}\label{eq:action-estimate}
    \int_{1/2}^{1}H_{1,t}(y(t))dt-y^{*}\lambda=\int_{0}^{1}\gamma(S_{s})-S_{s}\gamma'(S_{s})ds<0.
  \end{equation}
  where the equality follows from the change of
  variables for integrals, and the inequality
  follows from the fact that $\gamma(S_{s})>0$
  must hold somewhere along the orbit $y(t)$,
  otherwise $\varphi_{t}$ would lie on the
  discriminant; convexity then yields the inequality
  $\gamma(S_{s})-S_{s}\gamma'(S_{s})<0$.

  As a final remark, we note that compactness of
  the orbit set follows from the assumption that
  $\varphi_{1},\psi_{1}\varphi_{1}$ do not lie on
  the discriminant. This completes the proof.
\end{proof}

\subsubsection{The orbit set in the autonomous
  case}
\label{sec:orbit-set-autonomous}

Let $H_{s,t}$ be as in
\eqref{eq:extender-ansatz}. The final axiom
\ref{item:extender-6} concerns a global property
of the orbit set of the system generated by
$H_{1,t}$, namely, that all orbits have the same
action. This property is not automatic, but there
is a straightforward criterion in the case the
input data was of type \ref{type-AM}, i.e., in the
case $\psi_{\sigma}$ autonomous.

\begin{lemma}
  Suppose that there is a unique
  $\sigma_{0}\in (0,1)$ such that
  $\psi_{\sigma_{0}}\varphi_{1}$ lies on the
  discriminant, and let $H_{s,t}$ be the extender
  ansatz given in \eqref{eq:extender-ansatz} for
  data $\psi_{\sigma},\varphi_{\tau}$ of type
  \ref{type-AM}. Then $H_{s,t}$ satisfies the
  final axiom \ref{item:extender-6}.
\end{lemma}
\begin{proof}
  The argument is similar to many arguments in
  symplectic homology theory and its cousin
  Rabinowitz--Floer homology, in particular, the
  arguments used to analyze translated points. The
  key idea is that $S$ is a positive
  one-homogeneous Hamiltonian and can be used as a
  coordinate on the symplectization. We observe:
  \begin{equation*}
    H_{s,t}=2\beta'(2t)F_{\beta(2t)}+2s\beta'(2t-1)\gamma(S),
  \end{equation*}
  which can be solved explicitly:
  \begin{equation*}
    \Psi_{s,t}=\psi_{s\beta(2t-1)\gamma'(S)}\circ \varphi_{\beta(2t)},
  \end{equation*}
  here $\psi_{s\beta(2t-1)\gamma'(S)}$ is the
  diffeomorphism \emph{preserving the level sets
    of $S$} obtained by flowing by
  $2s\beta'(2t-1)\gamma'(S)X_{S}$ over the time
  interval $t\in [0,1]$. Therefore $\Psi_{1,1}$
  has a fixed point at $x$ if and only if:
  \begin{equation}\label{eq:slope-at-fixed-point}
    \gamma'(S(x))=\sigma_{0}.
  \end{equation}
  Since $\gamma$ is convex, and strictly convex
  when $\gamma'(S)=\sigma_{0}$, it holds that
  $S(x)$ is uniquely determined by $\sigma$ and
  \eqref{eq:slope-at-fixed-point}. As in
  \eqref{eq:action-estimate}, the action at such a
  fixed point is equal to
  $\gamma(S(x))-S(x)\gamma'(S(x))$. It follows
  there is a unique action value, i.e.,
  \ref{item:extender-6} follows.
\end{proof}

\subsubsection{Completing the proof of Theorem
  \ref{theorem:composition}}
\label{sec:comp-proof-thm-decomp}

Let us describe our strategy for achieving
\ref{item:extender-6} in general. Fix a positive
$1$-simplex $\sigma$; the first step is to use
Lemma \ref{lemma:can-be-decomposed} to decompose
$\sigma$ into $\sigma_{1},\sigma_{2},\sigma_{3}$
where $\sigma_{1},\sigma_{3}$ are isomorphisms in
$\mathscr{C}(Y)$ and $\sigma_{2}$ is of type
\ref{type-M}. For the purposes of proving Theorem
\ref{theorem:composition}, we can replace $\sigma$
by $\sigma_{2}$, and just suppose that $\sigma$
was \ref{type-M} from the start.

Let us therefore suppose that $\sigma$ is
type \ref{type-M} for inputs
$\psi_{s},\varphi_{\tau}$, and without loss
of generality, let us assume the length of
the Moore path is $1$.

We prove the following statement concerning
the generic intersections with the
discriminant:
\begin{lemma}\label{lemma:genericity-statement}
  Suppose that $\varphi_{1}$ and
  $\psi_{1}\varphi_{1}$ do not have discriminant
  points. For a generic perturbation of $\psi_{s}$
  relative its endpoints, we can ensure that:
  \begin{equation*}
    D=\set{(y,s):\psi_{s}\circ \varphi_{1}\text{ has $y$ as a discriminant point}}
  \end{equation*}
  is a finite set, and the projection map
  $(y,s)\in D\mapsto s\in (0,1)$ is injective.

  Moreover, we can assume that:
  \begin{equation}\label{eq:transversality}
    \mathrm{im}(\d\psi_{s}\circ \d\varphi_{1}-1)\text{ is transverse to }T_{s}
  \end{equation}
  at all fixed points of
  $\psi_{s}\circ \varphi_{1}$ in $SY$, where
  $T_{s}$ is the vector field generating the
  positive isotopy $\psi_{s}$ (the ideal
  restriction of $T_{s}$ is an $s$-dependent Reeb
  flow on $Y$).
\end{lemma}
The proof is given in
\S\ref{sec:gener-inters-with}. Let us briefly
comment on the conclusions:
\begin{itemize}
\item discriminant points, when they occur, are
  isolated in $Y\times (0,1)$;
\item condition \eqref{eq:transversality} says
  that the subspace
  $\mathrm{im}(\d \psi_{s}\circ \d\varphi_{1}-1)$
  in $TSY$ has codimension $1$ at any lift to
  $SY$ of a discriminant point in $Y$.
\end{itemize}
Because positivity of $\psi_{s}$ is an open
condition, we can assume that $\psi_{s}$ and
$\varphi_{1}$ satisfy the conclusions of Lemma
\ref{lemma:genericity-statement} by a small
perturbation of $\sigma$.

The third step is to ``chop up'' the $s$-interval
so as to isolate the discriminant points; this is
illustrated in Figure \ref{fig:chopping-up}. To be
clear, ``chopping up'' refers to the process of
decomposing
$[0,1]=[s_{0},s_{1}]\cup [s_{1},s_{2}]\cup \dots
[s_{n-1},s_{n}]$ with $0=s_{0}<\dots<s_{n}=1$ so
that:
\begin{itemize}
\item $\psi_{s_{j}}\varphi_{1}$ does not have
  discriminant points for $j=0,\dots,n$.
\end{itemize}
Then one considers the $1$-simplices
$\sigma^{j}_{s,t}$ given by:
\begin{equation}\label{eq:chop-up}
  \sigma_{j,s,t}=\sigma_{s_{j-1}+s(s_{j}-s_{j-1}),t}=\psi_{(s_{j-1}+s(s_{j}-s_{j-1}))\beta(2t-1)}\circ \varphi_{\beta(2t)}.
\end{equation}
It is clear that $\sigma$ can be decomposed into
$\sigma_{1},\dots,\sigma_{n}$. Observe that:
\begin{equation*}
  \psi_{(s(s_{j}-s_{j-1})\beta(2t-1)+s_{j-1}\beta(2t-\tau))}\circ \varphi_{\beta(2t)}.
\end{equation*}
is a homotopy relative the boundary of the square,
as $\tau\in [0,1]$, which agrees with
\eqref{eq:chop-up} when $\tau=1$, and which is of
type \ref{type-M} when $\tau=0$ with data:
\begin{enumerate}[label=(c\arabic*)]
\item\label{chop-up-1}
  $\varphi^{j}_{t}=\psi_{s_{j-1}t}\varphi_{t}$,
\item\label{chop-up-2}
  $\psi^{j}_{s}=\psi_{s(s_{j}-s_{j-1})+s_{j-1}}\circ \psi_{s_{j-1}}^{-1}$.
\end{enumerate}
Thus we can replace $\sigma$ by the type
\ref{type-M} data given by \ref{chop-up-1} and
\ref{chop-up-2}. By choosing the partition fine
enough, we can suppose that $\sigma_{s,1}$ has at
most one transverse intersection with the
discriminant.

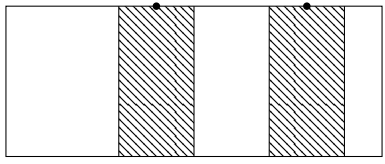
\begin{figure}[h]
  \centering
  \begin{tikzpicture}
    \draw (0,0) rectangle (5,2); \path[every
    node/.style={circle,fill,inner sep=1pt}]
    (2,2)node{} (4,2)node{}; \draw (2.5,0)--+(0,2)
    (1.5,0)--+(0,2) (4.5,0)--+(0,2)
    (3.5,0)--+(0,2); \fill[pattern=north west
    lines] (1.5,0) rectangle +(1,2) (3.5,0)
    rectangle +(1,2);
  \end{tikzpicture}
  \caption{Chopping up a 1-simplex into a
    concatenation so that each piece has either 0
    or 1 intersection with the discriminant; the
    shaded parts have 1 intersection with the
    discriminant, located at the midpoint.}
  \label{fig:chopping-up}
\end{figure}

The idea now is to pick the partition
$s_{0}<\dots<s_{n}$ so that $s_{j}-s_{j-1}$
becomes very small. Using this trick, we first
determine what happens if the case $\sigma_{s,1}$
has no intersections with the discriminant:
\begin{lemma}\label{lemma:case-no-intersections}
  Suppose that $\psi_{s}\varphi_{1}$ does not have
  discriminant points, for any value of
  $s\in [0,1]$. For a sufficiently fine partition
  $0=s_{0}<\dots<s_{n}=1$, the extender ansatz
  from \S\ref{sec:extender-ansatz} associated to
  the data given by \ref{chop-up-1} and
  \ref{chop-up-2} satisfies \ref{item:extender-6};
  moreover it has no 1-periodic orbits at all.

  Relatedly, if
  $t\mapsto
  \psi_{st}\varphi_{t}$ has no
  discriminant orbits in
  $\kappa\subset \pi_{0}(\Lambda Y)$, then the
  extender ansatz for \ref{chop-up-1} and
  \ref{chop-up-2} has no $1$-periodic orbits in
  $\kappa$ for sufficiently fine partition.
\end{lemma}
\begin{proof}
  Unpacking the definitions, we need to show that,
  for $s_{j}-s_{j-1}$ small enough, the system
  generated by:
  \begin{equation}\label{eq:extender-ansatz-j}
    H_{1,t}=2\beta'(2t)F^{j}_{\beta(2t)}+2\beta'(2t-1)\gamma(S^{j}_{\beta(2t-1)})
  \end{equation}
  has no $1$-periodic orbits, where:
  \begin{itemize}
  \item $F^{j}_{t}$ generates
    $\psi_{s_{j-1}t}\varphi_{t}$, and,
  \item $S^{j}_{s}$ generates
    $\psi_{s(s_{j}-s_{j-1})+s_{j-1}}\circ \psi_{s_{j-1}}^{-1}$.
  \end{itemize}
  In particular, $S_{s}^{j}$ is proportional to
  $s_{j}-s_{j-1}$ and hence converges in the
  $C^{\infty}$ topology to the zero function, as
  the partition gap tends to zero. The argument
  then follows from a standard compactness
  argument: since $F_{t}^{j}$ generates
  $\psi_{s_{j-1}t}\varphi_{t}$, the system
  generated by \eqref{eq:extender-ansatz-j} has a
  time-1 map which is close to
  $\psi_{s_{j-1}}\varphi_{1}$; since the latter
  has no orbits (uniformly in $s_{j-1}\in [0,1]$),
  and this is an open condition on the generating
  Hamiltonian, we conclude that
  \eqref{eq:extender-ansatz-j} defines a trivial
  extender, provided the partition is chosen small
  enough.

  The same argument applies verbatim to the
  modified statement when the free homotopy class
  $\kappa$ is included.
\end{proof}

\begin{lemma}\label{lemma:case-single-intersection}
  Suppose $\psi_{s}\varphi_{1}$ has exactly one
  discriminant point at $s=s_{0}$, and the
  intersection is transverse in the sense of
  \eqref{eq:transversality}. After localizing to a
  sufficiently small interval around $s=s_{0}$,
  the extender ansatz satisfies
  \ref{item:extender-6}.
\end{lemma}
\begin{proof}
  We will prove that, after a sufficiently small
  localization, the extender ansatz has a single
  $1$-periodic orbit. Evidently, this implies
  there is a single action value.

  Let us consider a sequence of localizations to
  the interval
  $[s_{0}-\epsilon_{n},s_{0}+\epsilon_{n}]$ of the
  extender ansatz:
  \begin{equation*}
    H^{n}_{1,t}=2\beta'(2t)F_{\beta(2t)}^{n}+2\beta'(2t-1)\gamma(S_{\beta(2t-1)}^{n}),
  \end{equation*}
  where $F_{t}^{n}$ generates
  $\psi_{(s_{0}-\epsilon_{n})t}\varphi_{t}$, and
  $S_{s}^{n}$ generates
  $\psi_{s2\epsilon_{n}+s_{0}-\epsilon_{n}}\circ
  \psi_{s_{0}-\epsilon_{n}}^{-1}$.

  Following the same argument used in the proof of
  Lemma \ref{lemma:case-no-intersections},
  $H_{1,t}^{n}$ generates a system $\xi_{n,t}$
  whose time-$1$ map $\xi_{n,1}$ is $C^{\infty}$
  close to $\psi_{s_{0}}\varphi_{1}$. In
  particular, if $\xi_{n,1}$ has a fixed point
  $z_{n}$, then $z_{n}$ must converge to some
  point on the Liouville flow line $Z_{\R}(y)$
  through the unique discriminant point $y$ of
  $\psi_{s_{0}}\varphi_{1}$.

  The time-$1$ map of $\xi_{n,t}$ can be written
  as a composition of two maps:
  \begin{itemize}
  \item $\psi_{s_{0}-\epsilon_{n}}\varphi_{1}$,
  \item the isotopy generated by
    $V_{n,t}=2\epsilon_{n}\gamma'(S^{n}_{t})T_{2\epsilon_{n}t+s_{0}-\epsilon_{n}}$.
  \end{itemize}
  Since $S_{t}^{n}\approx 2\epsilon_{n} r$, where
  $r$ is the Hamiltonian generating $T_{s_{0}}$,
  we conjugate with the Liouville flow
  $Z_{-\log(2\epsilon_{n})}$ so that
  $2\epsilon_{n}r$ converges to $r$; then the
  second part is approximately the flow by
  $2\epsilon_{n}\gamma'(r)T_{s_{0}}$, with error
  given by terms which are of order
  $\epsilon_{n}^{2}$.

  Similarly, up to terms of order
  $\epsilon_{n}^{2}$, the first part can be
  replaced by $\psi_{s_{0}}\varphi_{1}$ followed
  by the flow of $-\epsilon_{n}T_{s_{0}}$ for time
  $1$. Thus, in this analysis to first order, we
  have that $\xi_{n,1}$ is approximately:
  \begin{itemize}
  \item $\psi_{s_{0}}\varphi_{1}$, followed by,
  \item the time $1$ flow of
    $(2\gamma'(r)-1)\epsilon_{n}T_{s_{0}}$.
  \end{itemize}
  There is a unique fixed point located at
  $\gamma'(r)=1/2$ on the Liouville flow line
  $Z_{\R}(y)$. This is non-degenerate, by our
  transversality assumptions. This non-degeneracy
  ensures that the analysis up to first order in
  $\epsilon_{n}$ is stable under small
  perturbations of order $\epsilon_{n}^{2}$; a
  more precise argument uses the Banach fixed
  point theorem to prove the existence of a unique
  fixed point. Thus $\xi_{n,1}$ has a unique fixed
  point for $n$ sufficiently large, as desired.
\end{proof}

\subsection{On the generic intersections with the
  discriminant}
\label{sec:gener-inters-with}

The discriminant may be highly singular, but it
has a canonical resolution as a smooth (Fréchet)
manifold, as follows. Consider the graph map:
\[ \Gamma:\mathrm{Cont}(Y)\times SY \rightarrow SY
  \times SY \] sending $(\phi,y)$ to
$(y,\phi(y))$; this map is submersive (by the
abundance of contact Hamiltonians), and the
inverse image of the diagonal yields a Fréchet
manifold $\Gamma^{-1}(\Delta)$. The projection
(forgetting the $SY$ factor) yields a continuous
surjection:
\begin{equation}
  \Gamma^{-1}(\Delta)/\R\to \text{the discriminant}
\end{equation}
which is the aforementioned resolution; here the
quotient corresponds to the $\R$-action by the
Liouville flow on $SY$. This suggests the
following characterization\footnote{A toy model
  for discriminant points are critical points of a
  smooth function with critical value zero.} of
when a finite-dimensional family of
contactomorphisms should be called transverse to
the discriminant:
\begin{definition}
  Let \( P \) be a finite dimensional smooth
  manifold. We say that a smooth map
  \( f:P\to \mathrm{Cont}(Y) \) is
  \emph{transverse to the discriminant} if:
  \begin{equation*}
    \Gamma_{f}:(p,y)\in P\times SY\mapsto \Gamma(f(p),y)\in SY\times SY
  \end{equation*}
  is transverse to the diagonal $\Delta$. One can
  also speak of maps which are transverse along
  some closed subset $A\subset P$.
\end{definition}

\begin{proposition}\label{prop:trans1}
  Let \( P \) be a smooth, finite dimensional,
  compact manifold, and $f:P \to \mathrm{Cont}(Y)$
  a smooth map. Suppose \( f \) is transverse to
  the discriminant along some closed, possibly
  empty, subset \( A \subset P. \) Then \( f \) is
  homotopic relative \( A \) to a smooth map which
  is everywhere transverse to the discriminant and
  arbitrarily \( C^{\infty} \) close to \( f. \)
\end{proposition}
\begin{proof}
  This follows from the abundance of contact
  Hamiltonians and the Sard--Smale theorem. The
  details are left to the reader.
\end{proof}

By counting dimensions, one can also bound the
number of discriminant points which appear at any
given moment in generic $1$-parameter families of
contactomorphisms:
\begin{proposition}\label{prop:trans2}
  Suppose that $P$ is a one-dimensional manifold
  and consider a smooth map
  $f:P\to \mathrm{Cont}(Y)$ which is transverse to
  the discriminant. After a generic perturbation,
  there is at most one discriminant
  point\footnote{In fact, if $P$ has dimension
    $d$, then there are at most $d-1$ many
    discriminant points which occur above any
    given point $p\in P$, for generic $f$. We will
    not use this fact.} which occurs at any given
  point $p$.
\end{proposition}
\begin{proof}
  The subset $D=\Gamma_{f}^{-1}(\Delta)/\R$ is a
  zero dimensional submanifold of the total space
  $P\times Y$. Thus, for any given $p\in P$, there
  are a finite number of discriminant points, say
  $y_{1},\dots,y_{N}$. One perturbs $f$ in a
  neighborhood of each point $y_{i}$ in order to
  slightly move its basepoint $p$. The technical
  reason this is possible is that Sard Smale
  argument establishes the universal moduli space
  of all pairs $(f,p,y)$ such that $y$ is a
  discriminant point of $f(p)$ has a submersive
  projection $(f,p,y)\mapsto p$, (again by the
  abundance of contact Hamiltonians).
\end{proof}

We relate this notion of transversality to Lemma
\ref{lemma:genericity-statement}.
\begin{proposition}\label{prop:trans3}
  Suppose that
  $s\in [0,1]\mapsto \psi_{s}\circ \varphi_{1}$ is
  transverse to the discriminant, where $\psi_{s}$
  is positive isotopy. Then, for each $s$ such
  that $\psi_{s}\circ \varphi_{1}$ lies on the
  discriminant, it holds that
  $\mathrm{im}(\d\psi_{s}\circ \d\varphi_{1}-1)$
  is transverse to $T_{s}$ in $SY$, where $T_{s}$
  is the generator of $\psi_{s}$.
\end{proposition}
\begin{proof}
  By the above discussion, transversality ensures:
  \begin{equation*}
    (s,y)\in [0,1]\times SY\mapsto (y,\psi_{s}(\varphi_{1}(y)))
  \end{equation*}
  is transverse to the diagonal. Differentiating
  with respect to $y$ and $s$ at a fixed point yields the
  subspace of tangent vectors of the form
  \begin{equation*}
    v\oplus (\d\psi_{s}\circ \d\varphi_{1}(v)+T_{s}(y)).
  \end{equation*}
  If this is transverse to the diagonal, then we can solve:
  \begin{equation*}
    u\oplus u+v\oplus (\d\psi_{s}\circ \d\varphi_{1}(v)+T_{s}(y))=0\oplus w
  \end{equation*}
  for any $w$. The only option is $u=-v$, so:
  \begin{equation*}
    -v+\d\psi_{s}\circ \d\varphi_{1}(v)+T_{s}(y)=w
  \end{equation*}
  can be solved for any $w$, for some $v$. This
  yields the desired result.  
\end{proof}

Combining Propositions \ref{prop:trans1},
\ref{prop:trans2}, and \ref{prop:trans3} yields
Lemma
\ref{lemma:genericity-statement}.\hfill$\square$

\subsection{The local Floer cohomology of an
  extender}
\label{sec:local-floer-cohomology-of-an-extender}

In this section we develop the local Floer
cohomology of an extender.

\subsubsection{Admissible perturbations}
\label{sec:admiss-pert}

Let $\psi_{s,t}$ be an extender, as in
\S\ref{sec:extenders}. Fix a Liouville equivariant
$\omega$-tame almost complex structure $J$ on
$SY$.

An \emph{$\epsilon$-admissible perturbation}
for $\psi_{s,t},J$ is a family $\delta_{s,t}$ of
Hamiltonian diffeomorphisms such that:
\begin{itemize}
\item $\delta_{s,0}=\delta_{0,t}=\id$, for all
  $s,t$,
\item $\delta_{s,t}$ is compactly supported,
  uniformly in $s,t$,
\item $\psi_{1,t}\delta_{1,t}$ is a non-degenerate
  Hamiltonian system, for which all Floer
  differential cylinders are cut transversally,
\item the Hamiltonian generator of
  $t\mapsto \delta_{s,t}$, and its first
  derivative with respect to $s$, are bounded in
  absolute value by $\epsilon$.
\end{itemize}

As in the statement of Theorem
\ref{theorem:main-local-hf}, we let
$\kappa\subset \pi_{0}(\Lambda Y)$ be a collection
of free homotopy classes of loops.

\begin{lemma}\label{lemma:minimum-principle}
  If $\epsilon$ is small enough, then the Floer
  differential on:
  $$\mathrm{CF}(\psi_{1,t}\delta_{1,t},J,\kappa)$$
  squares to zero, and the chain homotopy type of
  the resulting complex is independent of the
  choice of $J$ or the $\epsilon$-admissible
  perturbation $\delta_{s,t}$.
\end{lemma}
\begin{proof}
  The Floer differential is defined as usual (as a
  sum of the index $1$ Floer differential
  cylinders, recording the input and output
  orbits, and the orientation line, as an
  endomorphism). To prove it squares to zero, we
  need only prove the required compactness result:
  \begin{itemize}
  \item \emph{any sequence of Floer cylinders
      $u_{n}$ remains in a compact subset of the
      symplectization $SY$}.
  \end{itemize}
  The other necessary a priori estimates on the
  derivatives of $u_{n}$, say, follow from the
  standard machinery once the above $C^{0}$ bound
  is proved.

  To analyze this compactness problem, we
  introduce the symplectization coordinate $r$. We
  assume that $\psi_{1,t}$ is Liouville
  equivariant, and $\delta_{s,t}=0$, on the ends
  $r\le 0$ and $r\ge r_{0}$, for some $r_{0}>0$,
  (positioning the left end at $r\le 0$ is without
  any loss of generality, as we can translate by
  the Liouville flow).

  Then, using the mean-value property for energy
  density as in \cite[Appendix
  B]{robbin-salamon-AIHP-AN-2001}, we conclude
  that the $C^{0}$ sizes of $\bd_{s}u_{n}$ and
  $\bd_{s}u_{n}-X_{t}(u_{n})$ must be small; here
  $X_{t}$ is the Hamiltonian vector field of
  $\psi_{1,t}\delta_{1,t}$. This is because the
  energy of Floer cylinders can be made
  arbitrarily small (by shrinking $\epsilon$).

  In particular, if $u_{n}$ fails to remain
  bounded, then one can find numbers $s_{n}$ so
  that subcylinders
  $u_{n}([s_{n}-1,s_{n}+1]\times \R/\Z)$ are
  mapped into the region $r\le 0$ (negative end),
  or the region $r\ge r_{0}$ (positive end). In
  either case, we obtain a contradiction (for
  $\epsilon$ small enough), because in these ends
  the Hamiltonian system
  $\psi_{1,t}\delta_{1,t}=\psi_{1,t}$ has no
  orbits, so these subcylinders occupy a minimum
  quantum of energy $\hbar$ (by a standard
  argument, e.g., \cite{salamon-notes-1997} or
  \cite[Proposition
  2.2.]{brocic-cant-JFPTA-2024}). The
  contradiction is ensured by picking $\epsilon$
  small enough.

  A similar argument works for proving a priori
  $C^{0}$-bounds on continuation cylinders
  relating two choices $J,\delta_{s,t}$ and
  $J',\delta'_{s,t}$, if both are
  $\epsilon$-admissible and $\epsilon$ is small
  enough. In this fashion we prove that the
  resulting chain complexes are chain homotopy
  equivalent (the quasi-isomorphisms are
  continuation maps, as in
  \cite{hofer-salamon-1995,salamon-notes-1997}).
\end{proof}

By virtue of the preceding result, the following
is well-defined:

\begin{definition}\label{definition:local-floer-cohomology-CHT}
  For any extender $\psi_{s,t}$ and any
  $\kappa\subset \pi_{0}(\Lambda Y)$, we define
  $\mathrm{CF}_{\mathrm{loc}}(\psi_{s,t},\kappa)$
  to be the chain homotopy type of
  $\mathrm{CF}(\psi_{1,t}\delta_{1,t},J,\kappa)$
  for any $J$ and any $\epsilon$-admissible
  $\delta_{s,t}$, for $\epsilon$ small enough.
\end{definition}

\subsubsection{Insertion of an extender into the filling}
\label{sec:insert-extend-into}

Let us present the $G$-filling $W$ of $Y$ as a
compact aspherical domain $\Omega$ with contact
type boundary, with the positive half of
$S\partial W$ attached as the convex end. This
presentation specifies a radial coordinate $r$ so
that $\partial \Omega=\set{r=1}$. We fix the
geometric set-up as shown in Figure
\ref{fig:geom-set-up-cone}.

Pick an extender $\psi_{s,t}$ whose ideal
restriction is the $1$-simplex $\varphi_{s,t}$,
and pick an $\epsilon$-admissible perturbation
$\delta_{s,t}$.

Fix regular Borel data $\psi_{\eta,0,t}$ whose
ideal restriction is $\varphi_{0,t}$. We can
suppose that the $t$-generator $X_{\eta,0,t}$ of
$\psi_{\eta,0,t}$ is Liouville equivariant outside
of $\Omega$. We can also suppose (by appropriate
choice of the initial radial coordinate) that all
orbits of $X_{v,0,t}$ remain entirely in $\Omega$,
when $v$ is any critical point on $EG$.

Introduce $X_{s,t}$ as the $t$-generator of the
perturbed extender $\psi_{s,t}\delta_{s,t}$; we
suppose that $X_{s,t}=X_{0,t}$ holds on $r\le 1$,
$X_{s,t}=X_{1,t}$ holds on $r\ge r_{0}$, and that
$\delta_{s,t}=\id$ holds for $r$ outside
$[1,r_{0}]$. This $X_{s,t}$ lifts from $SY$ to
$S\partial W$, as a family of $G$-invariant vector
fields.

This set-up ensures that the piecewise definition:
\begin{equation}\label{eq:piecewise-patch}
  X_{\eta,s,t}=\left\{
    \begin{aligned}
      &X_{\eta,0,t}\text{ inside }\Omega,\\
      &\d Z_{L}\circ X_{s,t}\circ Z_{-L}\text{ for }r\ge 1,
    \end{aligned}
  \right.
\end{equation}
is smooth and integrates to a path of Borel data
whose ideal restriction is $\varphi_{s,t}$
(throughout we always use the same $J$ in our
Borel data). In other words, we conjugate the
extender $\psi_{s,t}$ and the
$\epsilon$-admissible perturbation by Liouville
flow. This has the effect of exponentially scaling
the actions of the orbits arising from the
extender; importantly, orbits with negative action
will have a \emph{very} negative action when we
increase $L$.

\begin{claim}
  A generic choice of $X_{\eta,0,t}$ and
  $\delta_{s,t}$ ensures that
  \eqref{eq:piecewise-patch} is a regular path of
  Borel data.
\end{claim}
\begin{proof}
  The engine for ensuring regularity is the usual
  Sard-Smale argument of
  \cite{floer-hofer-salamon-duke-1995,
    mcduff-salamon-book-2012}.
  
  Any continuation cylinder (or Floer differential
  cylinder) which intersects $\Omega$ can be
  assumed to be regular by the generic choice of
  $X_{\eta,0,t}$. The other cylinders are
  contained entirely in the symplectization end;
  since the action is free on this part, the
  cylinders project to a cylinder in $SY$ for data
  determined by $X_{s,t}$. One ensures regularity
  for the projected cylinders by generic variation
  of $\delta_{s,t}$ (and the Sard-Smale
  theorem). The regularity of the original
  cylinders (before projection) follows since the
  projection map is a covering map (variations of
  the cylinder before projection correspond to
  variations of the cylinder after projection).
\end{proof}

To summarize, we have used the perturbed extender
to build an explicit regular path of Borel data
whose ideal restriction is the $1$-simplex under
consideration. The construction depends on the
extender $\psi_{s,t}$, the initial Borel datum
$X_{\eta,0,t}$, the parameter $L$, and the choice
of perturbation $\delta_{s,t}$.

\begin{figure}[h]
  \begin{tikzpicture}[scale=.7]
    \draw (0,1)--+(8,0) (0,0)--+(8,0)
    (0,0)to[out=180,in=0](-1,-1)to[out=180,in=180](-1,2)to[out=0,in=180](0,1);
    \draw (-1,1)to[out=-120,in=120](-1,0)
    (-1,0.8)to[out=-60,in=60](-1,0.2); \draw
    (1,0.5) circle (0.2 and 0.5)
    +(0,-0.5)node[below]{$r=1$} (4,0.5)
    circle (0.2 and
    0.5)+(0,-0.5)node[below]{$r=e^{L}\vphantom{e^{L}}$};
  \end{tikzpicture}
  \caption{Regions in the $G$-filling $W$ of
    $Y$. The length parameter $L$ should be
    considered as quite large.}
  \label{fig:geom-set-up-cone}
\end{figure}
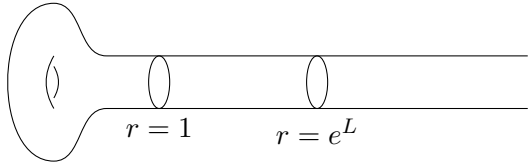

\subsubsection{Proof of Theorem
  \ref{theorem:main-local-hf}}
\label{sec:proof-theorem-main-local-hf}

In this section we prove that the cone of the
morphism
$\mathrm{CF}_{\mathrm{eq}}(\varphi_{0,t})\to
\mathrm{CF}_{\mathrm{eq}}(\varphi_{1,t})$
associated to the $1$-simplex $\varphi_{s,t}$
arising as the ideal restriction of the extender
$\psi_{s,t}$ lies in the local Floer cohomology
chain homotopy class
$\mathrm{CF}_{\mathrm{loc}}(\psi_{s,t},\kappa)$
from Definition
\ref{definition:local-floer-cohomology-CHT}, where
$\kappa$ is the collection of $W$-contractible
orbits.

We fix the geometric set-up from
\S\ref{sec:insert-extend-into}, yielding the
regular path of Borel data $X_{\eta,s,t}$. We will
now analyze the cone of the continuation map
associated to this regular path of Borel data.

The equivariant chain complex for the Borel data
determined by $X_{\eta,1,t}$ is generated as a
$\mathbf{k}[[x]]$-module by summands
$\mathfrak{o}(\gamma)\otimes \mathfrak{o}(v)$
where $v$ is a critical point of the
pseudogradient on $BG$ lying above the pole
$[1:0:\cdots]$, and $\gamma$ is a contractible
orbit of $X_{v,1,t}$, as we have explained already
in \S\ref{sec:defin-equiv-chain},
\S\ref{sec:case-p-3},
\S\ref{sec:defin-infin-funct}. These summands come
in two types:
\begin{enumerate}[label=(\alph*)]
\item\label{item:summand-a} summands where $\gamma$ is contained in
  $\Omega$, and is an orbit of $X_{0,1,t}$,
\item\label{item:summand-b} summands where
  $\gamma$ passes through $r\in [1,r_{0}]$, and
  projects to a $W$-contractible orbit for the
  vector field $X_{1,t}$ contributing to the local
  Floer cohomology of the extender $\psi_{s,t}$
  and the $\epsilon$-admissible perturbation
  $\delta_{s,t}$.
\end{enumerate}

On the other hand, the equivariant chain complex
for the Borel data determined by $X_{\eta,0,t}$ is
generated as a $\mathbf{k}[[x]]$-module by only
those summands of type \ref{item:summand-a}. The
summands of type \ref{item:summand-a} for
$\mathrm{CF}_{\mathrm{eq}}(X_{\eta,0,t})$ and
$\mathrm{CF}_{\mathrm{eq}}(X_{\eta,1,t})$ are
literally the same. In other words, there is a
decomposition:
\begin{equation}\label{eq:cone-decomposition}
  \mathrm{CF}_{\mathrm{eq}}(X_{\eta,1,t})=
  \mathrm{CF}_{\mathrm{eq}}(X_{\eta,0,t})\oplus E,
\end{equation}
where $E$ is the $\mathbf{k}[[x]]$-module
generated by the summands of type
\ref{item:summand-b}.
\begin{claim}\label{claim:cone-matrix}
  With respect to the direct sum decomposition
  \eqref{eq:cone-decomposition}, it holds that:
  \begin{equation*}
    d_{eq}=\left[
      \begin{matrix}
        d_{eq}&\Delta\\
        0&d_{E}
      \end{matrix}
    \right]\qquad
    \mathfrak{c}=
    \left[
      \begin{matrix}
        \id\\
        0
      \end{matrix}
    \right],
  \end{equation*}
  where $d_{eq}$ is the differential on
  $\mathrm{CF}_{\mathrm{eq}}(X_{\eta,1,t})$ and
  $\mathfrak{c}$ is the continuation
  map. Moreover, the term $d_{E}$ only counts
  cylinders which remain outside of $\Omega$.

  In the deduction, we allow the operations of
  increasing the length parameter $L$ in
  \S\ref{sec:insert-extend-into}, and picking the
  perturbation term $\delta_{s,t}$ smaller, if
  necessary.
\end{claim}
\begin{proof}
  By \ref{item:extender-5}, and the operation of
  increasing $L$ which scales actions
  exponentially, we may assume that all of the
  orbits in summand \ref{item:summand-b} have
  actions which are much more negative than the
  actions of the orbits in summand
  \ref{item:summand-a}. Even taking into account
  the curvature of Hamiltonian connections arising
  in the definition of the equivariant
  differential, we can then preclude the existence
  of cylinders with input in \ref{item:summand-a}
  and output in \ref{item:summand-b}, in the
  definition of $d_{eq}$. This proves the first
  part.

  Furthermore, by subsequently picking
  $\delta_{s,t}$ small enough, and using
  \ref{item:extender-6}, we may suppose that all
  orbits contributing to summand
  \ref{item:summand-b} have nearby actions; then,
  by the same energy estimate used in Lemma
  \ref{lemma:minimum-principle} we conclude that
  all cylinders contributing to $d_{E}$ remain in
  the complement of $\Omega$.

  For the conclusion about the continuation map
  $\mathfrak{c}$, we need to be careful that this
  ``increasing $L$'' operation does not
  simultaneously increase the curvature of the
  Hamiltonian connection used to define
  $\mathfrak{c}$ --- indeed, since the curvature
  is non-zero in the region where the extender is
  supported, there is a risk that we also scale
  the curvature exponentially when we conjugate by
  the Liouville flow. However, this is where we
  use assumption \ref{item:extender-4}; since
  $K_{s,1}$ is non-negative, it follows, from the
  same arguments used in the cited references in
  Remark \ref{remark:energy-estimates} (which we
  used to get the a priori energy estimates on the
  moduli spaces of \S\ref{sec:moduli-spaces}),
  that the curvature of the Hamiltonian connection
  used to define $\mathfrak{c}$ is non-positive
  everywhere on the region $r\ge 1$, and so
  exponential scaling will not introduce
  arbitrarily positive curvature.

  To conclude that the matrix entry in
  $\mathfrak{c}$ acting on the summands of type
  \ref{item:summand-a} is the identity, we use the
  standard fact (used in all Floer theoretic works
  concerning continuation maps) that continuation
  cylinders for $s$-independent continuation data
  which contribute to $\mathfrak{c}$ are only the
  stationary cylinders.
\end{proof}

\begin{corollary}
  Assume the set-up, conclusion, and notation of
  Claim \ref{claim:cone-matrix}. The cone of the
  continuation map $\mathfrak{c}$ is in the chain
  homotopy class of $(E,d_{E})$.
\end{corollary}
\begin{proof}
  This is standard homological algebra; the projection map:
  \begin{equation*}
    C\oplus E\oplus C[1]\to E
  \end{equation*}
  is the desired chain homotopy equivalence, when
  the domain is given the cone differential (with
  sign conventions due to supergradings left to
  the reader):
  \begin{equation*}
    \left[
      \begin{matrix}
        d_{C}&\Delta&\id\\
        0&d_{E}&0\\
        0&0&d_{C}
      \end{matrix}
    \right].
  \end{equation*}
  The inverse map sends $e\in E$ to $0\oplus e\oplus \Delta e$.
\end{proof}

To complete the proof of Theorem
\ref{theorem:main-local-hf}, we prove:
\begin{claim}
  Assume the set-up, conclusion, and notation of
  Claim \ref{claim:cone-matrix}. The complex
  $(E,d_{E})$ is in the chain homotopy class of
  $\mathrm{CF}_{\mathrm{loc}}(\psi_{s,t},\kappa)$.
\end{claim}
\begin{proof}
  This claim should be understood as the
  well-known ``Cartan isomorphism'' in equivariant
  cohomology, which asserts that the
  $G$-equivariant cohomology of a space with a
  free $G$-action is isomorphic to the ordinary
  cohomology of the quotient space. Versions of
  the Cartan isomorphism in equivariant Floer
  theory have appeared before, see, e.g.,
  \cite{sampietro-christ-2025}.

  For the Borel equivariant cohomology of spaces,
  the idea is the following, if a space $M$ admits
  a free action, then:
  \begin{equation*}
    M\times_{G}EG\to M/G,
  \end{equation*}
  sending $[(m,\eta)]$ to $[m]$ is a fibration
  with contractible fibers, and hence is a
  homotopy equivalence. The task of the present
  claim is to translate this into Floer theoretic
  language.

  Let us identify $E$ with the direct sum of generators:
  \begin{equation}\label{eq:strange-decomposition}
    E=\bigoplus \mathfrak{o}(\zeta)\otimes \mathfrak{o}(\eta)\otimes \mathbf{k}\simeq \mathrm{CF}(X_{1,t})\otimes \mathrm{CM}(EG)
  \end{equation}
  where $\eta$ is a critical point of the
  pseudogradient on $EG$ and $\zeta$ is a
  $W$-contractible orbit in $SY$. This is a small
  sleight of hand: this complex is identified with
  the usual complex by the identification:
  \begin{equation*}
    \mathfrak{o}(\zeta)\otimes \mathfrak{o}(x^{k}v_{i,g})= x^{k}\mathfrak{o}(g\gamma_{\zeta})\otimes \mathfrak{o}(v_{i}),
  \end{equation*}
  where we pick some distinguished lift
  $\gamma_{\zeta}$ to $S\bd W$ of each
  $W$-contractible orbit $\zeta$; here $v_{i,g}$
  is the lift of $v_{i}$ as indicated in Figure
  \ref{fig:mod-p-equivariant}. That this
  identification is well-defined uses the freeness
  of the action on $S\partial W$.
  
  Because the data is equivariant in the region
  $r\ge 1$, the trajectories $(\pi,u)$
  contributing to $d_{E}$ project to Floer
  cylinders in $SY$ and come in two types:
  \begin{itemize}
  \item $\pi$ has index $1$ and $u$ has index $0$,
  \item $\pi$ has index $0$ and $u$ has index $1$.
  \end{itemize}
  It follows that, with respect to the
  decomposition \eqref{eq:strange-decomposition},
  the differential is a tensor product
  differential on
  $\mathrm{CF}(X_{1,t})\otimes \mathrm{CM}(EG)$.

  Let $1$ denote the sum of the local minima in
  $\mathrm{CM}(EG)$ (in the notation of Figure
  \ref{fig:mod-p-equivariant}, this element $1$ is
  the sum $v_{0,g}$ as $g$ ranges over all
  elements of $G$), and consider the chain map:
  \begin{equation*}
    a\in \mathrm{CF}(X_{1,t})\mapsto a\otimes 1\in \mathrm{CF}(X_{1,t})\otimes
    \mathrm{CM}(EG).
  \end{equation*}
  The contractibility of $EG$ and standard Morse
  theoretic arguments imply that this is a chain
  homotopy equivalence, as desired.
\end{proof}

\subsubsection{Proof of Theorem \ref{theorem:zig-zag}}
\label{sec:proof-zig-zag}

The first part of the theorem follows from Lemma
\ref{lemma:case-no-intersections} and Theorem
\ref{theorem:main-local-hf}; briefly, one
decomposes the $1$-simplex into a composition of
$1$-simplices associated to extenders which have
zero local Floer cohomology in the free homotopy
class of $W$-contractible orbits.

The second part follows from a soft argument
involving approximating paths by zig-zags of
positive and negative paths; if we assume the
original path remains in the complement of the
$W$-contractible discriminant, then we can pick
the approximation so that each segment of the
zig-zag is also in the complement of the
$W$-contractible discriminant.\hfill$\square$

% 5
\section{Tying up loose ends}
\label{sec:tying-up-loose}

In this section, we prove various technical lemmas
needed to complete the proofs of the main
applications stated in
\S\ref{sec:statement-main-result}. Throughout we
assume that $Y$ admits a $G$-filling $W$, with
$G=\Z/p\Z$, as in \S\ref{sec:defin-meas}, and that
the $G$-action has at least one fixed point
$q_{0}$. At this stage, we have completed the
proof of our main structural theorems Theorems
\ref{theorem:main-infinity}, \ref{theorem:pss},
and \ref{theorem:zig-zag}.

\subsection{The unit element}
\label{sec:unit-element}

We prove a few lemmas about unit elements. Let us
briefly recall the construction; in
$\mathrm{CM}_{\mathrm{eq}}(X_{\eta})$, for any
Morse--Borel datum $X_{\eta}$ on $W$, there is a
well-defined cycle $1$ by summing up all local
minima. Then, for any $1$-simplex $\varphi_{s,t}$
in the PSS category $\mathscr{P}(Y)$ starting at
$\id$ and ending at
$\varphi_{1,t}\in \mathscr{C}(Y)$, one has a chain
map:
\begin{equation*}
  \mathrm{PSS}:\mathrm{CM}_{\mathrm{eq}}(X_{\eta})\to \mathrm{CF}_{\mathrm{eq}}(\varphi_{1,t}).
\end{equation*}
This gives a class
$1(\varphi_{s,t})\in
\mathrm{HF}_{\mathrm{eq}}(\varphi_{1,t})$, and
pushing forward by the colimit map
$\mathrm{HF}_{\mathrm{eq}}(\varphi_{1,t})\to
\mathrm{SH}_{\mathrm{eq}}(W)$ gives a class
$1(\varphi_{s,t})\in \mathrm{SH}_{\mathrm{eq}}$.

\subsubsection{The unit element is well-defined}
\label{sec:unit-element-well}

We prove Lemma
\ref{lemma:unit-well-defined-in-colimit} stating
that the unit element
$1(\varphi_{s,t})\in \mathrm{SH}_{\mathrm{eq}}(W)$
is well-defined independently of the $1$-simplex
$\varphi_{s,t}$ chosen.

\begin{proof}[Proof of Lemma \ref{lemma:unit-well-defined-in-colimit}]
  The key idea is the following: for any
  $1$-simplex $\varphi_{s,t}$ in $\mathscr{P}(Y)$
  starting at $\id$, there is a $2$-simplex whose
  $01$ face is $\varphi_{s,t}$ and whose $02$ face
  is a Reeb flow $R_{ast}$ for all sufficiently
  large speeds $a>0$.

  Indeed, if $R_{ast}\varphi_{s,t}^{-1}\varphi_{1,t}$ is
  positive (which it certainly is for $a$ large
  enough), then we can simply use Lemma
  \ref{lemma:compo-concat}.

  The result then follows, since the axioms of an
  $\infty$-functor:
  $$\mathscr{P}(Y)\to
  \mathrm{N}_{\mathrm{dg}}\mathrm{Ch}(\mathbf{k}[[x]])$$
  imply
  $1(R_{ast})=\mathfrak{c}(1(\varphi_{s,t}))$,
  where $\mathfrak{c}$ is the chain map associated
  to the $[1,2]$ edge. Passing to the colimit, we
  conclude $1(\varphi_{s,t})=1(R_{ast})$, for all
  sufficiently large $a$, for any $1$-simplex
  $\varphi_{s,t}$, and the desired result follows.
\end{proof}

\subsubsection{The unit element is not eternal}
\label{sec:unit-element-not}

We prove the assertion:
\begin{lemma}\label{lemma:unit-not-eternal}
  If $\mathrm{SH}_{\mathrm{eq}}(W)\neq 0$, the
  unit $1$ does not lie in the image of:
  $$\mathrm{HF}_{\mathrm{eq}}(R_{-\epsilon
    t}) \to \mathrm{SH}_{\mathrm{eq}}(W)$$ for any
  negative Reeb flow $R_{-\epsilon t}$.
\end{lemma}
\begin{proof}
  This is related to the results of
  \cite{ritter-jtopol-2013,
    ritter_negative_line_bundles,
    cant-hedicke-kilgore-arXiv-2023,
    cant-arXiv-2024,
    djordjevic-uljarevic-zhang-arXiv-2025} on
  non-existence of eternal classes in the
  symplectic cohomology of Liouville manifolds.

  Let us fix an equivariant Morse--Borel datum $X$
  on $W$ as in \S\ref{sec:equiv-pseud}, and
  consider the chain complexes:
  \begin{itemize}
  \item $\mathrm{CM}_{\mathrm{eq}}(X)$,
  \item
    $\mathrm{CM}_{\mathrm{neq}}(X)=\mathrm{CM}_{\mathrm{eq}}(X)\otimes_{\mathbf{k}[[x]]}\mathbf{k}$
    (\emph{the non-equivariant quotient}).
  \end{itemize}
  Since $\mathrm{HM}_{\mathrm{eq}}(X)$ is a
  finitely generated $\mathbf{k}[[x]]$-module, it fits into a canonical short exact sequence $0\to T\to \mathrm{HM}_{\mathrm{eq}}(X)\to F\to 0$,
  % \begin{equation*}
  %   \begin{tikzcd}
  %     0\arrow[r]&T\arrow[r]&\mathrm{HM}_{\mathrm{eq}}(X)\arrow[r]&F\arrow[r]&0
  %   \end{tikzcd}
  % \end{equation*}
  where $T$ is a torsion
  $\mathbf{k}[[x]]$-module and $F$ is a free
  $\mathbf{k}[[x]]$-module (by classification
  of modules over a principal ideal
  domain).
  \begin{claim}\label{subclaim:free-quot}
    The projection of the unit element to the
    free quotient $F$ cannot be written as
    $x a$ for any element $a$.
  \end{claim}
  \begin{proof}[Proof of Claim
    \ref{subclaim:free-quot}]
    This follows from a similar argument to
    the one used in the localization
    statement
    \S\ref{sec:proof-prop-not-torsion}; one
    can reduce the general case to the case
    $W=\mathrm{pt}$, by ``localizing'' at a
    local minimum $q_{0}$ on the fixed point
    submanifold. Here we need to again appeal
    to the fact that the characteristic of
    the coefficient field $\mathbf{k}$
    matches the order of the prime cyclic
    group $G$, to conclude that the
    projection of the unit to the free quotient
    cannot be written as $xa$, since in this
    case:
    \begin{equation}\label{eq:cohomology-of-BG}
      \mathrm{HM}^{*}_{\mathrm{eq}}(\mathrm{pt})\simeq \mathbf{k}[[x]]1\oplus \mathbf{k}[[x]]\theta,
    \end{equation}
    where $\theta$ is an element of degree
    $1$ (here we assume that $p\ge 3$, for
    simplicity, the result is similar and
    easier in the case $p=2$).
  \end{proof}
  
  Using Theorem \ref{theorem:commutative-square}, if $1$ lies in the image of
  $\mathrm{HF}_{\mathrm{eq}}(R_{-\epsilon
    t})\to \mathrm{SH}_{\mathrm{eq}}$,
  then there would be some element
  $b\in \mathrm{HM}_{\mathrm{eq}}(W)$ so that:
  \begin{enumerate}[label=(\alph*)]
  \item\label{item:go-to-zero} $1-b$ is mapped
    (via PSS and continuation) to zero in
    $\mathrm{SH}_{\mathrm{eq}}$,
  \item\label{item:inwards-outwards} $b$ is
    represented by a cycle lying in the image of
    the Morse continuation map
    $\mathrm{CM}_{\mathrm{eq}}(-X)\to
    \mathrm{CM}_{\mathrm{eq}}(X).$
  \end{enumerate}

  Next we use the product structure on the
  equivariant cohomology
  $\mathrm{HM}_{\mathrm{eq}}(W)$; we do not
  need to assume that this product structure
  respects the $\mathbf{k}[[x]]$-module
  structure, and so we can use the soft
  construction in
  \S\ref{sec:pair-pants-product} (below); in
  any case, we only need the product structure
  in Morse theory for the next steps.

  Observe that $(1-b)^{kp}=1-b^{kp}$ for any
  number $k$ (where $p$ is the prime
  characteristic), and so that
  \ref{item:go-to-zero} and
  \ref{item:inwards-outwards} still hold, with $b$
  replaced by $b^{kp}$.

  By taking $k$ large enough, we can then assume
  $b^{kp}$ is mapped to zero in the
  non-equivariant group
  $\mathrm{HM}_{\mathrm{neq}}(W)$ (the homology of
  the complex
  $\mathrm{CM}_{\mathrm{neq}}(X)$). This is
  because every cycle lying in the image of
  $\mathrm{CM}_{\mathrm{neq}}(-X)$ is
  nilpotent\footnote{It is important here that
    $\omega$ vanishes on holomorphic spheres; see
    \cite[\S1.7]{cant-arXiv-2024}. Otherwise
    quantum corrections can ruin this argument.},
  for degree reasons; this is same argument as
  \cite[\S2.3.3]{cant-hedicke-kilgore-arXiv-2023}.

  In particular, it follows by the long-exact
  sequence:
  \begin{equation*}
    \mathrm{HM}_{\mathrm{eq}}(W)\to \mathrm{HM}_{\mathrm{eq}}(W)\to \mathrm{HM}_{\mathrm{neq}}(W),
  \end{equation*}
  that $b^{kp}=xa$ for some
  $a\in \mathrm{HM}_{\mathrm{eq}}(W)$. Thus
  $1-xa$ is mapped to zero in
  $\mathrm{SH}_{\mathrm{eq}}(W)$, via PSS and
  continuation.

  The final step is to pass to the free
  quotient using:
  \begin{claim}\label{claim:injective}
    The free quotient of
    $\mathrm{HM}_{\mathrm{eq}}(W)\simeq
    \mathrm{HF}_{\mathrm{eq}}(R_{\epsilon
      t})$ is mapped injectively into the
    free quotient of
    $\mathrm{SH}_{\mathrm{eq}}(W)$ (the
    quotient by the torsion submodule).
  \end{claim}
  \begin{proof}[Proof of Claim \ref{claim:injective}]
    Suppose that $\zeta\in \mathrm{HF}_{\mathrm{eq}}(R_{\epsilon t})$ is not torsion; consider the long exact sequence of the cone:
    \begin{equation*}
      \dots \to \mathrm{Cone}\to \mathrm{HF}_{\mathrm{eq}}(R_{\epsilon t})\to \mathrm{HF}_{\mathrm{eq}}(R_{s t})\to \mathrm{Cone}\to \mathrm{HF}_{\mathrm{eq}}(R_{\epsilon t})\to \dots.
    \end{equation*}
    Suppose that $\zeta$ becomes torsion
    after continuation to $s$. Then
    $x^{k}\zeta$ is sent to zero for some
    power $k$ large enough. Thus $x^{k}\zeta$
    lies in the image of the map from Cone to
    $\mathrm{HF}_{\mathrm{eq}}(R_{\epsilon
      t})$. But the $x^{k+\ell}\zeta=0$ for
    some further power $\ell$, because all
    cones are torsion (by Theorem
    \ref{theorem:main-infinity}). This
    contradicts the hypothesis that $\zeta$
    is not torsion, and proves the claim.
  \end{proof}
  
  Then, combining \ref{claim:injective} and
  \ref{subclaim:free-quot}, we conclude
  $1-xa$ must inject into the free quotient
  of $\mathrm{SH}_{\mathrm{eq}}(W)$,
  contradicting the earlier deduction that
  $1-xa$ was sent to zero. This completes the
  proof of the lemma.
\end{proof}

\subsection{Pair of pants product}
\label{sec:pair-pants-product}

In this section, we briefly outline the
construction of the pair-of-pants product on
equivariant Floer cohomology, which is required to
establish the sub-additivity property
\ref{item:R-sub-additivity} in
\S\ref{sec:sub-additivity}. Due to the technical
nature of this construction, the similarity with
the main results of \cite{cant-arXiv-2024}, and
the existence of similar equivariant product
structures in
\cite{gonzalez-mak-pomerleano-arXiv-2023}, we do
not give a detailed construction.

The structural theorem we claim is:
\begin{theorem}\label{theorem:product-structure}
  Assume the setup and conclusion of Theorem
  \ref{theorem:main-infinity}. Then there exists
  a natural transformation $\Pi$ between the two
  functors:
  \begin{itemize}
  \item
    $\varphi_{0,t},\varphi_{1,t}\in
    \mathrm{h}\mathscr{C}\times
    \mathrm{h}\mathscr{C}\mapsto
    \mathrm{HF}_{\mathrm{eq}}(\varphi_{0,t})\otimes
    \mathrm{HF}_{\mathrm{eq}}(\varphi_{1,t})$,
  \item constant functor
    $\mathrm{h}\mathscr{C}\times\mathrm{h}\mathscr{C}\to
    \mathrm{SH}_{\mathrm{eq}}$,
  \end{itemize}
  and $\Pi_{\varphi_{0,t},\varphi_{1,t}}$ factors
  through the natural map
  $\mathrm{HF}_{\mathrm{eq}}(\varphi_{0,t}\varphi_{1,t})\to
  \mathrm{SH}_{\mathrm{eq}}$.

  Furthermore, the induced product:
  \begin{equation*}
    \Pi:\mathrm{SH}_{\mathrm{eq}}\otimes \mathrm{SH}_{\mathrm{eq}}\to \mathrm{SH}_{\mathrm{eq}}
  \end{equation*}
  is unital with $1$ acting as the unit element.
\end{theorem}
The construction is similar to the non-equivariant
approach detailed in \cite[\S3]{cant-arXiv-2024},
adapted to the equivariant setting, and with
coefficients over a field $\mathbf{k}$ with
arbitrary characteristic.

The generalization from ``non-equivariant'' to
``equivariant'' is akin to the generalization of
``non-family Floer cohomology'' to ``family Floer
cohomology'' in the sense of
\cite{hutchings-agt-2008}. In this context, we
refer the reader to the discussion in
\cite[\S3.6.3]{brocic-cant-arXiv-2025} which
constructs (in an unsurprising manner) a product
structure on some version of family Floer
cohomology. See
\cite{gonzalez-mak-pomerleano-arXiv-2023} for an
approach in this vein.

Let $\Sigma = \mathbb{C} \setminus \{0, 1\}$ be
the pair-of-pants surface equipped with two
positive cylindrical inputs $C_0, C_1$ and one
negative output $C_\infty$. For a triple of
contact isotopies
$\varphi_{0,t}, \varphi_{1,t}, \varphi_{\infty,t}$
such that
$\varphi_{0,t} \circ \varphi_{1,t} \le
\varphi_{\infty,t}$, we will consider a suitable space
of Hamiltonian connections on $SY \times \Sigma$
whose ideal restrictions match the respective
isotopies at the cylindrical ends. These Hamiltonian connections are constructed exactly as in \cite{cant-arXiv-2024}.

These lift to $G$-invariant Hamiltonian
connections on $S\partial W\times \Sigma$. The
subtle part is how to extend these to the compact
part of $W$ in a way which is compatible with
Borel data $\varphi_{i,\eta,t}$ extending
$\varphi_{i,t}$.

We do not really want to get into a technical
discussion of Hamiltonian connections on the pair
of pants surface (for this, we refer to the cited
references) but let us just say that:
\begin{itemize}
\item one can speak of families of Hamiltonian
  connections $\mathfrak{H}_{\eta}$ on
  $\Sigma\times W$, parametrized by points
  $\eta\in EG$, which agree with the connections
  determined by Borel data $\varphi_{i,\eta,t}$ in
  the ends, and whose ideal restriction are the
  $G$-invariant connections described above;
\item given a function $f:\Sigma\mapsto EG$, it
  makes sense to speak of the Hamiltonian
  connection obtained by setting $\eta=f(z)$ (in a
  similar manner to how we set $\eta=\pi(s)$ when
  defining the equivariant operations on
  cylinders).
\end{itemize}

Just as the equivariant operations in
\S\ref{sec:defin-equiv-chain}, etc, were defined
using flow lines of a pseudogradient on $BG$, the
product we are describing will be defined in terms
of \emph{Morse flow trees} on $BG$. There are
various approaches to this problem. For our
purposes, let us just suppose there is a
well-defined space $\mathscr{T}$ of trees $\tau$
which have evaluation maps $\tau:T\to BG$ and
which agree with flow lines of our pseudogradients
on the legs of the flow tree. Let us denote by
$\bar{\tau}$ the lift of a flow tree to $EG$. Here
$T$ is the underlying space of a flow tree (i.e.,
three copies of $[0,\infty)$ connected at the
vertex).

\begin{figure}[h]
  \centering
  \begin{tikzpicture}[scale=.5]
    \draw[draw=none]
    (-2,0)coordinate(Z0)--(0,0)coordinate(Z1)--(2,0)--+(1,1)coordinate(X0)--+(3,1)coordinate(X1)--+(1,-1)coordinate(Y0)--+(3,-1)coordinate(Y1);

    \draw (Z0) circle (0.2 and 0.4) (Z1) circle
    (0.2 and 0.4) (X0) circle (0.2 and 0.4) (X1)
    circle (0.2 and 0.4) (Y0) circle (0.2 and 0.4)
    (Y1) circle (0.2 and 0.4); \foreach \x in
    {X,Y,Z} { \draw
      (\x0)+(0,0.4)coordinate(\x2)--coordinate(\x6)+(2,0.4)coordinate(\x3)
      +(0,-0.4)coordinate(\x4)--+(2,-0.4)coordinate(\x5);
    } \draw (Z3)to[out=0,in=180](X2)
    (X4)to[out=180,in=180](Y2)
    (Z5)to[out=0,in=180](Y4);

    \node at (X6) [above] {$C_{1}$}; \node at (Y6)
    [above] {$C_{0}$}; \node at (Z6) [above]
    {$C_{\infty}$};

    \begin{scope}[shift={(10,0)}]
      \draw (0,0)to[out=60,in=180](3,1) (0,0)to[out=-60,in=180](3,-1) (0,0)--(-3,0);
    \end{scope}
    
  \end{tikzpicture}
  \caption{Pair of pants surface and the flow tree $T$.}
  \label{fig:pop}
\end{figure}
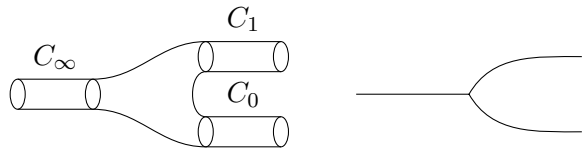

The moduli space used to define the product
consists of rigid pairs $(\bar{\tau}, u)$, where
$\tau\in \mathscr{T}$, and $u: \Sigma \to W$ is a
solution to the Floer equation associated to the
Hamiltonian connection $\mathfrak{H}_{\eta}$
obtained by setting $\eta=\bar{\tau}(p(z))$ where
$p:\Sigma\to T$ is an appropriate map sending the
pair of pants surface onto the trivalent
graph. This moduli space carries a $G$-action, and
we count the rigid $G$-orbits. The asymptotics are
interpreted using the same language of
``distinguished lifts'' common to all of our
equivariant operations.

To make a long story short, this induces a chain
map which descends to a product on homology:
\begin{equation*}
  \ast: \mathrm{HF}_{\mathrm{eq}}(\varphi_{0,t}) \otimes \mathrm{HF}_{\mathrm{eq}}(\varphi_{1,t}) \to \mathrm{HF}_{\mathrm{eq}}(\varphi_{\infty,t}).
\end{equation*}
This a priori depends on the choice of Hamiltonian
connections $\mathfrak{H}_{\eta}$ used (unlike the
setting of Borel data, it is not clear whether the
space of Hamiltonian connections on $SY$ with
non-positive curvature behaves as a contractible
space). However, the same arguments used in
\cite{cant-arXiv-2024} show that the map obtained
by post-composing $\ast$ with the colimit map
yield a well-defined product
$\Pi_{\varphi_{0,t},\varphi_{1,t}}$ as in Theorem
\ref{theorem:product-structure}.

By standard arguments, this operation satisfies
two properties:
\begin{enumerate}
\item The product commutes with the continuation
  maps.
\item The element
  $1 \in \mathrm{SH}_{\mathrm{eq}}(W)$ from in
  \S\ref{sec:unit-element} is the unit for this
  product.
\end{enumerate}
These ensure that if the colimit maps for
$\varphi_{0,t},\varphi_{1,t}$ hit the unit, then
the colimit map for $\varphi_{\infty,t}$ also hits
the unit. This, as in \cite{cant-arXiv-2024}, is
the mechanism used to verify the sub-additivity of
the spectral invariants
\ref{item:R-sub-additivity}.

\begin{remark}
  What we \emph{do not} verify is that the
  resulting product structure respect the
  $\mathbf{k}[[x]]$-module structure. This seems
  to be a subtle point. We do believe the product
  can be made to be $\mathbf{k}[[x]]$-linear on
  the level of homology groups. Whether this can
  be done on chain level seems far less
  certain. We leave this question for future
  research. If one proves the product structure is
  $\mathbf{k}[[x]]$-linear, then the invariants of
  type $\mu$ will be super-additive.
\end{remark}

\subsection{The axioms for the spectral
  invariants}
\label{sec:axioms-spectr-invar}

We establish the axioms \ref{item:R-spectrality}
through \ref{item:R-sub-additivity} for the
spectral invariants $c_{R}$ defined in
\eqref{eq:defin-spec-inv} in
\S\ref{sec:defin-meas}.

\subsubsection{Spectrality}
\label{sec:spectrality}

Property \ref{item:R-spectrality} follows from
Theorem \ref{theorem:zig-zag} and the fact the
unit is not eternal
\S\ref{sec:unit-element-not}. Indeed, the unit not
being eternal implies $c_{R}(\varphi_{t})\in
\R$. Having established this finitness, if $s$ is
not in the spectrum, then $s$ cannot be the
infimal value for which
$\mathrm{HF}(\varphi_{t}^{-1}\circ R_{st})\to
\mathrm{SH}_{\mathrm{eq}}(W)$ hits the unit
element, since the continuation maps for slightly
lower values of $s$ are isomorphisms by Theorem
\ref{theorem:zig-zag}.

Moreover, since the statement of Theorem
\ref{theorem:zig-zag} includes the refinement by
the $W$-contractible discriminant, we conclude
that

\subsubsection{Monotonicity}
\label{sec:monotonicity}

Property \ref{item:R-monotonicity} follows from
the fact that, if
$\varphi_{0,t}\le \varphi_{1,t}$, then there is a
1-simplex
$R_{st}\circ \varphi_{1,t}^{-1}\to R_{st}\circ
\varphi_{0,t}^{-1},$ for any value of $s$.

\subsubsection{Continuity from above}
\label{sec:cont-from-above}

Property \ref{item:R-continuity-from-above}
follows by definition when computing
$c_{R}(\varphi_{t})$ when $\varphi_{t}$ lies on
the discriminant. Otherwise it follows from
Theorem \ref{theorem:zig-zag}.

\subsubsection{Normalization}
\label{sec:normalization}

Property \ref{item:R-normalization} follows from
Lemma \ref{lemma:unit-not-eternal} in
\S\ref{sec:unit-element-not}.

\subsubsection{Sub-additivity}
\label{sec:sub-additivity}

Property \ref{item:R-sub-additivity} follows from
the existence of the pair of pants product
outlined in \S\ref{sec:pair-pants-product}; once
the product structure is set-up, the argument is
exactly the same as one used to prove
\cite[Theorem~4]{cant-arXiv-2024} and
\cite[Theorem~2.11]{djordjevic-uljarevic-zhang-arXiv-2025}.

\subsection{The axioms for the integer-valued
  measurements}
\label{sec:axioms-integ-valu}

In this section, we establish the axioms
\ref{item:G-monotonicity} through
\ref{item:G-discriminant} for the measurement
$\mu$ in \S\ref{sec:defin-meas}. We also prove the
implication \eqref{eq:implication} concerning the
vanishing of $\mathrm{SH}_{\mathrm{neq}}(W)$, and
Proposition \ref{prop:agrees-with-CZ} relating the
values of $\mu$ on linear symplectic isotopies
with the Conley-Zehnder index.

\subsubsection{On the non-equivariant symplectic
  cohomology}
\label{sec:non-equiv-sympl}

In this section we prove implication
\eqref{eq:implication}. Observe that for a Borel
data $(\psi_{\eta,t},J)$ it holds that:
\begin{equation}\label{eq:cf-neq}
  \mathrm{CF}_{\mathrm{neq}}(\psi_{\eta,t},J)=
  \left\{
    \begin{aligned}
      &\mathrm{CF}(\psi_{v_{0,e}})\otimes \mathfrak{o}(v_{0})\oplus \mathrm{CF}(\psi_{v_{1,e}})\otimes \mathfrak{o}(v_{1})&&\text{if }p\ge 3,\\
      &\mathrm{CF}(\psi_{v_{0,+}})&&\text{if }p=2,
    \end{aligned}
  \right.
\end{equation}
where the differential only counts the flow lines
which lie above the pole $[1:0:\cdots]$ in
projective space (either $\mathbb{C}P^{\infty}$ or
$\mathbb{R}P^{\infty}$, depending on $p\ge
3$). See \S\ref{sec:defin-infin-funct} and
\S\ref{sec:defin-borel-equiv} for details on the
chain complex.

It follows fairly tautologically that
$\mathrm{SH}_{\mathrm{neq}}(W)=\mathrm{SH}(W;\mathbf{k})$
in the case $p=2$, so we henceforth assume
$p\ge 3$.

Recall $\mathrm{SH}_{\mathrm{neq}}(W)$ is the
colimit of
$\mathrm{HF}_{\mathrm{neq}}(\psi_{\eta,t},J)$, and
so we need to prove this colimit vanishes under
the assumption that the ``ordinary'' symplectic
cohomology $\mathrm{SH}(W;\mathbf{k})$
vanishes. We will use a spectral sequence
argument.

With respect to the direct sum decomposition in
\eqref{eq:cf-neq} we can write the differential on
$\mathrm{CF}_{\mathrm{neq}}(\psi_{\eta,t},J)$ as a
lower triangular matrix, where the diagonal
entries are the ordinary Floer differentials.

Now, for any element
$e\in \mathrm{SH}_{\mathrm{neq}}(W)$, we can find
some $\psi_{\eta,t}$ so that $e$ lies in the image
of the continuation map from
$\mathrm{HF}_{\mathrm{neq}}(\psi_{\eta,t},J)$. This
element $e$ then splits into a summand
$e_{0}+e_{1}$. Since the element $e_{1}$ is a
cycle, and the ordinary symplectic cohomology
vanishes, we may assume (by replacing
$\psi_{\eta,t}$ by something closer to the
colimit) that $e_{0}$ is an exact cycle, say
$e_{0}=df_{0}$. Then $e$ is cohomologous to
$e_{1}'=e_{1}+\Delta f_{0}$, where $\Delta$ is the
off-diagonal term in the differential. But now
$e_{1}'$ is a cycle. By passing further into the
colimit \emph{and using the fact that the
  continuation maps are also lower triangular}, we
may suppose $e_{1}'$ is an exact cycle. But thus
the entire cycle $e$ is exact (up to passing
closer to the colimit). Since $e$ was an arbitrary
element of the colimit, we conclude
$\mathrm{SH}_{\mathrm{neq}}(W)=0$, as desired.

\subsubsection{Integer valuedness}
\label{sec:integer-valuedness}

Since $\mu(\varphi_{t})$ is defined as a supremum
of the set of integers $d$ for which $x^{-d}1$
lies in the image of the colimit map, to prove
$\mu(\varphi_{t})$ is finite, it is sufficient (to
obtain the integer valuedness of $\mu$) to prove
that this set of integers is non-empty and bounded
from above.

That the set of integers is bounded from above
follows from the structure theorem for finitely
generated modules over the ring
$\mathbf{k}[[x]]$. We argue by contradiction, and
suppose
$\mathrm{HF}_{\mathrm{eq}}(\varphi_{t})\to
\mathrm{SH}_{\mathrm{eq}}(W)$ can hit $x^{-d}1$
for arbitrarily large values of $d$. Since $x$
acts invertibly on $\mathrm{SH}_{\mathrm{eq}}(W)$,
the $x$-torsion part of
$\mathrm{HF}_{\mathrm{eq}}(\varphi_{t})$ is mapped
to zero under the colimit map. On the other hand,
since the cones of continuation maps are
$x$-torsion, the free part of
$\mathrm{HF}_{\mathrm{eq}}(\varphi_{t})$ is mapped
injectively into
$\mathrm{SH}_{\mathrm{eq}}(\varphi_{t})$. By our
assumption, there exists $a_{0},a_{d}$ in the free
part such that:
\begin{itemize}
\item $a_{0}\mapsto 1$
\item $a_{d}\mapsto x^{-d}1$.
\end{itemize}
But then $a_{0}=x^{d}a_{d}$ by the aforementioned
injectivity. In particular, $a_{0}$ can be divided
by $x^{d}$ for arbitrarily large $d$. In a
finitely generated $\mathbf{k}[[x]]$-module, this
can only happen if $a_{0}=0$, which implies
$1\in \mathrm{SH}_{\mathrm{eq}}(W)$ vanishes,
contradicting Lemma \ref{lemma:unit-not-eternal}
that the unit was not eternal.

To prove the set of integers is non-empty (i.e.,
$\mu$ is not $-\infty$), we need to prove that
\emph{some} power of the unit is hit. However,
since the cones of continuation maps are torsion,
the map
$\mathrm{HF}_{\mathrm{eq}}(R_{-at})\to
\mathrm{HF}_{\mathrm{eq}}(R_{\epsilon t})$ hits
$x^{k}1$ for large enough $k$, for any $a>0$. Thus
$\mu(R_{-at})$ is at least $-k$. By monotonicity
(see \S\ref{sec:monotonicity-1}) it follows that
$\mu(\varphi_{t})$ is at least $-k$ if
$\varphi_{t}$ is greater than $R_{-at}$. Since
$a>0$ was arbitrary, we conclude the desired
result.

\subsubsection{Monotonicity}
\label{sec:monotonicity-1}

Property \ref{item:G-monotonicity} follows from
the fact that, if
$\varphi_{0,t}\le \varphi_{1,t}$, then image of
the map
$\mathrm{HF}_{\mathrm{eq}}(\varphi_{0,t})\to
\mathrm{SH}_{\mathrm{eq}}(W)$ is contained in the
image of the map
$\mathrm{HF}_{\mathrm{eq}}(\varphi_{1,t})\to
\mathrm{SH}_{\mathrm{eq}}(W)$.

In particular, if
$\mathrm{HF}_{\mathrm{eq}}(\varphi_{0,t})\to
\mathrm{SH}_{\mathrm{eq}}(W)$ hits $x^{-d}1$, then
the same holds for $\varphi_{1,t}$; thus
$\mu(\varphi_{1,t})\ge \mu(\varphi_{0,t})$.

\subsubsection{Continuity from above}
\label{sec:cont-from-above-1}

For \ref{item:G-continuity} we use the same
argument as \ref{sec:cont-from-above}.

\subsubsection{Normalization}
\label{sec:normalization-1}

\ref{item:G-normalization} follows from Lemma \ref{lemma:unit-not-eternal} and \S\ref{sec:integer-valuedness}.

\subsubsection{Non-triviality}
\label{sec:non-triviality}

Every element of the colimit, including $x^{-d}1$,
lies in the image of
$\mathrm{HF}_{\mathrm{eq}}(\varphi_{t})\to
\mathrm{SH}_{\mathrm{eq}}(W)$ for some
zero-simplex $\varphi_{t}$. Thus
\ref{item:G-non-triviality} holds, i.e., $\mu$
attains arbitrarily large values.

\subsubsection{Discriminant}
\label{sec:discriminant}

Property \ref{item:G-discriminant} follows from
Theorem \ref{theorem:zig-zag}. Moreover, we have
the refinement from Theorem \ref{theorem:main-G}
where the discriminant is replaced by the
$W$-contractible discriminant (Definition
\ref{definition:W-contractible}), since the
statement of Theorem \ref{theorem:zig-zag}
includes this refinement.

\subsubsection{Agreement with the Conley-Zehnder
  index}
\label{sec:agre-with-conl}

In this section we prove Proposition
\ref{prop:agrees-with-CZ} that
$\mu(\varphi_{t})=\mathrm{CZ}(\varphi_{t})-n$
whenever $\varphi_{t}$ is a linear symplectic
isotopy whose time-1 map does not have $1$ as an
eigenvalue. In this setting, we focus only on the
case $G=\mathbf{k}=\Z/2\Z$.

The result follows easily from the fact that
$\varphi_{t}$ admits a canonical extension to
$\C^{n}$ as equivariant Borel data with a single
generator $\gamma(\varphi_{t})$ (the constant
orbit located at the origin). The underlying chain
complex is:
\begin{equation*}
  \mathbf{k}[[x]]\gamma(\varphi_{t}),
\end{equation*}
and all differentials vanish.

Now consider $s$ large enough that the
$\varphi_{t}$ admits a continuation to $R_{st}$,
and let $\mathfrak{c}$ be the continuation
map. Similarly let $\mathfrak{c}$ denote the
continuation map from $R_{\epsilon t}$ to $R_{st}$
for $\epsilon\in (0,1)$. It follows that:
\begin{itemize}
\item
  $\mathfrak{c}(\gamma(\varphi_{t}))=x^{k(\varphi_{t})}\gamma(R_{st}),$
\item
  $\mathfrak{c}(\gamma(R_{\epsilon
    t}))=x^{k(R_{\epsilon t})}\gamma(R_{st})=1$.
\end{itemize}
Note that we exclude the case that
$\mathfrak{c}=0$ by the structural theorems we
have proved above (e.g., otherwise $\mu$ would not
be a finite integer).

The integers $k(\varphi_{t})$ and
$k(R_{\epsilon t})$ can be computed using the
Fredholm index formula for Cauchy-Riemann
operators, and we have:
\begin{itemize}
\item
  $\mathrm{CZ}(\varphi_{t})-\mathrm{CZ}(R_{st})+k(\varphi_{t})=0,$
\item
  $\mathrm{CZ}(R_{\epsilon
    t})-\mathrm{CZ}(R_{st})+k(R_{\epsilon t})=0.$
\end{itemize}
Subtracting, we conclude:
\begin{equation*}
  \mathrm{CZ}(\varphi_{t})-n=k(R_{\epsilon t})-k(\varphi_{t}).
\end{equation*}
Because
$x^{k(\varphi_{t})}\gamma(R_{st})=x^{k(\varphi_{t})-k(R_{\epsilon
    t})}1$, we conclude
$\mu(\varphi_{t})=k(R_{\epsilon
  t})-k(\varphi_{t})$, and the desired result
follows.\hfill$\square$

\subsection{Acknowledgements}
\label{sec:acknowledgements}

First and foremost, the authors wish to
acknowledge and thank Joseph Helfer for a crucial suggestion leading to our Definition
\ref{definition:simplices}. Additionally, the
authors wish to thank Igor Uljarevi\'c for
invaluable discussions about contact Floer
homology. The first named author wishes to
thank Laurent Côté, and Julio Sampietro
Christ for enlightening discussions
surrounding equivariant cohomology. The third
named author wishes to thank Fabio Gironella and Zhengyi Zhou for
interesting discussions surrounding the paper
\cite{gironella-zhou-arXiv-2021}. The authors
also thank the organizers of the Simons
Center program: \emph{Contact geometry,
  general relativity, and thermodynamics}.

The first named author was partially supported by
funding from the Fondation Courtois, the ANR grant
CoSy, and the Université de Montréal DMS.

The second named author was partially supported by
an NSF Mathematical Sciences Postdoctoral Research
Fellowship, Award No. DMS-2503414.

The third named author is partially supported by
National Key R\&D Program of China
No. 2023YFA1010500, NSFC No. 12301081, and NSFC
No. 12361141812.

\bibliographystyle{amsalpha-doi}
\bibliography{citations}
\end{document}